\documentclass[reqno]{amsart}

\usepackage[all]{xy}
\usepackage{graphicx}

\usepackage{amssymb}
\usepackage{amsmath}
\usepackage{mathrsfs}
\usepackage{epsfig}
\usepackage{amscd}
\usepackage{graphicx, color}

\def\E{\ifmmode{\mathbb E}\else{$\mathbb E$}\fi} %natural numbers
\def\N{\ifmmode{\mathbb N}\else{$\mathbb N$}\fi} %natural numbers%
\def\R{\ifmmode{\mathbb R}\else{$\mathbb R$}\fi} %real numbers
\def\Q{\ifmmode{\mathbb Q}\else{$\mathbb Q$}\fi} %rational numbers
\def\C{\ifmmode{\mathbb C}\else{$\mathbb C$}\fi} %complex numbers
\def\H{\ifmmode{\mathbb H}\else{$\mathbb H$}\fi} %complex numbers
\def\Z{\ifmmode{\mathbb Z}\else{$\mathbb Z$}\fi} %integers
\def\P{\ifmmode{\mathbb P}\else{$\mathbb P$}\fi} %real numbers
\def\T{\ifmmode{\mathbb T}\else{$\mathbb T$}\fi} %real numbers
\def\SS{\ifmmode{\mathbb S}\else{$\mathbb S$}\fi} %real numbers
\def\DD{\ifmmode{\mathbb D}\else{$\mathbb D$}\fi} %real numbers

\newcommand{\e}{\varepsilon}

\newcommand{\del}{\partial}
\newcommand{\Cont}{{\operatorname{Cont}}}

\newcommand{\ben}{\begin{enumerate}}
\newcommand{\een}{\end{enumerate}}
\newcommand{\be}{\begin{equation}}
\newcommand{\ee}{\end{equation}}
\newcommand{\bea}{\begin{eqnarray}}
\newcommand{\eea}{\end{eqnarray}}
\newcommand{\beastar}{\begin{eqnarray*}}
\newcommand{\eeastar}{\end{eqnarray*}}
\newcommand{\bc}{\begin{center}}
\newcommand{\ec}{\end{center}}

\theoremstyle{theorem}
\newtheorem{thm}{Theorem}[section]
\newtheorem{cor}[thm]{Corollary}
\newtheorem{lem}[thm]{Lemma}
\newtheorem{prop}[thm]{Proposition}

\theoremstyle{definition}
\newtheorem{defn}[thm]{Definition}
\newtheorem{rem}[thm]{Remark}

\newtheorem{hypo}[thm]{Hypothesis}
\newtheorem{choice}[thm]{Choice}
\newtheorem{cond}[thm]{Condition}
\newtheorem{situ}[thm]{Situation}

\newtheorem*{thm*}{Theorem}

\numberwithin{equation}{section}

\def\R{{\mathbb R}}
\def\osc{{\hbox{\rm osc }}}
\def\Crit{{\hbox{Crit}}}

\def\E{{\mathbb E}}
\def\Z{{\mathbb Z}}
\def\C{{\mathbb C}}
\def\R{{\mathbb R}}
\def\P{{\mathbb P}}

\def\N{{\mathbb N}}

\def\11{{\mathbb I}}

\def\delbar{{\overline \partial}}

\def\rank{{\text{\rm rank}}}
\def\id{{\text{\rm id}}}

\def\C{\mathbb{C}}
\def\Z{\mathbb{Z}}

\def\T{\mathbb{T}}

\def\Q{\mathbb{Q}}

\def\E{\ifmmode{\mathbb E}\else{$\mathbb E$}\fi} %natural numbers
\def\N{\ifmmode{\mathbb N}\else{$\mathbb N$}\fi} %natural numbers
\def\R{\ifmmode{\mathbb R}\else{$\mathbb R$}\fi} %real numbers
\def\Q{\ifmmode{\mathbb Q}\else{$\mathbb Q$}\fi} %rational numbers
\def\C{\ifmmode{\mathbb C}\else{$\mathbb C$}\fi} %complex numbers
\def\H{\ifmmode{\mathbb H}\else{$\mathbb H$}\fi} %complex numbers
\def\Z{\ifmmode{\mathbb Z}\else{$\mathbb Z$}\fi} %integers
\def\P{\ifmmode{\mathbb P}\else{$\mathbb P$}\fi} %real numbers
\def\SS{\ifmmode{\mathbb S}\else{$\mathbb S$}\fi} %real numbers
\def\DD{\ifmmode{\mathbb D}\else{$\mathbb D$}\fi} %real numbers

\def\K{{\mathbb K}}
\def\R{{\mathbb R}}
\def\osc{{\hbox{\rm osc}}}
\def\Crit{{\hbox{Crit}}}
\def\E{{\mathbb E}}
\def\Z{{\mathbb Z}}
\def\C{{\mathbb C}}
\def\R{{\mathbb R}}

\def\N{{\mathbb N}}
\def\LL{{\mathcal L}}
\def\CL{{\mathcal L}}
\def\MM{{\mathcal M}}

\def\JJ{{\mathcal J}}

\def\FF{{\mathcal F}}

\def\ev{{\text{\rm ev}}}

\def\delbar{{\overline \partial}}

\def\e{\varepsilon}

\def\CA{{\mathcal A}}

\def\CC{{\mathcal C}}

\def\CF{{\mathcal F}}

\def\CH{{\mathcal H}}

\def\CJ{{\mathcal J}}

\def\CL{{\mathcal L}}
\def\CM{{\mathcal M}}

\def\CP{{\mathcal P}}

\def\CP{{\mathcal P}}

\def\CW{{\mathcal W}}

\def\supp{\operatorname{supp}}

\def\Dev{\operatorname{Dev}}

\def\coker{\operatorname{Coker}}
\def\span{\operatorname{span}}

\def\Image{\operatorname{Image}}

\def\uapp{u_{\text{\rm app}}}

\def\Glue{\operatorname{Glue}}

\def\PreG{\operatorname{PreG}}

\begin{document}
\quad \vskip1.375truein

\title[Gluing theory for contact instantons]
{Gluing theories of contact instantons and of
 pseudoholomoprhic curves in SFT}
\author{Yong-Geun Oh}
\address{Center for Geometry and Physics, Institute for Basic Science (IBS),
77 Cheongam-ro, Nam-gu, Pohang-si, Gyeongsangbuk-do, Korea 790-784
\& POSTECH, Gyeongsangbuk-do, Korea}
\email{yongoh1@postech.ac.kr}

\date{}

\begin{abstract} We develop the gluing theory of contact instantons in the
context of open strings and in the context of closed strings \emph{with vanishing charge},
for example in the symplectization context. This
is one of the key ingredients for the study of (virtually) smooth moduli space of
(bordered) contact instantons needed for the construction of contact instanton Floer
cohomology and more generally for the construction of Fukaya-type category of
Legendrian submanifolds in contact manifold $(M,\xi)$. As an application, we
apply the gluing theorem to give the construction of
(cylindrical) Legendrian contact instanton homology that enters in our solution
\cite{oh:entanglement1} to Sandon's question  for the nondegenerate case. We also
apply this gluing theory to that of moduli spaces of holomorphic buildings
arising in Symplectic Field Theory (SFT) by \emph{canonically} lifting
the former to that of the latter.
\end{abstract}



\thanks{This work is supported by the IBS project \# IBS-R003-D1}

\maketitle

\tableofcontents

\section{Introduction}
\label{sec:intro}

A contact manifold $(M,\xi)$ is a $2n-1$ dimensional manifold
equipped with a completely non-integrable distribution $\xi$ of rank $2n-2$,
called a contact structure. A Legendrian submanifold $R$ of $(M,\xi)$
is a submanifold of $M$ that is tangent to $\xi$ and of maximal dimension.
We mention that both notions of contact structure and Legendrian submanifolds
are scale-invariant.

When a contact form $\lambda$ with , i.e., $\ker \lambda = \xi$
is given, the complete non-integrability of $\xi$
can be expressed by the non-vanishing property
$$
\lambda \wedge (d\lambda)^{n-1} \neq 0.
$$
Furthermore $(\xi_x, d\lambda|_x)$ becomes a symplectic vector space at each $x \in M$, and
the Legendrian property of $R$ is then equivalent to
the Lagrangian property of $T_xR \subset (\xi_x,d\lambda|_x)$ for all $x \in R$.
We denote by
$$
\CC(\xi) = \CC(M,\xi)
$$
the set of contact forms of $\xi$.

The Reeb vector field $R_\lambda$ associated to the contact
form $\lambda$ is the unique vector field $X$ satisfying
\be\label{eq:Liouville} X \rfloor \lambda = 1, \quad X \rfloor
d\lambda = 0.
\ee
Therefore the tangent bundle $TM$ has the splitting $TM = \xi \oplus
\span\{R_\lambda\}$. We denote by
$$
\pi_\lambda: TM \to \xi
$$
the corresponding projection, and by $\Pi = \Pi_\lambda: TM \to TM$ the associated
idempotent.

\begin{defn}[CR-almost complex structure]\label{defn:adapted-J}
We call an endomorphism $J: TM \to TM$  an \emph{$\lambda$-adapted CR almost complex structure} if
it satisfies $J(\xi) \subset \xi$  and
$$
J(R_\lambda) = 0\, \quad J^2 = - \Pi _\lambda (= -id|_\xi \oplus 0).
$$
\end{defn}
In particular $J(TM) \subset \xi$. Following \cite{oh-wang:connection,oh-wang:CR-map1}, we call
a triple $(M, \lambda, J)$ a contact triad for the contact manifold $(M, \xi)$.
We call the  metric given by
$$
g=g_\xi+\lambda\otimes\lambda
$$
the \emph{triad metric} associated to the contact triad $(M,\lambda, J)$.

The purposes of the present paper is three-fold. One is to provide one of the key
ingredients, the gluing theory of the moduli spaces of contact instantons, especially
their relative counterparts, in the study of the moduli spaces of contact instantons with Legendrian boundary condition.
As usual in the symplectic Floer theory, the relevant gluing theorem is used in establishing various
algebraic relations that arise from a family of moduli spaces of contact instantons defined over punctured
Riemann surfaces. This gluing theorem, combined with the compactness result established in \cite{oh:contacton,oh-savelyev} and
in \cite{oh:entanglement1},
is  applied to our constructions of contact instanton Floer cohomology of Legendrian
submanifolds in \cite{oh-yso:J1B} and of a Fukaya-type
category of Legendrian submanifolds \cite{oh:entanglement2}. This is our second purpose.
The third purpose is
to apply this gluing theory to that of the moduli space of holomorphic buildings
appearing in Symplectic Field Theory (SFT) proposed by Eliashberg-Givental-Hofer
\cite{EGH}.

\subsection{Contact instanton equation and its analytic characteristics}

Analytic characteristics of the contact instanton equation are different from those of
well-known pseudoholomorphic curves in symplectic geometry, although it also
shares some. (Compare the form of the a priori coercive elliptic estimates established in \cite{oh-wang:CR-map1,oh-wang:CR-map2,oh:contacton-Legendrian-bdy} with that
of e.g., \cite{HWZ-asymptotics,bourgeois}.)
We will show that analytic characteristics of the linearized operator of
contact instanton equation are also different from those of pseudoholomorphic curves. Because of these
differences, we need to develop a new gluing theory that fits the contact instantons. Here we briefly
describe what kind of geometro-analytic equation we look at for
its gluing theory, and explain the differences between the analytic characteristic
of contact instantons and that of pseudoholomorphic curves
by examining the linearization of contact instantons.

Let
$$
\vec R = (R_1, \cdots, R_k)
$$
be an arbitrary tuple of pairwise disjoint connected Legendrian submanifolds. We call it a \emph{Legendrian link}
in general. We assume that the link is \emph{transversal} in the sense
that all Reeb chords for the tuples are nondegenerate. Such a transversality can be achieved
by a generic tuple $\vec R$. (See \cite[Appendix B]{oh:contacton-transversality}.)
We then consider the associated boundary value problem for a map $w: \dot \Sigma \to M$
satisfying the equation,
\be\label{eq:contacton-Legendrian-bdy-intro}
\begin{cases}
\delbar^\pi w = 0, \, \quad d(w^*\lambda \circ j) = 0\\
w(\overline{z_iz_{i+1}}) \subset R_i, \quad i = 1, \ldots, k
\end{cases}
\ee
which we call the contact instanton equation,
with the given asymptotics $\vec \gamma = \{\gamma_i\}_{i=1}^k$
on a boundary punctured Riemann surface $\dot \Sigma$

We consider the maps defining two equations in \eqref{eq:contacton-Legendrian-bdy-intro}
\be\label{eq:upsilon1}
\Upsilon_1: \CF \to \Omega^{(0,1)}(w^*\xi); \quad \Upsilon_1(w) = \delbar^\pi w
\ee
\be\label{eq:upsilon2}
\Upsilon_2: \CF \to \Omega^2(\dot \Sigma); \quad \Upsilon_2(w) = d(w^*\lambda \circ j)
\ee
and the $\Upsilon(w): = (\Upsilon_1,\Upsilon_2)$, where $\CF
= C^\infty((\dot \Sigma, \del \dot \Sigma),(M,\vec R))$ is the function space
$$
C^\infty((\dot \Sigma, \del \dot \Sigma),(M,\vec R)): = \{ w \in C^\infty(\dot \Sigma, M) \mid
w(\overline{z_iz_{i+1}}) \subset R_i, \, i = 1, \ldots, k\}.
$$

\emph{At the moment, we make our discussion in the formal level} working locally on
the domain $\dot \Sigma$ without being specific about the asymptotic conditions at the
punctures.
We have the tangent space
$$
T_w \CF = \{ Y \in \Gamma(w^*TM) \mid Y(\overline{z_iz_{i+1}}) \subset (\del w)^*TR_i\}
=: \Omega^0(w^*TM,(\del w)^*T\vec R).
$$
We compute the linearization which defines a linear map
$$
D\Upsilon(w): \Omega^0(w^*TM,(\del w)^*T\vec R) \to \Omega^{(0,1)}(w^*\xi) \oplus \Omega^2(\Sigma).
$$
We note
\beastar
\operatorname{rank}\Lambda^0(w^*TM) & = & 2n+1 \\
\operatorname{rank} \Lambda^{(0,1)}(w^*\xi) \oplus \Lambda^2(\Sigma) & = & 2n+1:
\eeastar
Recall that the equality of the two dimensions is a necessary condition for the linear operator
$D\Upsilon(w)$ to be an elliptic operator.

Using the contact triad connection $\nabla$ of $(M,\lambda,J)$ and the contact
Hermitian connection $\nabla^\pi$ for $(\xi,J)$ introduced in
\cite{oh-wang:connection,oh-wang:CR-map1}, the following linearization formula
is derived in \cite{oh:contacton} for the closed string case.
(See Appendix \ref{sec:connection} for a summary of
basic properties of the contact triad connection.)

We denote by $T = T_\nabla$ the torsion tensor of the connection $\nabla$.

\begin{thm}[Theorem 10.1 \cite{oh:contacton}]\label{thm:linearization-intro} In terms of the decomposition
$d\pi = d^\pi w + w^*\lambda\, R_\lambda$
and $Y = Y^\pi + \lambda(Y) R_\lambda$, we have
\bea
D\Upsilon_1(w)(Y) & = & \delbar^{\nabla^\pi}Y^\pi + B^{(0,1)}(Y^\pi) +  T^{\pi,(0,1)}_{dw}(Y^\pi)\nonumber\\
&{}& \quad + \frac{1}{2}\lambda(Y) (\CL_{R_\lambda}J)J(\del^\pi w)
\label{eq:DUpsilon1}\\
D\Upsilon_2(w)(Y) & = &  - \Delta (\lambda(Y))\, dA + d((Y^\pi \rfloor d\lambda) \circ j)
\label{eq:DUpsilon2}
\eea
where $B^{(0,1)}$ and $T_{dw}^{\pi,(0,1)}$ are the $(0,1)$-components of $B$ and
$T_{dw}^\pi$ respectively, where $B, \, T_{dw}^\pi: \Omega^0(w^*TM) \to \Omega^1(w^*\xi)$ are
 zero-order differential operators given by
$$
B(Y) := - \frac{1}{2}  w^*\lambda \left((\CL_{R_\lambda}J)J Y\right)
$$
and
$$
T_{dw}^\pi(Y) = \pi T(Y,dw)
$$
respectively.
\end{thm}
From the expression of the linearization map of
$\Upsilon(w) = (\Upsilon_1,\Upsilon_2)$, the map is an elliptic operator whose symbol is \emph{of mixed degree} given by the matrix
$$
\left(\begin{matrix} \frac{\eta + i \eta \circ j}2 Id & *_{12} \\ *_{21} & |\eta|^2 \end{matrix} \right)
\in \operatorname{Hom}(\pi^*E_1, \pi^*E_2)\Big|_\eta
$$
where $\eta \in w^*TM \setminus 0_{w^*TM}$ and $*_{12}$ has its homogeneous degree 0
and $*_{21}$ degree 1 in $\eta$.
(See \eqref{eq:matrixDUpsilon} for the matrix form of $D\Upsilon(w)$.) Here
$\pi: T^*\Sigma \setminus \{0_{T^*\dot \Sigma}\} \to \Sigma$
is the natural projection and $\eta \in T^*\Sigma \setminus \{0_{T^*\Sigma}\}$ and
$E_1, \, E_2$ are vector bundles on $\dot \Sigma$ given by
$$
E_1 = w^*TM = w^*\xi \oplus \R \langle R_\lambda(w) \rangle, \quad E_2 = \Lambda^{(0,1)}(w^*\xi)) \oplus \Lambda^2(T^*\Sigma)
$$
Therefore for any choice of
$k \geq 2, \, p > 2$, its completion becomes a Fredholm operator, when the domain
$\Sigma$ is a compact Riemann surface without boundary.
(See \cite[Appendix C]{oh:contacton} for the details. See also
 \cite{lockhart-mcowen} for detailed discussion of mixed type elliptic operators.)

In a more recent article \cite{oh:contacton-Legendrian-bdy}, the present author
studied the boundary valued problem for the  contact instanton equation and
established that the problem under the Legendrian boundary condition is a nonlinear elliptic boundary
value problem by deriving relevant a priori coercive estimates and proving the asymptotic
convergence results to Reeb chords on each boundary puncture of
a punctured Riemann surface $\dot \Sigma$:
We consider a given collection of Reeb chords as the asymptotic boundary conditions at the given punctures of
a punctured Riemann surface $\dot \Sigma$. Depending on the orientation of the strip-like coordinates at the punctures,
we decompose
$\vec \gamma = \overline \gamma \sqcup \underline \gamma$ into two subcollections
$$
\underline \gamma = \{\gamma_i^+\}_{i=1, \ldots, s^+}, \quad \overline \gamma = \{\gamma_j^-\}_{j=1, \ldots, s^-}
$$
with $k = s^- + s^+$. Then the linearization $D\Upsilon(w)$ associated to \eqref{eq:contacton-Legendrian-bdy-intro}
is an elliptic differential operator of mixed degree, and hence it extends to a bounded linear map
\be\label{eq:dUpsilon-intro}
D\Upsilon_{(\lambda,T)}(w): \Omega^0_{k,p}(w^*TM,(\del w)^*T\vec R;J;\overline \gamma,\underline \gamma) \to
\Omega^{(0,1)}_{k-1,p}(w^*\xi) \oplus \Omega^2_{k-2,p}(\Sigma)
\ee
which is a Fredholm operator. We refer readers to Section \ref{subsec:fredholm} for a detailed discussion.

\subsection{Gluing of contact instantons and its lifting to SFT}

Here we offer a bit more perspective and description of the lifting process of our gluing theory of contact instantons
to those of pseudoholomorphic curves on the symplectization and on symplectic buildings.

When the charge class $[w^*\lambda] = 0$ in $H^1(\dot \Sigma, \R)$ (see \cite{oh-savelyev} for the
precise definition), we can twist the equation
 \eqref{eq:contacton-Legendrian-bdy-intro} with the choice of
the primitive $f: \dot \Sigma \to \R$ satisfying $w^*\lambda \circ j = df$ and consider the combined equation
 for $(w,f)$
\be\label{eq:contacton-exact}
\begin{cases}
\delbar^\pi w = 0, \quad w^*\lambda \circ j = df,\\
w(\overline{z_iz_{i+1}}) \subset R_i, \quad i = 0, \ldots, k.
\end{cases}
\ee
This equation is indeed closely related to Hofer's pseudoholomorphic curve equation $u:\Sigma \to M \times \R$ on
the symplectization \cite{hofer:invent}, \emph{provided we ask $f$ to be of the form $f: = s \circ u$, i.e., if
we write $u = ( \pi_1 \circ u, \pi_2 \circ u) =: (w,f)$.}

A symplectization of $(M,\lambda)$ is the product $M \times \R$ with the symplectic form
$$
d(e^s \pi_1^*\lambda), \quad s := \pi_2, \,
$$
where $\pi_i$ with $i=1,\, 2$ are the projections of $M \times \R$ to the first and the second factors
respectively. Then following \cite{hofer:invent}, we equip $M \times \R$
with an almost complex structure $\widetilde J$ which is the one lifted from a
CR almost complex structure $J$ that satisfies
\be\label{eq:adapted-tildeJ}
\widetilde J|_{\xi} = J_\xi, \quad \widetilde J(R_\lambda) = - \frac{\del}{\del s}.
\ee
With respect to such $\widetilde J$, a $\widetilde J$-holomorphic curve $u$
on $(M \times \R, \widetilde J)$ is characterized by the equation
\be\label{eq:hofer-exact}
\delbar^\pi w = 0, \quad w^*\lambda \circ j = df.
\ee
This equation is precisely of the form of \eqref{eq:contacton-exact}
if we set $w = \pi_1 \circ u$ and $f = \pi_2 \circ u$ on $\dot \Sigma$.

We reduce the gluing theory of SFT, say a 2-story curve
$(u_1,u_2)$ in the 2-story cobordism $(W_1,W_2)$,
to that of contact instantons $w_1^-$ and $w_2^+$ where
$$
u_i|_{\del W_i} = w_i^- \cup w_i^+, \quad i =1, \, 2
$$
in two steps:
\begin{enumerate}
\item First lift the latter canonically to the \emph{asymptotic gluing theory}
of $u_1$ and $u_2$ near the ceiling of $W_1$ (= the floor of  $W_2$).
\item Then glue the asymptotically-glued piece to the finite parts of $u_1$ and
$u_2$.
\end{enumerate}
The former operation is purely canonical in that it corresponds just to
solving the equation for $f$
$$
df = w^*\lambda \circ j
$$
when $w$ is given, and the second operation is by now well-established
Taubes' gluing technique that is exercised ubiquitously in the literature of
Gromov-Witten-Floer theory. See \cite{ruan-tian}, \cite{mcduff-salamon-symplectic},
\cite{EES-jdg}, \cite{fooo:book2} , \cite{fooo:exponential}, \cite{oh-zhu:tropical}
to name a few. There are also many literature on the gluing construction on
the symplectization in various special circumstances. See
\cite{hutchings-taubes-gluing1,hutchings-taubes-gluing2}, \cite{CELN-cord},
\cite{bao-honda-kuranishi} to name a few.

In this regard, our gluing theory of contact instantons \emph{for the exact case}
\eqref{eq:hofer-exact} immediately solves the gluing problem in the symplectization.
In particular, our gluing theorem
provides an extension and simplification of the gluing results by Ekholm-Etnyre-Sullivan
from \cite{EES-jdg,EES-tams}, which concerns the \emph{symplectization} of
the case of $\R^{2n+1}$ and $P \times \R$ with symplectic $P$ respectively:
While Ekholm-Etnyre-Sullivan crucially uses the presence of the projection map
$\pi: M \to P$ to a symplectic manifold $P$ in their gluing construction, we handle arbitrary
tame contact manifold $(M,\xi)$.

It also solves the relevant gluing problem in SFT by reducing the full gluing problem
on the building to a soft topological problem of determining the radial
components $f$ of pseudoholomorphic curves $u = (w,f)$ on the
neck region of the glued building, and then applying the standard gluing theory
\emph{without involving scale-dependent gluing process} in the Gromov-Floer-Witten theory.
We refer readers to Part IV to see the simplicity of our discussion on the gluing
theory  for multi-story pseudoholomorphic curves in general symplectic buildings, \emph{with
the gluing theory of contact instantons under our disposal.}

We believe that our gluing theory will be particulary useful for the construction of
Kuranishi structures in SFT in that the abstract scheme laid out in the book \cite{fooo:book-kuranishi}
could be applicable whose study will be left elsewhere. Full construction of Kuranishi structures in SFT
verifying the axioms of SFT laid out in \cite{EGH} is announced by
Ishikawa \cite{ishikawa}. See also \cite{pardon-contact} for a relevant work on contact homology.

\begin{rem} We recall readers that one big difference of the gluing  of pseudoholomorphic curves
on a symplectic building from the standard gluing  in the Gromov-Witten-Floer theory on
a smooth symplectic manifold is that the irreducible components appearing in the
multi-story curves have different scales. Because of this, gluing two components in different stories
requires rescaling the targets, or stretching the neck, with a careful choice of scales for different irreducible components
similarly as used in relative Gromov-Witten theory. (See \cite{li-ruan}, \cite{fooo:chap10}, \cite{oh-zhu:scaled}
for the literature using such a scale-dependent gluing results via the target rescaling.)
Unfortunately, as far as we are aware, there is no reference of this scaled-gluing
scheme that can be easily used in the general framework of symplectic buildings.
An upshot of our treatment of the gluing theory in SFT is that \emph{once we have carried out the
gluing of contact instantons, which does not involve target rescaling,} the gluing result can be canonically lifted to the SFT setting
without involving any hard analysis \emph{on the symplectization} but only involves some soft topological step of solving the
equation $df = w^*\lambda \circ j$ for the primitive $f$ followed by the standard gluing theory in
the ordinary Gromov-Witten-Floer theory.
\end{rem}

\subsection{Construction of (cylindrical) contact instanton Legendrian Floer cohomology}

In \cite{oh:entanglement1}, we considered a two-component compact Legendrian link of the type
$$
(\psi(R),R)
$$
in any \emph{tame} contact manifolds $M$  for
contactomorphism $\psi$ such that
$$
\psi(R) \pitchfork Z_R \quad \& \quad \psi(R) \cap R = \emptyset:
$$
Here $Z_R$ is the union, called the Reeb trace of $R$ in \cite{oh:entanglement1},
$$
Z_R = \bigcup_{t \in \R} \phi_{R_\lambda}^t(R).
$$
(The notion of tame contact manifolds was also introduced therein.)
In the course of answering Sandon-Shelukhin's conjecture for \emph{nondegenerate}
$\psi$,  we considered the $\Z_2$-vector space
$$
CI_\lambda^*(\psi(R), R) : = \Z_2\langle\frak{Reeb}(\psi(R),R)\rangle
$$
and define a $\Z_2$-linear map
$$
\delta= \delta_{(\psi(R),R;H)} : CI_\lambda^*(\psi(R),R) \to  CI_\lambda^*(\psi(R),R)
$$
by counting its matrix element
$$
\langle \delta \gamma^-, \gamma_+ \rangle : = \#_{\Z_2}(\CM(\gamma_-,\gamma_+)).
$$
Here $CI_\lambda$ stands for \emph{contact instanton for $\lambda$} as well as the letter $C$ also
stands for \emph{complex} at the same time, and
$\CM(\gamma_-,\gamma_+)$ is the moduli space
$$
\CM(\gamma_-,\gamma_+) = \CM(M, J';R;\gamma_-,\gamma_+)
$$
of contact instanton Floer trajectories $u$ satisfying
$u(\pm\infty) = \gamma_\pm$ for a suitably chosen time-dependent family
$J'= \{J'_t\}_{0 \leq t \leq 1}$
as done in \cite[Section 6.2]{oh:entanglement1}.

\begin{choice}[CR almost complex structures] \label{choiceLJt}
Let $J_0 \in \CJ(\lambda)$. For given contact Hamiltonian $H = H(t,x)$,
we fix a time-dependent CR almost complex structure given by
\be\label{eq:Jt}
J' = \{J_t'\}_{0 \leq t \leq 1}, \quad J_t':= (\psi_H^t(\psi_H^1)^{-1})_*J_0.
\ee
of $\lambda$-adapted CR almost complex structures.
\end{choice}
\emph{We alert readers that for a general pair $(R_0,R_1)$ the above counting does not lead to the operator
$\delta_{(\psi(R),R;H)}$ satisfying $\delta_{(\psi(R),R;H)} \circ \delta_{(\psi(R),R;H)}  = 0$
when the pair $(M,R)$ or $(M,\psi(R))$ carry (nonconstant) bubbling similarly as in
symplectic geometry \cite{oh:cpam-obstruction}.} This will lead us to deform the map $\delta_{(\psi(R),R;H)}$
to
$$
\delta^{(\psi_*{\frak b},\frak b)} = \delta_{(\psi(R),R;H)}^{(\psi_*(\frak b),\frak b)}
$$
by the $\frak m_0$-term in \cite{oh:entanglement2} as in \cite{fooo:book1} by considering the one-punctured disc moduli space
$$
\CM_1((\H,\del \H),(M,R);J), \quad \CM_1((\H,\del \H),(M,\psi(R);\psi_*J).
$$
We recall that the dual version of this term corresponds to
the \emph{augmentation} in the literature of Legendrian contact homology.
(See \cite{chekanov:dga} and many other literature afterwards in contact topology.)
A full study of this construction and its DGA enhancement will be given in \cite{oh:entanglement2}.

In the mean time under the hypothesis
\be\label{eq:|H|-small}
\|H\| < T(M,\lambda)
\ee
considered in \cite{oh:entanglement1}, we can
 single out the `short' Reeb chords between $\psi(R)$ and $R$ that
are continued from the intersection
$$
R \cap Z_R = R
$$
where $Z_R$ is the Reeb trace defined by
$$
Z_R = \bigcup_{t \in \R} \phi_{R_\lambda}^t(R).
$$
(See \cite{chekanov:dmj}, \cite{oh:mrl} for a similar circumstance in the Lagrangian
Floer theory.)
Then we consider the moduli space of maps $w: \dot \Sigma \to M$ satisfying the equation
\be\label{eq:contacton-Legendrian-bdy}
\begin{cases}
\delbar^\pi w = 0 \, \quad d(w^*\lambda \circ j) = 0\\
w(\tau,0) \in \psi(R), \quad w(\tau,1) \in R
\end{cases}
\ee
when $\dot \Sigma = \R \times [0,1]$, i.e., for the case of cylindrical context.
The inequality \eqref{eq:|H|-small} and some energy bound
given in \cite[Theorem 1.24]{oh:entanglement1} also prevent the moduli spaces
from bubbling off.

One of the essential ingredients in the construction of Floer-type invariants involving
a family of moduli spaces is the gluing theorem that ensures some compatibility conditions for
the system of moduli spaces leading to a relevant algebraic structures, such as the
Floer homology or more generally Fukaya category, especially in the chain level.
In the present paper, we utilize the gluing result for the moduli spaces of \emph{cylindrical}
contact instantons and their relative versions with Legendrian boundary condition and provide
the gluing details that enters in
the author's proof of Sandon-Shelukhin's conjecture for the
nondegenerate case. (See \cite[Theorem 10.6]{oh:entanglement1}.)

As in the standard Floer theory \cite{floer-intersections},
the proof of the following theorem will be finished once the relevant
gluing theorem for the moduli spaces of contact instantons is established.

\begin{thm}[Theorem 10.6, \cite{oh:entanglement1}]\label{thm:definition-intro}
Suppose $(M,\xi)$ is tame and $R \subset M$ is a compact Legendrian submanifold.
Let $\lambda$ be a tame contact form such that
\begin{itemize}
\item $\psi = \psi_H^1$ and $\|H\| < T_\lambda(M,R)$.
\item the pair $(\psi(R), R)$ is transversal in the sense that $\psi(R) \pitchfork Z_R$,
\end{itemize}
Let $J' = \{J'_t\}$ be as above. Then
\be\label{eq:d2=0}
\delta \circ \delta = 0
\ee
and hence we define its homology which we denote by
$$
HI_\lambda^*((\psi(R),R;J').
$$
Furthermore for two different choices of such $J'$ (i.e., of $J_0$) or of $H$, the complex
$(CI_\lambda^*(\psi(R),R), \delta)$ are chain-homotopic to each other.
\end{thm}

It is shown in \cite{oh:entanglement1} that $HI_\lambda^*((\psi(R),R;J')$ is isomorphic to
$H^*(R,\Z_2)$ in the given circumstance, and hence concludes Sandon-Shelukhin's conjecture
\cite{sandon-translated}, \cite{shelukhin-contactomorphism} for the
nondegenerate case.

To perform the aforementioned gluing theory, we first need to develop a relevant
Fredholm theory of the moduli spaces of contact instantons, especially for the
case of bordered contact instantons with Legendrian boundary.  This has been established
respectively for the closed string case in \cite{oh:contacton} and
for the bordered contact instantons  with Legendrian boundary conditions  in
\cite{oh:contacton-transversality} for the perturbations of $J$ or of Legendrian boundaries.

One point we would like to highlight in the study of generic nondegeneracy of Reeb chords is that
we  consider the chords \emph{in the sense of Moore paths} whose elements
 we represent by the pairs $(T,\gamma)$ such that
\be\label{eq:Moore-path}
T \geq 0, \quad \gamma:[0,T] \to M, \quad T = \int \gamma^*\lambda.
\ee
We emphasize here that we include the zero period $T= 0$ and consider the constant paths $(0,\gamma)$.
Such a path exists only when $\psi(R) \cap R \neq \emptyset$, i.e., $\psi = id$.

For the purpose of computation, we need to notice that the intersection $R \cap Z_R$
is Bott-Morse in a suitable sense.  In a way, this existence of many constant solutions is responsible
both for the existence and for the isomorphism property proved in \cite[Subsection 10.2]{oh:entanglement1}
which is nothing but the one stated in Theorem \ref{thm:definition-intro}.

\begin{rem} It is an interesting open problem to equip
Kuranishi structures on the compactified moduli spaces of contact instantons which are suitably
compatible so that they give rise to the (Legendrian) contact DGA that appears in the
case of trivial symplectic cobordism, i.e., the case of symplectization of contact manifolds.
(See \cite{bao-honda-kuranishi}, \cite{pardon-contact}, \cite{ishikawa}.) We believe that
the general abstract framework of the Kuranishi structure from \cite{fooo:book-kuranishi}
or some variation of that of \cite{pardon-contact}, \cite{bao-honda-kuranishi} applies to
the current case of contact instantons too. We hope to come back to this elsewhere.
\end{rem}

Some part of the present paper is written under the presumption that the readers are familiar with the
contents of \cite{oh-wang:CR-map1}, \cite{oh:contacton}, \cite{oh:contacton-Legendrian-bdy},
\cite{oh:contacton-transversality} and \cite{oh:entanglement1}
in addition to the fine details of the gluing theory of
pseudoholomorphic curves in symplectic geometry, more specifically in
the Gromov-Witten-Floer theory such as those presented in
\cite{floer-unregularized} and \cite{ruan-tian} (or \cite{mcduff-salamon-symplectic}),
for example. Therefore it will be useful for readers to consult them, whenever necessary,
while reading the present paper.
However the main constructions of approximate solutions and of
approximate right inverse and new elements entering in the gluing construction
of contact instantons are fully explained so that they can be read independently.

The research carried out in the present work was first presented in the online
seminar of IBS Center for Geometry and Physics and in the workshop
``Recent developments in Lagrangian Floer theory'' held in Simons Center for
Geometry and Physics in March 2022. We would like to thank
the organizers and participants of the workshop for the opportunity and especially
to Georgios Dimitroglou Rizell for some comment on
the invariance proof of Legendrian contact homology (see Remark \ref{rem:invariance}).
We also thank Tobias Ekholm for
the information on the literature involving gluing results of
the pseudoholomorphic curves on the symplectization in various special cases.

\bigskip

\noindent{\bf Convention and Notations:}

\medskip

\begin{itemize}
\item {(Contact Hamiltonian)} We define the contact Hamiltonian of a contact vector field $X$ to be
$$
- \lambda(X).
$$
\item
For given time-dependent function $H = H(t,x)$, we denote by $X_H$ the associated contact Hamiltonian vector field
whose associated Hamiltonian $- \lambda(X_t)$ is given by $H = H(t,x)$, and its flow by $\psi_H^t$.
\item When $\psi = \psi_H^1$, we say $H$ generates $\psi$ and write $H \mapsto \psi$.
\item {(Developing map)} $\Dev(t \mapsto \psi_t)$: denotes the time-dependent contact Hamiltonian generating
the contact Hamiltonian path $t \mapsto \psi_t$.
\item {(Reeb vector field)} We denote by $R_\lambda$ the Reeb vector field associated to $\lambda$
and its flow by $\phi_{R_\lambda}^t$.
\item {(Linearized contact instanton homology)} We denote by $CI_\lambda^*(R_0,R_1)$ the $\lambda$-linearized contact instanton complex and $HI_\lambda^*(R_0,R_1)$ its cohomology, when defined.
 \item $(\dot \Sigma,j)$: a punctured Riemann surface (with boundary) and $(\Sigma,j)$ the associated compact Riemann surface.
\item We always regard the tangent map $du$ as a $u^*TM$-valued one-form and write
$du = d^\pi u + u^*\lambda \otimes R_\lambda$ with respect to the decomposition $TM = \xi \oplus \R \langle R_\lambda \rangle$.
\item $T(M,\lambda)$: the infimum of the action of closed $\lambda$-orbits.
\item $T(M,\lambda;R)$: the infimum of the action of self $\lambda$-chords.
\item $T_\lambda(M,R) = \min\{ T(M,\lambda), T(M,\lambda;R)\}$.
\item $\vec R = (R_1, \cdots, R_k)$: a $k$-tuple of connected pairwise disjoint
Legendrian submanifolds, or called a Legendrian link with $k$ components.
\item $\Upsilon(w) = (\delbar^\pi w, d(w^*\lambda \circ j))$: the defining section of contact instantons.
\item $\CH^{(0,1)}(M,\lambda)|_w = \Omega^{(0,1)}(w^*\xi)
\bigoplus \Omega(\dot \Sigma) \otimes R_\lambda$: the formal domain of
of the covariant linearization $D\Upsilon(w)$ of $\Upsilon$ at $w$
\item $\CH^{(0,1)}(M,\lambda)$: the formal codomain of the map $\Upsilon$ defined by
$$
\CH^{(0,1)}(M,\lambda): = \bigcup_{w} \CH^{(0,1)}(M,\lambda)|_w.
$$
\item $\CH^{(0,1)}_{k-1,p}(M,\lambda)|_w$: the codomain
$$
\Omega^{(0,1)}_{k-1,p}(w^*\xi)
\bigoplus \Omega^2_{k-2,p} (\dot \Sigma) \otimes R_\lambda
$$
of the covariant linearization $D\Upsilon(w)$ of $\Upsilon$ at $w$ acting on its domain
$$
\Omega^0_{k,p}((\dot \Sigma, \del \dot \Sigma),(w^*TM,(\del w)^*T \vec R).
$$
\item {[{\bf Usage of constant $C$}]}: There are various constants denoted by $C> 0$ which vary
place by place, but which is a uniform constant independent of relevant parameters
that enter in its usage. The meaning of uniformity of the choice of this constant will be
clear when it is used.
\end{itemize}

\section{Review of basic analytical estimates of contact instantons}

Let a contact manifold $(M,\xi)$ be given.
In this section, we fix a contact triad
$$
(M,\lambda, J)
$$
where $J \in \CJ(\lambda)$ is a $\lambda$-adapted CR-almost complex structure.
All relevant estimate is in terms of the associated triad metric $g = g_{\lambda,J}$
given by
$$
g = d\lambda(\cdot, J\cdot) + \lambda\otimes \lambda.
$$
Recall a contact form $\lambda$ admits a decomposition $TM = \xi \oplus \R\langle R_\lambda \rangle$. We denote
the associated projection to $\xi$ by $\pi: TM \to \xi$ and decompose
$$
v = v^\pi + \lambda(v) \, R_\lambda, \quad v^\pi: = \pi(v).
$$

We consider the equation of (unperturbed) contact instantons
\be\label{eq:contacton-Legendrian-bdy}
\begin{cases}
\delbar^\pi w = 0, \quad d(w^*\lambda \circ j) = 0, \\
w(\overline{z_iz_{i+1}}) \subset R_i
\end{cases}
\ee
for a general Legendrian tuple $\vec R = (R_1, \ldots, R_k)$ on general contact manifold $(M,\lambda)$.

\subsection{Coercive elliptic estimates and subsequence convergence}
\label{subsec:coercive}

We collect some basic results on the analysis of the equation
established in \cite{oh-wang:CR-map1} or in \cite{oh:contacton-Legendrian-bdy}.

We start with the following local $W^{2,2}$-estimates.

\begin{thm}[Theorem 1.3 \cite{oh:contacton-Legendrian-bdy}]\label{thm:local-W22} Let $w: \R \times [0,1] \to M$ satisfy
\eqref{eq:contacton-Legendrian-bdy}.
Then for any relatively compact domains $D_1$ and $D_2$ in
$\dot\Sigma$ such that $\overline{D_1}\subset D_2$ with
$w(\del D_2) \subset R_0$ or $w(\del D_2) \subset R_1$,
 we have
$$
\|dw\|^2_{W^{1,2}(D_1)}\leq C_1 \|dw\|^2_{L^2(D_2)} + C_2 \|dw\|^4_{L^4(D_2)} + C_3 \|dw\|^3_{L^3(\del D_2)}
$$
where $C_1, \ C_2$ are some constants which
depend only on $D_1$, $D_2$ and $(M,\lambda, J)$ and $C_3$ is a
constant which also depends on $R_0$ or $R_1$.
\end{thm}

Once this $W^{2,2}$-estimate is established, we then proceed the following higher
regularity estimates.

\begin{thm}[Theorem 1.4 \cite{oh:contacton-Legendrian-bdy}]\label{thm:higher-regularity}
 Let $w$ be a contact instanton satisfying
\eqref{eq:contacton-Legendrian-bdy}.
Then for any pair of domains $D_1 \subset D_2 \subset \dot \Sigma$ such that $\overline{D_1}\subset D_2$
 with $w(\del D_2) \subset R_0$ or $w(\del D_2) \subset R_1$, we have
$$
\|dw\|_{C^{k,\alpha}(D_1)} \leq C \|dw\|_{W^{1,2}(D_2)}
$$
for some constant $C = C(k,\alpha) > 0$ depending on $J$, $\lambda$, $D_1, \, D_2$, $R_0$ or $R_1$,
and $(k,\alpha)$ but independent of $w$.
\end{thm}

Let $w: \R \times [0,1] \to M$ be any smooth map.
As in \cite{oh-wang:CR-map1}, we define the total $\pi$-harmonic energy $E^\pi(w)$
by
\be\label{eq:endenergy}
E^\pi(w) = E^\pi_{(\lambda,J)}(w) = \frac{1}{2} \int |d^\pi w|^2
\ee
where the norm is taken in terms of the given metric $h$ on $\dot \Sigma$ and the triad metric on $M$.

Next we study the asymptotic behavior of contact instantons $w$ satisfying the following
hypotheses.
\begin{hypo}\label{hypo:basic}
Assume $w:\R \times [0,1] \to M$ satisfies the contact instanton equation \eqref{eq:contacton-Legendrian-bdy}
and
\begin{enumerate}
\item $E^\pi_{(\lambda,J;\dot\Sigma,h)}(w)<\infty$ (finite $\pi$-energy);
\item $\|d w\|_{C^0(\dot\Sigma)} <\infty$.
\end{enumerate}
\end{hypo}

For any $w$ satisfying  Hypothesis \ref{hypo:basic}, we associate two
natural asymptotic invariants at $\tau \pm \infty$.
\begin{defn}\label{defn:T-Q} The \emph{asymptotic action} is defined to be
\be
T : =  \lim_{\tau \to \infty} \int_{\{\tau\}\times [0,1]}(w|_{\{0\}\times [0,1]})^*\lambda\label{eq:TM-T}
\ee
and the \emph{asymptotic charge} is by
\be
Q : = \lim_{\tau \to \infty} \int_{\{\tau\}\times [0,1]}(w|_{\{0\}\times [0,1] })^*\lambda\circ j.\label{eq:TM-Q}
\ee
provided they exist.
(Here we only look at the positive end. The case of negative end is similar.)
\end{defn}
The above finite $\pi$-energy and $C^0$ bound hypotheses imply
\be\label{eq:hypo-basic-pt}
\int_{[0, \infty)\times [0,1]}|d^\pi w|^2 \, d\tau \, dt <\infty, \quad \|d w\|_{C^0([0, \infty)\times [0,1])}<\infty.
\ee
\begin{rem}
In general there is no reason why these limits exist and even if the limits exist, they may also depend on the choice of
subsequences under Hypothesis \ref{hypo:basic}.
In the closed string case, \cite{oh-wang:CR-map1} shows that the asymptotic charge $Q$ may not vanish which
is the key obstacle to the compactification and the Fredholm theory of contact instantons for $Q \neq 0$.
\end{rem}

As in \cite{oh-wang:CR-map1}, \cite{oh:contacton-Legendrian-bdy}, we call $T$ the \emph{asymptotic contact action}
and $Q$ the \emph{asymptotic contact charge} of the contact instanton $w$ at the given puncture.

\begin{thm}[Vanishing asymptotic charge; Theorem 6.7
\cite{oh:contacton-Legendrian-bdy}]\label{thm:subsequence-convergence} Let $\vec R = (R_1, \ldots, R_k)$
be a $k$-component (ordered) link. Let $p \in \del \Sigma$ and $(\tau,t) \in [0,\infty) \times [0,1]$
be a strip-like coordinates
on a punctured neighborhood $U \setminus \{p\} \subset \dot \Sigma$.
Assume that $w:[0, \infty)\times [0,1]\to M$ satisfies the contact instanton equations
\eqref{eq:contacton-Legendrian-bdy}
and Hypothesis \ref{hypo:basic}.
Then for any sequence $s_k\to \infty$, there exists a subsequence, still denoted by $s_k$, and a
massless instanton $w_\infty(\tau,t)$ (i.e., $E^\pi(w_\infty) = 0$)
on the cylinder $\R \times [0,1]$ that satisfies the following:
\begin{enumerate}
\item On any given compact subset $K \subset [0,\infty)$, we have
$$
\lim_{k\to \infty}w(s_k + \tau, t) = w_\infty(\tau,t)
$$
in the $C^l(K \times [0,1], M)$ sense for any $l$.
\item $w_\infty$ has $Q = 0$ and the formula $w_\infty(\tau,t)= \gamma(T\, t)$, where $\gamma$ is some Reeb chord
joining $R_0$ and $R_1$ with action $T$, provided $T \neq 0$.
\end{enumerate}
\end{thm}

We mention that the asymptotic action $T$ could be either positive, 0 or negative.

\subsection{Exponential $C^\infty$ convergence}
\label{subsec:exponential-convergence}

In this section, we state the basic exponential decay results by combining the arguments used in
\cite[Section 7]{oh:contacton-Legendrian-bdy} and the scheme of the proof used
in \cite[Section 11-12]{oh-wang:CR-map2} for
the closed string case \emph{under the hypothesis} of
vanishing charge $Q = 0$. Since in the current open string case, this vanishing is
proved and so all the arguments used therein can be repeated verbatim after adapting them to the presence
of the boundary condition in the arguments.

The exponential decay results proved in \cite[Section 11-12]{oh-wang:CR-map2}
are based on the following three steps:
\begin{enumerate}
\item $L^2$-exponential decay of the Reeb component of $dw$,
\item $C^0$ exponential convergence,
\item $C^\infty$-exponential decay of $dw - R_\lambda(w) \, d\tau$.
\end{enumerate}
As mentioned in \cite{oh-wang:CR-map2}, this exponential decays result is a refinement of
that of the pseudoholomorphic curves in the symplectization given in such as
\cite{HWZ-asymptotics,bourgeois,bao1} in that our estimates do not involve the radial
component in the symplectization but depend only on the component of contact manifolds.

We start with the exponential decay of the Reeb component $w^*\lambda$.
We equip a punctured neighborhood around a puncture $z_i \in \del \Sigma$  with
strip-like coordinates $(\tau,t) \in [0,\infty) \times [0,1]$.
Then we  consider a complex-valued function
$$
\alpha(\tau,t) = \left(\lambda(\frac{\del w}{\del t})- T \right) + \sqrt{-1}\left(\lambda(\frac{\del w}{\del\tau})\right).
$$
By the Legendrian boundary condition, we know $\alpha(\tau, i) \in \R$ for $i =0, \, 1$.
The following lemma was proved in the closed string case in \cite{oh-wang:CR-map2} and in \cite{oh:contacton-Legendrian-bdy}
for the open string case of Legendrian boundary condition.

\begin{lem}[Lemma 7.1 \cite{oh:contacton-Legendrian-bdy}]\label{lem:exp-decay-lemma}
Suppose the complex-valued functions $\alpha$ and $\nu$ defined on $[0, \infty)\times [0,1]$
satisfy
\beastar
\begin{cases}
\delbar \alpha = \nu, \\
\text{\rm Im } \alpha(\tau, i) = 0 \, & \text{\rm for } i = 0,\, 1, \\
\lim_{\tau\rightarrow +\infty}\alpha(\tau,t) =0
\end{cases}
\eeastar
and
$$
\|\nu\|_{L^2([0,1])}+\left\|\nabla\nu\right\|_{L^2([0,1])}\leq Ce^{-\delta \tau}
\qquad \text{\rm for some constants } C, \delta>0.
$$
Then $\|\alpha_\tau \|_{L^2([0,1])}\leq \overline{C}e^{-\delta \tau}$
for some constant $\overline{C}$, where $\alpha_\tau \in L^2([0,1])$ is the map
defined by $\alpha_\tau(t): = \alpha(\tau,t)$.
\end{lem}

Now the $C^0$-exponential convergence of $w(\tau,\cdot)$ to some Reeb chord as $\tau \to \infty$
can be proved from the $L^2$-exponential estimates presented in previous sections by
the same argument as the proof of \cite[Lemma 11.22]{oh-wang:CR-map2} whose proof is omitted.

\begin{prop}[Proposition 7.2 \cite{oh:contacton-Legendrian-bdy}]\label{prop:C0-convergence}
Under Hypothesis \ref{hypo:basic}, for any contact instanton $w$ satisfying the Legendrian boundary condition,
there exists a unique Reeb orbit $z(\cdot)=\gamma(T\cdot):[0,1] \to M$ with period $T>0$, such that
\be\label{eq:C0-expdecay}
\|d(w(\tau, \cdot), z(\cdot))
\|_{C^0([0,1])}\rightarrow 0,
\ee
as $\tau\rightarrow +\infty$,
where $d$ denotes the distance on $M$ defined by the triad metric.
\end{prop}

Then the following $C^0$-exponential convergence is also proved in the same way as that of
\cite[Proposition 11.23]{oh-wang:CR-map2}
\begin{prop}[Proposition 7.4 \cite{oh:contacton-Legendrian-bdy}]\label{prop:C0-expdecay}
There exist some constants $C>0$, $\delta>0$ and $\tau_0$ large such that for any $\tau>\tau_0$,
\beastar
\|d\left( w(\tau, \cdot), z(\cdot) \right) \|_{C^0([0,1])} &\leq& C\, e^{-\delta \tau}
\eeastar
\end{prop}

So far, we have established the following:
\begin{itemize}
\item $W^{1,2}$-exponential decay of $w$,
\item $C^0$-exponential convergence of $w(\tau,\cdot) \to z(\cdot)$ as $\tau \to \infty$
for some Reeb chord $z$ between two Legendrians $R, \, R'$.
\end{itemize}
Combining the above two, we have obtained $L^2$-exponential estimates of the full derivative $dw$.

Then we consider the equation
$$
\delbar_{\pi} w = 0, \quad d(w^*\lambda \circ j) = 0
$$
back. By the bootstrapping argument using the local uniform a priori estimates
on the cylindrical region from Theorem \ref{thm:higher-regularity} (see also \cite{oh-wang:CR-map1}),  for the details),
the following $C^\infty$-exponential decay results are proved in \cite{oh-wang:CR-map1}, \
\cite[Section 11]{oh-wang:CR-map2}
for the closed string case, which can be easily adapted to the open string case of Reeb chords.

\begin{prop}[Corollary 6.5 \cite{oh-wang:CR-map1}]\label{prop:dw-expdecay}
Let $w:[0, \infty)\times S^1\to M$ satisfy the contact instanton equations \eqref{eq:contacton-Legendrian-bdy}
and Hypothesis \eqref{eq:hypo-basic-pt}. Then there exists some $R_0 > 0$, $C > 0$ and $\delta> 0$ such that
for all $\tau \geq R_0$
\be\label{eq:pidw-expdecay}
\left|\pi \frac{\del w}{\del\tau}(\tau, t)\right|\leq Ce^{-\delta |\tau|}, \quad
\left|\pi \frac{\del w}{\del t}(\tau, t)\right| \leq Ce^{-\delta |\tau|},
\ee
and
\be\label{eq:lambdadw-expdecay}
\left|\lambda(\frac{\del w}{\del\tau})(\tau, t)\right| \leq C e^{-\delta |\tau|}, \quad
\left|\lambda(\frac{\del w}{\del t})(\tau, t)- T\right| \leq C e^{-\delta |\tau|}
\ee
\be\label{eq:higher-derivatives}
|\nabla^l dw(s+\tau, t)| \leq C_\ell e^{- \delta |\tau|} \quad \text{for any}\quad l\geq 1.
\ee
\end{prop}

\section{Review of Fredholm theory of (relative) contact instantons}

We recall the framework for the study of generic
nondegeneracy results for the Reeb orbits and chords.

We first introduce the following definition
\begin{defn} \label{defn:Moore-path2}Let $T \geq 0$ and consider a curve $\gamma:[0,1] \to M$ be a
smooth curve. We say $(\gamma,T)$ an \emph{iso-speed  Reeb trajectory} if the pair satisfies
$$
\dot \gamma(t) = T R_\lambda(\gamma(t)), \, \int \gamma^*\lambda = T
$$
for all $t \in [0,1]$. If $\gamma(1) = \gamma(0)$, we call $(\gamma,T)$ an
iso-speed closed Reeb orbit and $T$ the \emph{action} of $\gamma$.
\end{defn}
We also consider the relative version thereof.

\begin{defn}[Iso-speed Reeb chords \cite{oh:entanglement1}]
 Let $(R_0,R_1)$ be a pair of Legendrian submanifolds of $(M,\xi)$ and $T\geq 0$.
For given contact form $\lambda$, we say $(\gamma, T)$ with $\gamma:[0,1] \to M$
is a Reeb chord from $R_0$ to $R_1$
if
$$
\dot \gamma(t) = T \dot R_\lambda(\gamma(t)), \quad \gamma(0) \in R_0, \, \gamma(1) \in R_1.
$$
We call any such $(\gamma, T)$ an \emph{iso-speed Reeb chord} and
\emph{nonnegative} if $T \geq 0$.
\end{defn}

 We alert readers that  any constant curve valued at
a point from $R_0 \cap R_1$ is a iso-speed Reeb chord with speed 0.
\begin{rem}  Note that when $R_0 = R_1$, we have `lots of iso-speed Reeb chords'
arising from the constant chords. We will show that this component of constant chords
is nondegenerate in the Bott-Morse sense.
This is important in the study of contact instanton Legendrian Floer homology we introduce
especially in its calculation when $R_1$ is contact isotopic to $R_0$ and $C^1$-close thereto
in our solution to Sandon's question from \cite{sandon-translated}.
This is the main reason why our generic nondegeneracy includes the constant trajectories over
the fixed domain $[0,1]$.
\end{rem}

\begin{defn} Let $(\gamma,T)$ be an iso-speed closed Reeb orbit in the sense as above.
When $|T| > 0$ is minimal among such that $\gamma(1) = \gamma(0)$ with $\int \gamma^*\lambda \neq 0$,
we call the pair $(\gamma,T)$ a \emph{simple} iso-speed closed Reed orbit.
\end{defn}

We now study the property of nondegeneracy of the pair $(\gamma, T)$ by formulating the notion
of nondegeneracy precisely including the case of constant trajectories, i.e., the case with $T = 0$.

Let $(\gamma, T)$ be a closed Reeb orbit of action $T$.
By definition, we can write $\gamma(T) = \phi^T_{R_\lambda}(\gamma(0))$
for the Reeb flow $\phi^T= \phi^T_{R_\lambda}$ of the Reeb vector field $R_\lambda$.
In particular $p = \gamma(0)$ is a fixed point of the diffeomorphism $\phi^T$.
Since $L_{R_\lambda}\lambda = 0$, $\phi^T$ is a contact diffeomorphism and so
induces an isomorphism
$$
\Psi_\gamma : = d\phi^T(p)|_{\xi_p}: \xi_p \to \xi_p
$$
which is the linearization restricted to $\xi_p$ of the Poincar\'e return map.

\begin{defn} Let $T> 0$. We say a $T$-closed Reeb orbit $(T,\lambda)$ is \emph{nondegenerate}
if $\Psi_\gamma:\xi_p \to \xi_p$ with $p = \gamma(0)$ has not eigenvalue 1.
\end{defn}
When $T=0$, it is well-known that the constant loop is nondegenerate in the Bott-Morse sense.

For $T > 0$, the following generic nondegeneracy result is well-known to the experts.
(See \cite[Appendix A]{ABW} for its proof.)

\begin{thm}[Albers-Braam-Wendl]\label{thm:ABW}
There exists a residual subset
$$
\CC^{\text{\rm reg}}(M,\xi) \subset \CC(M,\xi)
$$
such that for any $\lambda \in \CC^{\text{\rm reg}}(M,\xi)$ all the
closed Reeb orbits $\lambda$ are nondegenerate if $T> 0$.
\end{thm}

The following generic nondegeneracy result
for Reeb chords which extends the above generic nondegeneracy results to
the case of open strings of Reeb chord and to the Bott-Morse situation of constant chords
is proved in \cite[Appendix B]{oh:contacton-transversality}.

\begin{thm} \label{thm:Reeb-chords}
Let $(M,\xi)$ be a contact manifold. Let  $(R_0,R_1)$ be a pair of Legendrian submanifolds
allowing the case $R_0 = R_1$.
\begin{enumerate}
\item For a given pair $(R_0,R_1)$, there exists a residual subset
$$
\CC^{\text{\rm reg}}(\xi;R_0,R_1) \subset \CC(M,\xi)
$$
such that for any $\lambda \in \CC^{\text{\rm reg}}(\xi;R_0,R_1) $ all
Reeb chords from $R_0$ to $R_1$ are nondegenerate for $T > 0$ and
Bott-Morse nondegenerate when $T = 0$.
\item For a given contact form $\lambda$, there exists a residual subset of pairs $(R_0,R_1)$
of Legendrian submanifolds such that  all Reeb chords from $R_0$ to $R_1$ are
nondegenerate for $T > 0$ and
Bott-Morse nondegenerate when $T = 0$.
\end{enumerate}
\end{thm}

\subsection{Off-shell description of moduli spaces}

For the exposition of this section, we adapt the one given in
\cite{oh:contacton} to the current context of bordered contact instantons by
incorporating the Legendrian boundary condition.

Let $(\Sigma, j)$ be a bordered compact Riemann surface, and let $\dot \Sigma$ be the
punctured Riemann surface with $\{z_1,\ldots, z_k \} \subset \del \Sigma$, we consider the moduli space
$$
\CM_k((\dot \Sigma,\del \dot \Sigma),(M,\vec R)), \quad \vec R = (R_1,\cdots, R_k)
$$
 of finite energy maps $w: \dot \Sigma \to M$ satisfying the equation
\eqref{eq:contacton-Legendrian-bdy}.

We will be mainly interested in the two cases:
\begin{enumerate}
\item A generic nondegenerate case of $R_1, \cdots, R_k$ which in particular
are mutually disjoint,
\item The case where $R_1, \cdots, R_k = R$.
\end{enumerate}

The  second case is transversal in the Bott-Morse sense both for the Reeb
chords and for the moduli space of contact instantons, which is
rather straightforward and easier to handle, and so omitted.

For the first case, all the asymptotic Reeb chords  are nonconstant and have nonzero
action $T \neq 0$. In particular, the relevant punctures $z_i$
are not removable. Therefore we have the decomposition of the finite energy moduli space
$$
\CM_k((\dot \Sigma,\del \dot \Sigma),(M,\vec R)) =
\bigcup_{\vec \gamma \in \prod_{i=0}^{k-1}\frak{Reeb}(R_i,R_{i+1})} \CM(\vec \gamma),
\quad \gamma = (\gamma_1, \ldots, \gamma_k).
$$
Depending on the choice of strip-like coordinates we divide the punctures
$$
\{z_1, \cdots, z_k\} \subset \del \Sigma
$$
into two subclasses
$$
p_1, \cdots, p_{s^+}, q_1, \cdots, q_{s^-} \in \del \Sigma
$$
as the positive and negative boundary punctures. We write $k = s^+ + s^-$.

Let $\gamma^+_i$ for $i =1, \cdots, s^+$ and $\gamma^-_j$ for $j = 1, \cdots, s^-$
be two given collections of Reeb chords at positive and negative punctures
respectively. Following the notations from \cite{behwz}, \cite{bourgeois},
we denote by $\underline \gamma$ and $\overline \gamma$ the corresponding
collections
\beastar
\underline \gamma & = & \{\gamma_1^+,\cdots, \gamma_{s^+}\} \\
\overline \gamma & = & \{\gamma_1^+,\cdots, \gamma_{s^-}\}.
\eeastar
For each $p_i$ (resp. $q_j$), we associate the strip-like
coordinates $(\tau,t) \in [0,\infty) \times S^1$ (resp. $(\tau,t) \in (-\infty,0] \times S^1$)
on the punctured disc $D_{e^{-2\pi K_0}}(p_i) \setminus \{p_i\}$
(resp. on $D_{e^{-2\pi R_0}}(q_i) \setminus \{q_i\}$) for some sufficiently large $K_0 > 0$.

\begin{defn}\label{defn:Banach-manifold} We define
\be\label{eq:offshell-space}
\CF(\dot \Sigma, \del \dot \Sigma),(M, \vec R);J;\underline \gamma,\overline \gamma)
\ee
to the set of smooth maps satisfying the boundary condition
\be\label{eq:bdy-condition}
w(z) \in R_i \quad \text{ for } \, z \in \overline{z_{i-1}z_i} \subset \del \dot \Sigma
\ee
and the asymptotic condition
\be\label{eq:limatinfty}
\lim_{\tau \to \infty}w((\tau,t)_i) = \gamma^+_i(T_i(t+t_i)), \qquad
\lim_{\tau \to - \infty}w((\tau,t)_j) = \overline \gamma_j(T_j(t-t_j))
\ee
for some $t_i, \, t_j \in S^1$, where
$$
T_i = \int_{S^1} (\gamma^+_i)^*\lambda, \, T_j = \int_{S^1} ( \gamma^-_j)^*\lambda.
$$
Here $t_i,\, t_j$ depends on the given analytic coordinate and the parameterizations of
the Reeb chords.
\end{defn}
We will fix $j$ and its associated K\"ahler metric $h$.
We regard the assignment
$$
\Upsilon: w \mapsto \left(\delbar^\pi w, d(w^*\lambda \circ j)\right), \quad
\Upsilon: = (\Upsilon_1,\Upsilon_2)
$$
as a section of the (infinite dimensional) vector bundle:
We first formally linearize and define a linear map
$$
D\Upsilon(w): \Omega^0(w^*TM,(\del w)^*T\vec R) \to \Omega^{(0,1)}(w^*\xi) \oplus \Omega^2(\Sigma)
$$
where we have the tangent space
$$
T_w \CF = \Omega^0(w^*TM,(\del w)^*T\vec R).
$$
For the optimal expression of the linearization map and its relevant
calculations, we use the contact triad connection $\nabla$ of $(M,\lambda,J)$ and the contact
Hermitian connection $\nabla^\pi$ for $(\xi,J)$ introduced in
\cite{oh-wang:connection,oh-wang:CR-map1}.

\begin{rem} When one deals with the case of Morse-Bott family of Reeb chords (or orbits),
one needs to use a weighted Sobolev spaces with weight $e^{\delta|\tau|}$ in
setting up the appropriate off-shell function spaces for a suitable choice of $\delta > 0$
which depends on the exponential decay result established for the finite $\pi$-energy contact
instantons in \cite{oh-wang:CR-map2}.
To describe the choice of exponential weight $\delta > 0$, we need to recall the covariant linearization of the map $
D\Phi_\lambda: W^{1,2}(\gamma^*\xi) \to L^2(\gamma^*\xi)
$
of the map
$$
\Phi_\lambda: (\gamma, T) \mapsto \dot \gamma - T\, R_\lambda(\gamma)
$$
for a given $T$-periodic Reeb orbit $(\gamma, T)$. The operator has the expression
\be\label{eq:DUpsilon}
D\Phi_\lambda = \frac{D^\pi}{dt} - \frac{T}{2}(\CL_{R_\lambda}J) J=: A_{(\gamma, T)}
\ee
where $\frac{D^\pi}{dt}$ is the covariant derivative
with respect to the pull-back connection $\gamma^*\nabla^\pi$ along the Reeb orbit
$\gamma$ and $(\CL_{R_\lambda}J) J$ is (pointwise) symmetric operator with respect to
the triad metric. (See \cite[Lemma 3.4]{oh-wang:CR-map1}.)
We choose $\delta> 0$ so that $0 < \delta/p < 1$ is smaller than the spectral gap
\be\label{eq:gap}
\text{gap}(\underline \gamma,\overline \gamma): = \min_{i,j}
\{d_{\text H}(\text{spec}A_{(T_i,\gamma_i)},0),\, d_{\text H}(\text{spec}A_{(T_j,\gamma_j)},0)\}.
\ee
\end{rem}

Let $k \geq 2$ and $p > 2$. We denote by
\be\label{eq:CWkp}
\CW^{k,p}: = \CW^{k,p}((\dot \Sigma, \del \dot \Sigma),(M, \vec R); \underline \gamma,\overline \gamma)
\ee
the completion of the space \eqref{eq:offshell-space}.
It has the structure of a Banach manifold modelled by the Banach space given by the following

\begin{defn}\label{defn:tangent-space} We define
$$
W^{k,p} (w^*TM, (\del w)^*T\vec R; \underline \gamma,\overline \gamma)
$$
to be the set of vector fields $Y = Y^\pi + \lambda(Y) R_\lambda$ along $w$ that satisfy
\be\label{eq:tangent-element}
\begin{cases}
Y^\pi \in W^{k,p}((\dot\Sigma, \del \dot \Sigma), (w^*\xi, (\del w)^*T\vec R),  \\
\lambda(Y) \in W^{k,p}((\dot \Sigma, \del \dot \Sigma),(\R, \{0\})),\\
Y^\pi(\del \dot \Sigma) \subset T\vec R, \qquad \lambda(Y)(\del \dot \Sigma) = 0.
\end{cases}
\ee
\end{defn}
Here we use the splitting
$$
TM = \xi \oplus \span_\R\{R_\lambda\}
$$
where $\span_\R\{R_\lambda\}: = \CL$ is a trivial line bundle and so
$$
\Gamma(w^*\CL) \cong C^\infty(\dot \Sigma, \del \dot \Sigma).
$$
The above Banach space is decomposed into the direct sum
\be\label{eq:tangentspace}
W^{k,p}((\dot\Sigma,\del \dot \Sigma),( w^*\xi, (\del w)^*T\vec R))
\bigoplus W^{k,p}((\dot \Sigma,\del \dot \Sigma), ( \R, \{0\})) \otimes R_\lambda :
\ee
by writing $Y = (Y^\pi, g R_\lambda)$ with a real-valued function $g = \lambda(Y(w))$ on $\dot \Sigma$.
Here we measure the various norms in terms of the triad metric of the triad $(M,\lambda,J)$.

Now for each given $w \in \CW^{k,p}((\dot \Sigma,\del \dot \Sigma), (M, \vec R);J;\underline \gamma,\overline \gamma)$,
we consider the Banach space
$$
\Omega^{(0,1)}_{k-1,p}(w^*\xi): = W^{k-1,p}(\Lambda^{(0,1)}(w^*\xi))
$$
the $W^{k-1,p} $-completion of $\Omega^{(0,1)}(w^*\xi) = \Gamma(\Lambda^{(0,1)}(w^*\xi)$ and form the bundle
\be\label{eq:CH01}
\bigcup_{w \in \CW^{k,p}} \Omega^{(0,1)}_{k-1,p}(w^*\xi)
\ee
over $\CW^{k,p}$.

\begin{defn}\label{defn:CHCM01} We associate the Banach space
\be\label{eq:CH01-w}
\CH^{(0,1)}_{k-1,p}(M,\lambda)|_w: = \Omega^{(0,1)}_{k-1,p}(w^*\xi) \oplus \Omega^2_{k-2,p}(\dot \Sigma)
\ee
to each $w \in \CW^{k,p}$ and form the bundle
$$
\CH^{(0,1)}_{k-1,p}(M,\lambda): = \bigcup_{w \in \CW^{k,p}} \CH^{(0,1)}_{k-1,p}(M,\lambda)|_w
$$
over $\CW^{k,p}$.
\end{defn}

Then we can regard the assignment
$$
\Upsilon_1: w \mapsto \delbar^\pi w
$$
as a smooth section of the bundle $\CH^{(0,1)}_{k-1,p}(M,\lambda) \to \CW^{k,p}$. Furthermore
the assignment
$$
\Upsilon_2: w \mapsto d(w^*\lambda \circ j)
$$
defines a smooth section of the trivial bundle
$$
\Omega^2_{k-2,p}(\Sigma) \times \CW^{k,p} \to \CW^{k,p}
$$
for the Banach manifold
$$
\CW^{k,p}:= \CW^{k,p}((\dot \Sigma,\del \dot \Sigma),(M, \vec R);J;\underline \gamma,\overline \gamma).
$$
We summarize the above discussion to the following lemma.

\begin{lem}\label{lem:Upsilon} Consider the vector bundle
$$
\CH^{(0,1)}_{k-1,p}(M,\lambda) \to \CW^{k,p}.
$$
The map $\Upsilon$ continuously extends to a continuous section
$$
\Upsilon: \CW^{k,p} \to \CH^{(0,1)}_{k-1,p} (\xi;\vec R).
$$
\end{lem}
We have already computed the linearization of $\Upsilon$ in the previous section.

With these preparations, the following is a consequence of the exponential estimates established
in \cite{oh-wang:CR-map1} for the case of vanishing charge.
(The relevant off-shell analytical framework
for the case $Q(p_i) \neq 0$ is treated in \cite{oh-savelyev}.)

\begin{prop}[Theorem 1.12 \cite{oh-wang:CR-map1}]
Assume $\lambda$ is nondegenerate and $Q(p_i) = 0$.
Let $w:\dot \Sigma \to M$ be a contact instanton and let $w^*\lambda = a_1\, d\tau + a_2\, dt$.
Suppose
\bea
\lim_{\tau \to \infty} a_{1,i} = 0, &{}& \, \lim_{\tau \to \infty} a_{2,i} = T(p_i)\nonumber\\
\lim_{\tau \to -\infty} a_{1,j} = 0, &{}& \, \lim_{\tau \to -\infty} a_{2,j} = T(p_j)
\eea
at each puncture $p_i$ and $q_j$.
Then $w \in \CW^{k,p}(\dot \Sigma, M;J;\underline \gamma,\overline \gamma)$.
\end{prop}

Now we are ready to define the moduli space of contact instantons with prescribed
asymptotic condition.
\begin{defn}\label{defn:tilde-modulispace} Consider the zero set of the section $\Upsilon$
\be\label{eq:defn-tildeMM}
\widetilde \CM((\dot \Sigma,\del \dot \Sigma),(M,\vec R);J;\underline \gamma,\overline \gamma) =  \Upsilon^{-1}(0)
\ee
in the Banach manifold $\CW^{k,p}((\dot \Sigma,\del \dot \Sigma),(M,\vec R);J;\underline \gamma,\overline \gamma)$, and
\be\label{eq:defn-MM}
\widetilde \CM((\dot \Sigma,\del \dot \Sigma),(M,\vec R);J;\underline \gamma,\overline \gamma)
= \widetilde \CM((\dot \Sigma,\del \dot \Sigma),(M,\vec R);J;\underline \gamma,\overline \gamma)/\sim
\ee
to be the set of isomorphism classes of contact instantons $w$.
\end{defn}

\subsection{Linearized operator and its ellipticity}
\label{subsec:fredholm}

Let $(\dot \Sigma, j)$ be a punctured Riemann surface, the set of whose punctures
may be empty, i.e., $\dot \Sigma = \Sigma$ is either a closed or a punctured
Riemann surface. In this subsection and the next, we lay out the precise relevant off-shell framework
of functional analysis, and  establish the Fredholm property of the linearization map.

Then we have the following explicit
formula thereof.

\begin{thm}[Theorem 10.1 \cite{oh:contacton}; See also Theorem 1.15
\cite{oh-savelyev}] \label{thm:linearization} In terms of the decomposition
$d\pi = d^\pi w + w^*\lambda\, R_\lambda$
and $Y = Y^\pi + \lambda(Y) R_\lambda$, we have
\bea
D\Upsilon_1(w)(Y) & = & \delbar^{\nabla^\pi}Y^\pi + B^{(0,1)}(Y^\pi) +  T^{\pi,(0,1)}_{dw}(Y^\pi)\nonumber\\
&{}& \quad + \frac{1}{2}\lambda(Y) (\CL_{R_\lambda}J)J(\del^\pi w)
\label{eq:DUpsilon1}\\
D\Upsilon_2(w)(Y) & = &  - \Delta (\lambda(Y))\, dA + d((Y^\pi \rfloor d\lambda) \circ j)
\label{eq:DUpsilon2}
\eea
where $B^{(0,1)}$ and $T_{dw}^{\pi,(0,1)}$ are the $(0,1)$-components of $B$ and
$T_{dw}^\pi$, where $B, \, T_{dw}^\pi: \Omega^0(w^*TM) \to \Omega^1(w^*\xi)$ are
 zero-order differential operators given by
\be\label{eq:B}
B(Y) =
- \frac{1}{2}  w^*\lambda \otimes \left((\CL_{R_\lambda}J)J Y\right)
\ee
and
\be\label{eq:torsion-dw}
T_{dw}^\pi(Y) = \pi T(Y,dw)
\ee
respectively.
\end{thm}

\subsection{Fredholm theory on punctured bordered Riemann surfaces}

By the (local) ellipticity shown in the previous subsection, it remains to examine the
Fredholm property of the linearized operator $D\Upsilon(w)$. For this purpose,
we need to examine the asymptotic behavior of the operator near punctures in strip-like
coordinates.

We first decompose the section $Y \in w^*TM$ into
$$
Y = Y^\pi + \lambda(Y) R_\lambda.
$$
Then we have
\bea
D\Upsilon_1^1(w)(Y^\pi) & = & \delbar^{\nabla^\pi}Y^\pi + B^{(0,1)}(Y^\pi) +  T^{\pi,(0,1)}_{dw}(Y^\pi), \label{eq:D}\\
D\Upsilon_2^1(w)(Y^\pi) & = & d((Y^\pi \rfloor d\lambda) \circ j),\label{eq:DUpsilon21}\\
D\Upsilon_1^2(w)(\lambda(Y) R_\lambda) & = & \frac12 \lambda(Y) \CL_{R_\lambda}J J \del^\pi w,
\label{eq:DUpsilon12}\\
D\Upsilon_2^2(w)(\lambda(Y) R_\lambda) & = & - \Delta(\lambda(Y))\, dA. \label{eq:DUpsilon22}
\eea
Noting that $Y^\pi$ and $\lambda(Y)$ are independent of each other, we write
$$
Y = Y^\pi + f R_\lambda, \quad f: = \lambda(Y)
$$
where $f: \dot \Sigma \to \R$ is an arbitrary function satisfying the boundary condition
$$
Y^\pi(\del \dot \Sigma) \subset T\vec R, \quad f(\del \dot \Sigma) = 0
$$
by the Legendrian boundary condition satisfied by $Y$. The following is obvious from
the expression of the $D\Upsilon_i^j(w)$.
\begin{lem} Suppose that $w$ is a solution to \eqref{eq:contacton-Legendrian-bdy}.
The operators $D\Upsilon_i^j(w)$ have the following continuous extensions:
\beastar
D\Upsilon_1^1(w)(Y^\pi)& : & \Omega^0_{k,p}(w^*\xi,(\del w)^*T\vec R) \to \Omega^{(0,1)}_{k-1,p}(w^*\xi) \\
D\Upsilon_2^1(w)(Y^\pi) &: & \Omega^0_{k,p}(w^*\xi,(\del w)^*T\vec R)
 \to
\Omega^2_{k-1,p}(\dot \Sigma) \hookrightarrow \Omega^2_{k-2,p}(\dot \Sigma) \\
D\Upsilon_1^2(w)((\cdot) R_\lambda) & : & \Omega^0_{k,p}(\dot \Sigma,\del \dot \Sigma)
 \to
\Omega^2_{k,p}(\dot \Sigma) \hookrightarrow \Omega^2_{k-2,p}(\dot \Sigma) \\
D\Upsilon_2^2(w)((\cdot) R_\lambda) & : & \Omega^0_{k,p}(\dot \Sigma,\del \dot \Sigma) \to \Omega^2_{k-2,p}(\Sigma).
\eeastar
\end{lem}

We regard the domains of  $D\Upsilon_i^2$ for $i=1, \,2$
as $C^\infty(\dot \Sigma, \del \dot \Sigma)$ using the isomorphism
$$
C^\infty(\dot \Sigma, \del \dot \Sigma) \cong \Omega^0(\dot \Sigma, \del \dot \Sigma) \otimes R_\lambda.
$$
The following Fredholm property of the linearized operator is proved in
\cite{oh:contacton-transversality}.

\begin{prop}\label{prop:closed-fredholm}
Suppose that $w$ is a solution to \eqref{eq:contacton-Legendrian-bdy}.
Consider the completion of $D\Upsilon(w)$,
which we still denote by $D\Upsilon(w)$, as a bounded linear map
from $\Omega^0_{k,p}(w^*TM,(\del w)^*T\vec R)$ to
$\Omega^{(0,1)}(w^*\xi)\oplus \Omega^2(\Sigma)$
for $k \geq 2$ and $p \geq 2$. Then
\begin{enumerate}
\item The off-diagonal terms of $D\Upsilon(w)$ are relatively compact operators
against the diagonal operator.
\item
The operator $D\Upsilon(w)$ is homotopic to the operator
\be\label{eq:diagonal}
\left(\begin{matrix}\delbar^{\nabla^\pi} + T_{dw}^{\pi,(0,1)}+ B^{(0,1)} & 0 \\
0 & -\Delta(\lambda(\cdot)) \,dA
\end{matrix}
\right)
\ee
via the homotopy
\be\label{eq:s-homotopy}
s \in [0,1] \mapsto \left(\begin{matrix}\delbar^{\nabla^\pi} + T_{dw}^{\pi,(0,1)} + B^{(0,1)}
& \frac{s}{2} \lambda(\cdot) (\CL_{R_\lambda}J)J (\pi dw)^{(1,0)} \\
s\, d\left((\cdot) \rfloor d\lambda) \circ j\right) & -\Delta(\lambda(\cdot)) \,dA
\end{matrix}
\right) =: L_s
\ee
which is a continuous family of Fredholm operators.
\item And the principal symbol
$$
\sigma(z,\eta): w^*TM|_z \to w^*\xi|_z \oplus \Lambda^2(T_z\Sigma), \quad 0 \neq \eta \in T^*_z\Sigma
$$
of \eqref{eq:diagonal} is given by the matrix
\beastar
\left(\begin{matrix} \frac{\eta + i\eta \circ j}{2} Id  & 0 \\
0 & |\eta|^2
\end{matrix}\right).
\eeastar
\end{enumerate}
\end{prop}

Now we wrap-up the discussion of the Fredholm property of
the linearization map
\be\label{eq:matrixDUpsilon}
D\Upsilon_{(\lambda,T)}(w): \Omega^0_{k,p}(w^*TM, (\del w)^*T\vec R;J;\underline \gamma,\overline \gamma) \to
\Omega^{(0,1)}_{k-1,p}(w^*\xi) \oplus \Omega^2_{k-2,p}(\dot \Sigma).
\ee
Then by the continuous invariance of the Fredholm index, we obtain
\be\label{eq:indexDXiw}
\operatorname{Index} D\Upsilon_{(\lambda,T)}(w) =
\operatorname{Index} \left(\delbar^{\nabla^\pi} + T^{\pi,(0,1)}_{dw}  + B^{(0,1)}\right) + \operatorname{Index}(-\Delta).
\ee
The computation of index is given in \cite{oh:contacton} for the closed string case
and will be given for the current open string case elsewhere.

\begin{rem}\label{rem:weighted-setting}
Suppose $\delta > 0$ satisfies the inequality
$$
0\leq \delta < \min\left\{\frac{\text{\rm gap}(\vec \gamma)}{p}, \frac{2\pi}{p}\right\}
$$
where $\text{\rm gap}(\vec \gamma)$ is the spectral gap, given in \eqref{eq:gap},
of the asymptotic operators $A_{(T_j,z_j)}$ or $A_{(T_i,z_i)}$
associated to the corresponding punctures. Then the above Fredholm property
also holds in the weighted Sobolev space setting with exponential weight $e^{\delta|\tau|}$.
\end{rem}

\subsection{Generic transversality of (relative) contact instantons}
\label{subsec:moduli-space}

In this section, we summarize two main generic transversality results for the
contact instanton moduli spaces
needed for the study of the moduli problem of contact instanton maps, whose proofs
are not very different from that of the closed string case from \cite{oh:contacton} so
will be brief in its proof.  One transversality result is for the perturbation of
contact forms and the other is for the perturbation of boundary Legendrian submanifolds.

Let a contact manifold $(M,\xi)$ be given.
We consider the contact forms $\lambda$ of $(M,\xi)$ such that
all Reeb chords are nondegenerate. The set of such contact forms is
residual in $\CC(M,\xi)$. (See \cite[Appendix B]{oh:contacton-transversality} for the proof.)

Now we involve the set $\CJ(M,\lambda)$ given in Definition \ref{defn:adapted-J}.
We study the linearization of the map $\Upsilon^{\text{\rm univ}}$ which is the map $\Upsilon$ augmented by
the argument $J \in \CJ(M,\lambda)$. More precisely, we define the universal section
$$
\Upsilon^{\text{\rm univ}}: \CF \times  \CJ(M,\lambda) \to \CH^{(0,1)}(M,\lambda)
$$
given by
$$
\Upsilon^{\text{\rm univ}}(j, w, J) = \left(\delbar_J^\pi w, d(w^*\lambda \circ j)\right)
$$
and study its linearization at each $(j,w,J) \in (\Upsilon^{\text{\rm univ}})^{-1}(0)$.
In the discussion below, we will fix the complex
structure $j$ on $\Sigma$, and so suppress $j$ from the argument of $\Upsilon^{\text{\rm univ}}$.

Following the procedure of considering the set $\CJ^\ell(M,\lambda)$ of $\lambda$-adapted
$C^\ell$ CR-almost complex structures $J$ inductively as $\ell$ grows (see \cite{mcduff-salamon-symplectic}, \cite[Section 10.4]{oh:book1}
for the detailed explanation),
we denote the zero set $(\Upsilon^{\text{\rm univ}})^{-1}(0)$ by
$$
\MM(M,\lambda,\vec R;\overline \gamma, \underline \gamma) = \left\{ (w,J)
\in \CW^{k,p}(\dot \Sigma, M;\overline \gamma, \underline \gamma)) \times \JJ^\ell(M,\lambda)
\, \Big|\, \Upsilon^{\text{\rm univ}}(w, J) = 0 \right\}
$$
which we call the universal moduli space. Denote by
$$
\Pi_2: \CW^{k,p}(\dot \Sigma, M, \vec R;\overline \gamma, \underline \gamma) \times \JJ^\ell(M,\lambda) \to
\JJ^\ell(M,\lambda)
$$
the projection. Then we have
\be\label{eq:MMK}
\MM(J;\overline \gamma, \underline \gamma)= \MM(M,\lambda, \vec R; J;\overline \gamma, \underline \gamma)
 = \Pi_2^{-1}(J) \cap \MM(M,\lambda,  \vec R;\overline \gamma, \underline \gamma).
\ee
We state the following standard statement that often occurs in this kind of generic
transversality statement via the Sard-Smale theorem.

\begin{thm}\label{thm:trans} Let $0 < \ell < k -\frac{2}{p}$.
Consider the moduli space $\MM(M,\lambda;\overline \gamma, \underline \gamma)$. Then
\begin{enumerate}
\item $\MM(M,\lambda, \vec R;\overline \gamma, \underline \gamma)$ is
an infinite dimensional $C^\ell$ Banach manifold.
\item The projection
$$
\Pi_2|_{(\Upsilon^{\text{\rm univ}})^{-1}(0)} : (\Upsilon^{\text{\rm univ}})^{-1}(0) \to \JJ^\ell(M,\lambda)
$$
is a Fredholm map and its index is the same as that of $D\Upsilon(w)$
for a (and so any) $w \in  \MM(M,\lambda, \vec R; J;\overline \gamma, \underline \gamma)$.
\end{enumerate}
\end{thm}

An immediate corollary of Sard-Smale theorem is that for a generic choice of $J$
$$
\Pi_2^{-1}(J) \cap (\Upsilon^{\text{\rm univ}})^{-1}(0)= \MM(J;\overline \gamma, \underline \gamma)
$$
is a smooth manifold.

\section{Energy of contact instantons}
\label{sec:energy}

The following is a part of energy, called the \emph{$\pi$-energy} in \cite{oh:entanglement1}
(or $\omega$-energy in \cite{behwz})
which is the horizontal part of total energy.

\begin{defn}\label{defn:pi-energy}
For a smooth map $\dot \Sigma \to M$, we define the $\pi$-energy of $w$ by
\be\label{eq:Epi}
E^\pi(j,w) = \frac{1}{2} \int_{\dot \Sigma} |d^\pi w|^2.
\ee
\end{defn}

The Riemann surfaces that we concern about in relation to the gluing theory
are of the following three types:
\begin{situ}[Charge vanishing]\label{situ:charge-vanish}
\begin{enumerate}
\item
First, we mention that the \emph{starting} Riemann surface will be an open Riemann surface of genus zero
with a finite number of boundary punctures, and mostly
$$
\dot \Sigma \cong \R \times [0,1]
$$
together with an contact instanton with Legendrian pair boundary condition $(R_0,R_1)$, or
$$
\dot \Sigma \cong D^2
$$
with \emph{moving Legendrian boundary condition}.
\item $\C$ which will appear in the bubbling analysis at an interior point of $\dot \Sigma$,
\item $\H = \{ z \in \C \mid \text{\rm Im } z \geq 0\}$ which will appear in the bubbling analysis at an boundary
point of $\dot \Sigma$.
\end{enumerate}
\end{situ}
An upshot is that \emph{the asymptotic charges vanish in all these three cases.} (See \cite{oh-wang:CR-map1} and
\cite{oh:contacton-Legendrian-bdy} and \cite{oh:entanglement1} for the discussion.)

This being said, we introduce the following class of test functions
following the modification made in \cite{behwz} of Hofer's original definition \cite{hofer:invent}.

\begin{defn}\label{defn:CC} We define
\be
\CC = \left\{\varphi: \R \to \R_{\geq 0} \, \Big| \, \supp \varphi \, \text{is compact}, \, \int_\R \varphi = 1\right\}
\ee
\end{defn}

Then on the given cylindrical neighborhood $D (p) \setminus \{p\}$, we can write
$$
w^*\lambda \circ j = df
$$
for some function $f: [0,\infty) \times S^1 \to \R$.

\begin{defn}[Contact instanton potential] We call the above function $f$
the \emph{contact instanton potential} of the contact instanton charge form $w^*\lambda \circ j$.
\end{defn}

We denote by $\psi$ the function determined by
\be\label{eq:psi}
\psi' = \varphi, \quad \psi(-\infty) = 0, \, \psi(\infty) = 1.
\ee
\begin{defn}\label{defn:CC-energy} Let $w$ satisfy $d(w^*\lambda \circ j) = 0$. Then we define
$$
E_{\CC}(j,w;p) = \sup_{\varphi \in \CC} \int_{D (p) \setminus \{p\}} df\circ j \wedge d(\psi(f))
= \sup_{\varphi \in \CC} \int_{D (p) \setminus \{p\}}  (- w^*\lambda ) \wedge d(\psi(f)).
$$
\end{defn}

We note that
$$
df \circ j \wedge d(\psi(f)) = \psi'(f) df\circ j \wedge df = \varphi(f) df\circ j  \wedge df \geq 0
$$
since
$$
df\circ j  \wedge df = |df|^2\, d\tau \wedge dt.
$$
Therefore we can rewrite $E_{\CC}(j,w)$ into
$$
E_{\CC}(j,w;p) = \sup_{\varphi \in \CC} \int_{D (p) \setminus \{p\}} \varphi(f) df \circ j \wedge df.
$$
The following proposition is proved in \cite{oh:entanglement1} which shows that the definition of $E_{\CC}(j,w)$ does not
depend on the constant shift in the choice of $f$.

\begin{prop}[Proposition 13.6, \cite{oh:entanglement1}]\label{prop:a-independent}
For a given smooth map $w$ satisfying $d(w^*\lambda \circ j) = 0$,
we have $E_{\CC;f}(w) = E_{\CC,g}(w)$ for any pair $(f,g)$ with
$$
df = w^*\lambda\circ j = dg
$$
on $D^2 (p) \setminus \{p\}$.
\end{prop}

This proposition enables us to introduce the following

\begin{defn}[Vertical energy]  Consider a smooth map  $w: \dot \Sigma \to M$.
\begin{enumerate}
\item For given puncture $p$ of $\dot \Sigma$,
we denote the common value of $E_{\CC,f}(j,w;p)$ by $E^\lambda_p(w)$,
and call it the \emph{$\lambda$-energy of $w$ at $p$}.
\item We define the \emph{vertical energy}, denoted by $E^\perp(j,w)$, to be the sum
$$
E^\perp(j,w) = \sum_{l=1}^k E^\lambda_{p_l}(w).
$$
\end{enumerate}
\end{defn}

Now we define the final form of the off-shell energy.
\begin{defn}[Total energy]\label{defn:total-enerty}
Let $w:\dot \Sigma \to M$ be any smooth map. We define the \emph{total energy} to be
the sum
\be\label{eq:final-total-energy}
E(j,w) = E^\pi(j,w) + E^\perp(j,w).
\ee
\end{defn}

\part{Gluing theory of (relative) contact instantons}
\label{part:gluing-contacton}

One of the essential ingredients in the construction of Floer-type invariants involving
a family of moduli spaces is the gluing theorem that ensures some compatibility conditions for
the system of moduli spaces leading to a relevant algebraic structures, such as the
Floer homology or more generally Fukaya category, especially in the chain level.
We will exemplify our gluing construction by focusing on the case that is needed for the
construction of (cylindrical) Legendrian contact instanton Floer homology which is
used for the proof of the nondegenerate version of Sandon-Shelukhin's conjecture
in \cite{oh:entanglement1}. The general case is the same except the addition of notational complexities
to do the writing of details of the construction.

\section{Description of the gluing problem}

Recall that for a general Legendrian pair $(R_0,R_1)$ the moduli space $\CM(\gamma_-, \gamma_+)$ is defined to be the quotient
$$
\CM(\gamma_-,\gamma_+) = \widetilde \CM(\gamma_-,\gamma_+)/\R
$$
i.e., the set of equivalence classes of maps $u: \R \times [0,1] \to M$ satisfying
\be\label{eq:Legendrian-trajectory}
\begin{cases}
\delbar^\pi u = 0, \quad d(u^*\lambda \circ j) = 0, \\
u(\tau,0) \in R_0, \quad u(\tau,1) \in R_1,\\
u(-\infty) = \gamma_-, \quad u(\infty) = \gamma_+
\end{cases}
\ee
with finite energy $E(u) < \infty$.

\subsection{Basic estimates for approximate solutions and right inverses}

By an abuse of notation, we will also denote an equivalence class of
map $u$ by the same letter $u$ as long as it does not cause confusion in the following
discussion.

Let $(T_\pm,\gamma_\pm)$ and $(T',\gamma')$ be nondegenerate Reeb chords for the pair $(\psi(R),R)$,
and assume that both moduli spaces
$$
\CM(\psi(R),R;\gamma_-,\gamma'),\quad  \CM(\psi(R),R;\gamma',\gamma_+)
$$
are transversal.  According to the well-established strategy of
gluing two Floer-type trajectories (see \cite{floer-unregularized},
\cite[Subsection15.5.2]{oh:book2}, for example),  the following two analytic ingredients,
together with relevant Fredholm theory, are needed to carry out Newton's iteration scheme.

In particular, it is essential to construct an approximate solution
 that is  good enough for the Newton iteration scheme not to diverge
 (see Appendix \ref{sec:nondegeneracy-chords}):
\begin{enumerate}
\item {[Approximate solution]} Construct a good approximate solution $u$ out of a pair
$$
(u_-,u_+,K) \to u_- \#_K u_+ =:\PreG(u_-,u_+,K)
$$
satisfying
$$
\|u_- \#_K u_+\| < \epsilon = \epsilon(\K_-,\K_+,K)
$$
for all $(u_-,u_+, K) \in \K_- \times \K_+ \times [K_0,\infty)$ where $\epsilon \to 0$ as $K_0 \to \infty$
uniformly over $(u_-,u_+) \in \K_- \times \K_+$ for given compact subsets
$$
\K_- \subset \CM(\psi(R),R;\gamma_-,\gamma'),\quad  \K_+ \subset\CM(\psi(R),R;\gamma',\gamma_+).
$$
\item {[Approximate right inverse]} Let $\Upsilon$ be the nonlinear map defined by
$$
\Upsilon(w) = (\delbar^\pi w, d(w^*\lambda \circ j)).
$$
Then we construct a good approximate right inverse $Q_{\text{\rm app}}(\uapp)
= Q_{\text{\rm app}}(u_+,u_-,K)$ for the
linearized operator $D\Upsilon(\uapp)$ out of $D\Upsilon(u_\pm)$ satisfying
\beastar
&{}& \|Q_{\text{\rm app}}(\uapp) \| \leq C,\\
&{}& \|(D\Upsilon(\uapp)) \circ  Q_{\text{\rm app}}(\uapp)  -id\| \leq \frac{1}{2}.
\eeastar
\end{enumerate}

At this stage, we apply the Picard contraction mapping theorem to
\eqref{eq:Gk}  as in the proof of Proposition \ref{prop:model}
to construct the genuine solution associated to $\PreG(u_-,u_+;K)$ which we
denote by $\Glue(u_-,u_+;K) = u_-\widehat\#_K u_+$: Here $\widehat \#$ stands for
"the operation $\#$ followed by perturbing to a genuine solution".
This defines a smooth map
$$
\Glue: \K_- \times \K_+ \times [K_0,\infty) \to \MM(\gamma_-,\gamma_+;B).
$$
This being said, we will mainly focus on constructing such a good approximate solution that satisfies
the above two estimates.

\subsection{Contact instanton gluing  versus  pseudoholomorphic curve gluing on
the symplectization} \label{subsec:comparison}

In this subsection,  we would like to point out difference between
the current gluing theory of contact instantons for the contact triad
$$
(M,\lambda, J)
$$
and that of pseudoholomorphic curves in the symplectization almost K\"ahler manifold
$$
(M \times \R, d(e^s \lambda), \widetilde J)
$$
where $\widetilde J$ is the \emph{cylindrical almost complex structure}
adapted to the triad $(M,\lambda, J)$: It is characterized by the requirement
$$
\widetilde J Y = J Y \quad \text{\rm for } Y \in \xi
$$
and
$$
\widetilde J(\frac{\del}{\del s}) = R_\lambda, \quad \widetilde J (R_\lambda) = - \frac{\del}{\del s}
$$
under the decomposition
$$
T(M \times \R) = \xi \oplus \R \langle R_\lambda \rangle \oplus \R \langle\frac{\del}{\del s} \rangle.
$$
Here are two outstanding differences between the two:
\begin{enumerate}
\item
Because of the translational symmetry  of $\widetilde J$, the associated moduli space
carries an $\R$-symmetry on $M \times \R$ induced by the target translation along the direction of $\R$,
one needs to quotient out the moduli space of $\widetilde J$-holomorphic curves as well as by
the action of domain automorphisms in the study of convergence properties e.g., such as its
compactification.  On the other hand, in the gluing theory of contact instantons,
there is no such \emph{a priori target symmetry}: As already pointed out in
\cite{oh-wang:CR-map1}, the results on regularity estimates and asymptotic exponential
convergence established in \cite{oh-wang:CR-map1} also apply to the horizontal component $w$
of $\widetilde J$-holomorphic maps $u = (w,f)$ for the symplectization, which
is finer than those established for $\widetilde J$-holomorphic map $u = (w,f)$
given e.g., in \cite{behwz}, \cite{bourgeois} in that the estimates on $w$ do not
involve the radial components $f$ at all. Existence of such estimates are plausible because the equation $w^*\lambda \circ j = df$
uniquely determine $f$ modulo addition of a uniform constant.
\item
Another point worthwhile to mention in relation to the compactification of contact instantons
on $M$ to be constructed in \cite{oh:entanglement2} is that the compactification is expected not to involve
unstable components, the so called `trivial cylinder' components that appear in the compactification
of pseudoholomorphic curves in the symplectization. (See \cite{oh:entanglement1} for the first step
towards compactification of contact instanton moduli spaces, and
\cite{EGH} and \cite{behwz} for the appearance of trivial cylinders in the pseudoholomorphic
curve compactification in the symplectization.) In this sense compactification of contact instantons
is similar to that of \cite{pardon-contact} that Pardon used in his construction of
contact homology using the virtual fundamental cycles on the symplectization $M \times \R$
and on symplectic buildings.
\end{enumerate}

\section{The off-shell function spaces for the gluing analysis}

Let $\dot \Sigma_\pm$ be a pair of punctured bordered Riemann surfaces.
We assume that each $\dot \Sigma_\pm$ are equipped strip-like coordinates
$(\tau,t)$ of $\R$ restricted to $[0,\infty) \times [0,1] \subset \R$ or $(-\infty,0] \subset \R$
at their punctures. We focus one puncture from each of $\Sigma_\pm$.

We denote
\beastar
C_- & = & \Sigma_- \setminus( [0,\infty) \times [0,1])\\
C_+ & = & \Sigma_+ \setminus  ((-\infty,0] \times [0,1])).
\eeastar
By considering translated coordinates  $(\tau',t) = (\tau \pm 2K,t)$
respectively, we assume and decompose $\Sigma_\pm$ into
\beastar
\Sigma_- & = & C_- \cup  \setminus ([-2K, \infty) \times [0,1]) \\
\Sigma_+ & = & C_+ \cup \setminus ( (-\infty, 2K] \times [0,1])
\eeastar
respectively.

Here $C_\pm$ is not compact if they carry additional punctures
other than those we are concerning at. However,
without loss of generalities for the purpose of gluing construction, we may and will
assume $C_\pm$ are compact subsets of $\Sigma_i$ \emph{after cutting-off fixed neighborhoods
of other strip-like regions than the one of our current interest} in the rest of the part.

For each $K > 0$, we put
\be\label{eq:SigmaK}
\Sigma_K = C_- \cup ([-2K,2K ]\times [0,1]) \cup  C_+.
\ee

Let
$
u_\pm: (\dot \Sigma_i,\partial \dot \Sigma_i)  \to  (M, \vec R)
$
be a bordered contact instanton of finite energy with
$$
u_\pm(\pm \infty,t) = \gamma'(T t), \quad T = \int (\gamma')^* \lambda.
$$
\begin{cond}\label{cond:small-diam}
We assume that
\begin{equation}\label{form3535}
\sup_{(\tau,t) \in \pm(0,\infty)} \text{dist} (u_\pm(\tau,t), \gamma'(t)) \le \epsilon_0
\end{equation}
where $\epsilon_0$ is some small constant at least smaller than the injectivity radius of $(M,g)$
and  the exponential decay holds.
\end{cond}
The following exponential decay result can be proved by the three-interval method employed
in \cite{mundet-tian} an analogue of which is a well-known lemma in the study of pseudoholomorphic curves.
(We omit its proof referring the proof of \cite[Lemma B.1]{oh:book1}, for example, in the
context pseudoholomorphic curves in symplectic geometry.)

\begin{prop}\label{prop:expdecay-smallmap}
 Let $\gamma'$ be nondegenerate Reeb chord from $R_0$ to $R_1$. Then
there exists $K_0> 0$ and $\epsilon_0 > 0$ such that for any $K \geq K_0$, if
$w: [-K,K] \times [0,1] \to (M; R_0,R_1)$ is a contact instanton satisfying
$$
\sup_{\tau \in [-K,K]} d_{C^0}(w_\tau,\gamma') \leq \epsilon_0,
$$
then we have
\be\label{eq:expdecay-smallmap}
|dw(\tau,t)| \leq C e^{-\delta \text{\rm dist}(\tau,\del[-K-1, K+1])}
\ee
for $\tau \in [-K,K] \times [0,1]$, where $\delta> 0$, $C > 0$ are independent of
$K \geq K_0$.
\end{prop}

The following is the tangent space of the domain of the section $\Upsilon$ or
a completion of the domain is the linearized operator $D\Upsilon(w)$
of an element $w \in \CF(M,\vec R; \underline \gamma,\overline \gamma)$.

\begin{defn}\label{defn:function-space} Let $\vec R^\pm = (R_1^\pm, \ldots, R_{k_\pm}^\pm)$ be a
Legendrian link and fix a metric with respect to which the link is totally geodesic
with respect to the Levi-Civita connection.
\begin{enumerate}
\item{[Domain norm]}
For each $k \geq 2$, we define a Banach space
$$
W^{k,p}((\Sigma_\pm \partial \Sigma_\pm );w^*TM,w^*T\vec R^\pm;\underline \gamma,
\overline \gamma)
$$
that satisfies
\be\label{eq:domain-norm}
 \sum_{\ell= 0}^k \int_{\Sigma_\pm}  \left |\nabla^\ell  Y \right|^2 < \infty.
\ee
We define a norm $\|Y\|_{W_{k,p}(\Sigma_\pm)}$ by setting
the left hand side of the above to be $\|Y\|_{W_{k,p}(\Sigma_\pm}^p$.
\item{[Codomain norm]} Next we consider the codomain of $D\Upsilon(w)$, which is given by
\be\label{eq:codomain}
\CH^{(0,1)}_{k-1,p}(M,\lambda)|_w = \Omega^{(0,1)}_{k-1,p}(w^*\xi)
\oplus \Omega^2_{k-2,p}(\dot \Sigma).
\ee
(See Definition \ref{defn:CHCM01} for the definition of $\CH^{(0,1)}(M,\lambda)$.)
We equip it with the norm
\be\label{eq:codomain-norm}
\|(\eta, g R_\lambda) \|: = \|\eta\|_{W^{k-1,p}} + \|g\|_{W^{k-2,p}}
\ee
\end{enumerate}
\end{defn}

\begin{rem}\label{rem:function-spaces}
\begin{enumerate}
\item We remark that because  the linearization
$D\Upsilon(w)$ is an elliptic operator of mixed degrees of 1 and 2, we need to assume
$k \geq 2$ instead of $k\geq 1$ so that the image of $D\Upsilon_2(w)$ lies in
$L_{p,loc}$. This makes construction of a right inverse of the linearized
operator $D\Upsilon(w)$ and the relevant estimates are
 more nontrivial than the case of pseudoholomorphic curves.
We refer readers to Section \ref{sec:right-inverse}.
\item It is rather satisfying to see that
imposing the totally geodesic requirement for the choice of metric on the
boundary condition does not have any obstruction since a generic configuration of
Legendrian links do not have intersection between different connected components
unlike the case of Lagrangian ``links''  in symplectic geometry: Even generically
it is not totally worry-free in the latter case because of the presence of intersections of
different components.
\end{enumerate}
\end{rem}

For $Y \in W^{k,p} ((\Sigma_i,\partial \Sigma_\pm);u_\pm^*TX,u_\pm^*TL)$,
we define
\begin{equation}\label{normformula}
\|Y\|^p_{W^{k,p}\delta} = \sum_{\ell=0}^{k} \int_{\Sigma_\pm}
\vert \nabla^\ell Y\vert^p \text{\rm vol}_{\Sigma_i}.
\end{equation}
in \eqref{eq:codomain-norm}

We next define a  Sobolev norm for the sections on
$\Sigma_K$.
Let $Y$ be a section of $u^*TX$ with $Y|_{\del \Sigma_K} \in \Gamma(u^*TL)$ such that
$$
Y \in W^{2,p}((\Sigma_K,\partial \Sigma_K);u^*TX,u^*TL).
$$
By the Sobolev embedding $W^{2,p} \hookrightarrow C^0$, $Y$ is continuous and so
the boundary condition makes sense.
(For this, we only need $W^{1,p}$, not $W^{2,p}$. We need the latter because
the map $\Upsilon_2$ involves taking 2 derivatives.)

\section{Construction of approximate solutions}

We will utilize the following type of
the estimates for the needed approximate solution in the gluing  of two contact instantons,
whose proof occupies the whole of this section. We define
\bea\label{eq:F}
\CF(\gamma_-,\gamma_+) & = & \left\{u:\R \times [0,1] \to M \, \mid
u(\tau,0) \in R_0, \quad u(\tau,1) \in R_1,\right. \nonumber \\
&{}& \qquad \left. u(-\infty) = \gamma_-, \, u(\infty) = \gamma_+, \quad E(u) < \infty
\right\}
\eea
for a general pair $(R_0,R_1)$ either with $R_0 \cap R_1$ or with $R_0 = R_1$.
Then we have the inclusion
$$
\widetilde \CM(\gamma_-,\gamma_+) \subset \CF(\gamma_-,\gamma_+).
$$

\begin{prop}\label{prop:pregluing-estimate}
Let $\K_- \subset \CM(\psi(R),R;\gamma_-,\gamma'),\quad  \K_+ \subset\CM(\psi(R),R;\gamma',\gamma_+)$
be given compact subsets. Then there exists a sufficiently large $K_0 > 0$ and
 a map
$$
\PreG: \MM_0(\gamma_-,\gamma') \times  \MM_0(\gamma',\gamma_+)
\times (K_0, \infty) \to \FF(\gamma_-,\gamma_+)
$$
that satisfies the following properties:
\begin{enumerate}
\item It induces a smooth embedding of $\K_- \times \K_+ \times (K_0,\infty)$ into
$\FF(\gamma_-,\gamma_+;B)$.
\item Denote $w = \PreG(u_-,u_+,K)$. Then $w$ satisfies
\be\label{eq:closedness}
\|d(w^*\lambda \circ j)\|_{L^p} \leq C e^{-\delta K}
\ee
and the error estimate
\be\label{eq:errorestimate}
\|\delbar^\pi w\|_{W^{1,2}} \leq C e^{-\delta K}, \quad K \geq K_0
\ee
for some $C = C(\K_-,\K_+), \, \delta = \delta(\gamma')$.
\end{enumerate}
\end{prop}

For each given pair $u_- \in \CM(\gamma_-,\gamma')$ and $u_+ \in \CM(\gamma',\gamma_+)$,
we take their representatives that satisfy the following normalization condition.
(Recall that the moduli space $\CM(\cdot, \cdot)$ is the quotient $\widetilde \CM(\cdot,\cdot)/\R$.)

\begin{cond}[Normalization of representative]
\be\label{eq:normalization}
\int_{(-\infty,0] \times [0,1]} |d^\pi u_\pm|^2 = \int_{[0,\infty) \times [0,1]} |d^\pi u_\pm|^2
\ee
for each $u_\pm$.  Such a normalization was also used by
Floer \cite{floer-intersections,floer-unregularized} in the Lagrangian Floer homology context
which  we adopt for our gluing problem here.
\end{cond}

\begin{proof}[Proof of Proposition \ref{prop:pregluing-estimate}]
First we mention that
the exponential convergence result, Proposition \ref{prop:C0-expdecay}, guarantees that
for each $K \geq K_0$ with $K_0$ sufficiently large, we can write
\beastar
u_-(\tau,t) & = & \exp_{\gamma_-(Tt)}\xi^-(\tau,t) \quad \text{\rm for \, }\, \tau \geq K_0,\\
u_+(\tau,t) & = & \exp_{\gamma_+(Tt)}\xi^+(\tau,t) \quad \text{\rm for \, }\, \tau \leq - K_0
\eeastar
for some sections $\xi^\pm \in \Gamma(\gamma_-^*TM)$ for each $\pm$ defined on the relevant domain
respectively.
\begin{lem} There exists $K_0 > 0$, depending on $u_\pm$ and $C = C(\K_\pm,\ell) > 0$ such that
\be\label{eq:xi-decay}
|\nabla^\ell\xi_\pm(\tau,t)|  \leq C_\pm  e^{-\delta|\tau|}
\ee
for all $|\tau| \geq K_0$.
\end{lem}
\begin{proof} This is an immediate consequence of
the exponential decay results from Subsection \ref{subsec:exponential-convergence}.
\end{proof}

For each $K \in \R_+ = [0,\infty)$, we define a family of cut-off functions $\chi_K:\R \to
[0,1]$ so that for $K \geq 1$, they satisfy
\be
\chi_K = \begin{cases} 1 & \quad \mbox{for } \tau \geq K+1 \\
1 & \quad \mbox{for } \tau \leq K-1.
\end{cases}
\ee
We also require
\be
\chi_K'(\tau) \geq 0 \quad \mbox{on }\,  [K-1,K+1].
\ee
As usual, we take a suitable partition of unity  of $\R$ given by $\{\chi_{2K},1-\chi_{2K}\}$.
Then we define
\be\label{eq:app}
u_{\text{\rm app}} = u_- \#_{K} u_+
\ee
defined by
\beastar
&{}& u_- \#_{K} u_+(\tau,t):= \\
&{}& \begin{cases} u_-(\tau + 2K,t)  & \text{for }\, \tau \leq -K -1 \\
\exp_{\gamma'} \left((1-\chi(\tau)) \xi^-(\tau + 2K) + \chi(\tau)\xi^+(\tau - 2K, t)\right)
& \text{for }\, \tau \in [-K -1,K+1] \\
u_+(\tau - 2K, t) & \text{for }\, \tau \geq K+1
\end{cases}
\eeastar
For notational convenience, we write
\be\label{eq:chipm}
\chi_{2K}^+ = \chi_{2K}, \quad \chi_{2K}^- = 1 - \chi_{2K}
\ee
henceforth.

The following error estimate for the vertical component $\Upsilon_2(w) = d(w^*\lambda \circ j)$
of $\Upsilon(w)$ is one of the key estimates, \emph{whose counterpart in the pseudoholomorphic
curve theory in symplectization does not exist, or rather would be mixed up with the estimates of
the horizontal components.}

\begin{lem}\label{lem:dw*lambdaj} There exists some $K_0 > 0$ depending only
on $\K_\pm$ and $\gamma'$ such that
$$
\| d(w^*\lambda \circ j)\| \leq C e^{-\delta K}
$$
for all $(u_-,u_+) \in \K_- \times \K_+$ and $K \geq K_0$.
\end{lem}
\begin{proof}
We compute
$$
w^*\lambda \circ j= \lambda\left(\frac{\del w}{\del t}\right) \, d\tau
+ \lambda\left(\frac{\del w}{\del \tau}\right)\, dt
$$
and
$$
d(w^*\lambda \circ j) = - \left(\frac{\del}{\del t}\left(\lambda\left(\frac{\del w}{\del t}\right)\right)
+ \frac{\del}{\del \tau}\left(\lambda\left(\frac{\del w}{\del \tau}\right)\right)\right) \, d\tau \wedge dt.
$$
Estimating this term as it is will be rather involved largely due to complexity of
notation of covariant derivatives and taking the 2 derivatives. Because of that, we
make some preparation to simplify the matter.

We first recall the exponential $C^0$-convergence from Proposition \ref{prop:C0-expdecay}.
We rewrite $w = \exp \circ \widetilde w$ for $\widetilde w = w \circ \exp_{\gamma'} ^{-1}$ and
$$
w^*\lambda = (\widetilde w)^* \circ \exp_{\gamma'}^*\lambda
$$
and write $\widetilde \lambda: = \exp_{\gamma'}^*\lambda$. Then we have
$$
w^*\lambda = (\widetilde w)^*(\widetilde \lambda).
$$
Combining the smoothness of the exponential map $\exp_{\gamma'}$ and bootstrap arguments,
and the well-known
fact
$$
\|d_v\exp - id\| \leq C |v|
$$
and the $C^0$ exponential convergence of $w_\pm(\pm \tau, \cdot) \to \gamma'$,
for the purpose of doing the required estimates, we can ignore $\exp_{\gamma'}$ in the definition of $u_-\#_K u_+$ above
and work with the map
$$
(\tau,t) \mapsto \chi^-_{2K}(\tau)) \xi^-(\tau + 2K) + \chi^+_{2K}(\tau) \xi^+(\tau - 2K,t)) =: \widetilde w(\tau,t)
$$
as a section of the bundle $(\gamma')^*TM \to [0,1]$, and $\widetilde \lambda$.
This being said, we will drop the `tilde' from $\widetilde w$ and $\widetilde \lambda$
in the following calculations for the simplicity of notations.

Then we have
\beastar
\frac{\del w}{\del \tau} & =  &\chi_{2K}'(\tau)(\xi^+(\tau - 2K,t) - \xi^-(\tau + 2K) \\
 &{}& + (1-\chi_{2K}(\tau)) \left(\nabla_\tau \xi^-\right)
 + \chi_{2K} \tau) \left(\nabla_\tau \xi^+ \right).
\eeastar
Therefore
\beastar
\lambda\left(\frac{\del w}{\del \tau}\right) & = &  \chi_{2K}'(\tau)
 \lambda(\xi^+(\tau - 2K,t) - \xi^-(\tau + 2K)) \\
 &{}& \quad + (1-\chi(\tau)) \lambda\left(\nabla_\tau \xi^-\right)
 + \chi(\tau) \lambda\left(\nabla_\tau \xi^+ \right)
 \eeastar
 Differentiating one more time, we get
 \beastar
\frac{\del}{\del \tau}\left(\lambda\left(\frac{\del w}{\del \tau}\right)\right)
& = &  \chi''_{2K} (\tau)
 \lambda(\xi^+(\tau - 2K,t) - \xi^-(\tau + 2K)) \\
 &{}& \quad + 2\chi'_{2K}(\tau) \lambda\left(\nabla_\tau \xi^+ - \nabla_\tau \xi^-\right) \\
 &{}& \quad +  (1-\chi_{2K} (\tau))\frac{\del}{\del \tau}
 \left(\lambda\left(\nabla_\tau \xi^-\right)\right)
 + \chi_{2K} (\tau) \frac{\del}{\del \tau}\left(\lambda\left(\nabla_\tau \xi^+ \right)\right)
 \eeastar
 and
 $$
 \frac{\del}{\del t}\left(\lambda\left(\frac{\del w}{\del t}\right)\right)
 = (1-\chi_{2K} (\tau))\frac{\del}{\del t} \left(\lambda\left(\nabla_t \xi^-\right)\right)
 + \chi_{2K} (\tau) \frac{\del}{\del t} \left(\lambda\left(\nabla_t \xi^+ \right)\right).
 $$
 By adding the two, we obtain
 \beastar
 &{}& \frac{\del}{\del t}\left(\lambda\left(\frac{\del w}{\del t}\right)\right)
+ \frac{\del}{\del \tau}\left(\lambda\left(\frac{\del w}{\del \tau}\right)\right) (\tau,t)\\
& = &  \underbrace{\chi_{2K}''(\tau) \lambda(\xi^+(\tau - 2K,t) - \xi^-(\tau + 2K))}_{(I}) \\
 &{}& \quad + \underbrace{2\chi_{2K}'(\tau)
 \lambda\left(\nabla_\tau \xi^+(\tau - 2K,t) - \nabla_\tau \xi^-(\tau + 2K))\right) }_{(II)}\\
 &{}& \quad +  \underbrace{(1-\chi_{2K}(\tau))
 \left( \frac{\del}{\del \tau} \left(\lambda\left(\nabla_\tau \xi^-\right)\right)(\tau - 2K,t)
 + \frac{\del}{\del t} \left(\lambda\left(\nabla_t \xi^-\right)\right)(\tau - 2K)\right)}_{(III)} \\
 &{}& \quad + \underbrace{\chi_{2K}(\tau) \left(\frac{\del}{\del \tau}
 \left(\lambda\left(\nabla_\tau \xi^+ \right)\right)(\tau + 2K)
 + \frac{\del}{\del t} \left(\lambda\left(\nabla_t \xi^+ \right)\right)(\tau + 2K))\right)}_{(IV)}.
 \eeastar
 We estimate terms in the above 4 lines separately. Since $w_\pm$ are
 contact instantons and so satisfy $d(w_\pm^*\lambda \circ j) = 0$, we derive that
(III) and (IV)  and their derivatives are  zero modulo exponential errors: The error
can be estimated from the inequality
$$
\|d_v\exp - id\| \leq C |v|
$$
and the $C^0$ exponential convergence of $w_\pm(\pm \tau, \cdot) \to \gamma'$.

To estimate $(II)$, we utilize
$$
\left|\nabla_\tau \xi^\pm(\tau,t)\right|   \leq C e^{- \delta |\tau|}
$$
coming from Proposition \ref{prop:dw-expdecay}, and hence
$$
\left|\lambda \left(\nabla_\tau \xi^+(\tau - 2K,t) - \nabla_\tau \xi^-(\tau + 2K)\right)\right|
\leq C (e^{- \delta |2K - \tau|} + e^{- \delta |2K + \tau|} )
$$
for $-K \leq \tau \leq K$. But we have
$$
 e^{- \delta |2K - \tau|} + e^{- \delta |2K + \tau|} \leq 2 e^{-K+1}
 $$
 for the region $-K -1 \leq \tau \leq K +1$.

 Finally we estimate the term (I) using the exponential convergence
 $$
 w_+(\tau, \cdot), \, w_-(\tau, \cdot) \to \gamma'
 $$
 as $\tau \to \infty$ and $\tau \to -\infty$ respectively coming from
 Proposition \ref{prop:C0-expdecay}. This finishes the proof of the lemma.
 \end{proof}

Now we need to examine the estimate of
$$
\left \|\delbar^\pi \uapp\right \|_{W^{1,p}}
$$
which is easier than that of $d(w^*\lambda \circ j)$ because the former involves
just one derivative and we omit the details. This finishes the proof of Proposition \ref{prop:pregluing-estimate}.
\end{proof}

Now we can rephrase \eqref{eq:closedness} and \eqref{eq:errorestimate}
into the following exponential estimates of the full error
$\Upsilon(\uapp)$.

\begin{cor}\label{cor:Upsilon-decay}
$$
\|\Upsilon(\uapp)\|_{W^{1,2}}\leq C e^{-\delta K}
$$
for all $K \geq K_0$.
\end{cor}

\section{The linearized operator in isothermal coordinates}
\label{sec:linearization-in-coordinates}

The following equivalence theorem was utilized in \cite{oh-wang:CR-map2},
\cite{oh:contacton-Legendrian-bdy}.

\begin{prop}[Corollary 4.3, \cite{oh-wang:CR-map2}]
\label{prop:equivalence}
Consider a smooth map $w: (\dot \Sigma, \del \dot \Sigma) \to (M, \vec R)$.
Let $(x,y)$ be an isothermal coordinates of $(\dot \Sigma,j)$
at a point in $\del \dot \Sigma$, and define a complex-valued function
$$
\alpha = \lambda\left(J\frac{\del w}{\del x}\right)
+ \sqrt{-1} \lambda\left(\frac{\del w}{\del y}\right).
$$
We set $\zeta = \pi \frac{\del w}{\del \tau}$. Then
the following system of equations for the pair $(\zeta,\alpha)$
\be\label{eq:main-eq-isothermal}
\begin{cases}\frac12(\nabla_x^\pi \zeta + J \nabla_y^\pi \zeta)
+ \frac{1}{2} \lambda(J\zeta)(\CL_{R_\lambda}J)\zeta
- \frac{1}{2} \lambda(\zeta)(\CL_{R_\lambda}J)J\zeta =0\\
\zeta(z) \in T R_{i+1} \quad \text{for } \, z \in \overline{z_iz_{i+1}} \subset \del \dot \Sigma
\end{cases}
\ee
for some $i = 0, \ldots, k$, and
\be\label{eq:equation-for-alpha}
\delbar \alpha = |\zeta|^2
\ee
with boundary condition  \\
$$
w(z) \in R_{i+1} \quad \text{for } \, z \in \overline{z_iz_{i+1}} \subset \del \dot \Sigma.
$$
is equivalent to \eqref{eq:contacton-Legendrian-bdy-intro}.
\end{prop}
\begin{proof} The combined equation  of \eqref{eq:main-eq-isothermal} and
\eqref{eq:equation-for-alpha} is equivalent to the system of real equations
$$
d(w^*\lambda \circ j) = 0, \quad *d(w^*\lambda) =  |\zeta|^2.
$$
The second equation is derived in Lemma 11.19 \cite{oh-wang:CR-map2}
for any contact instantons, which follows from the identity
$$
g_J = d\lambda(\cdot, J \cdot)
$$
applied to the vector $\zeta = \pi \frac{\del w}{\del x}$ for any contact Cauchy-Riemann map.
This finishes the proof.
\end{proof}

\begin{rem}\label{rem:alternating-process}
\begin{enumerate}
\item The system of the above two equations
forms a nonlinear elliptic system for $(\zeta,\alpha)$ which are coupled.
This is the lifting of contact instantons to \emph{gauged contact instantons}
as defined in \cite[Definition 2.1]{oh:contacton}.
\item This system
is utilized in \cite{oh:contacton-Legendrian-bdy} to derive higher regularity
estimate in the following way: $\alpha$ is fed into
\eqref{eq:main-eq-isothermal} through its coefficients and then $\zeta$ provides the input
for the equation \eqref{eq:equation-for-alpha} and then back and forth. Using this structure of
coupling,  the higher derivative estimates was derived in \cite{oh:contacton-Legendrian-bdy}
by alternating boot strap arguments between $\zeta$ and $\alpha$.
\end{enumerate}
\end{rem}

We now evaluate each of these one-forms on
$[0,\infty) \times [0,1] \subset \dot \Sigma$ against $\frac{\del}{\del \tau}$.
We note that by definition \eqref{eq:B} we have
\beastar
B^{(0,1)}(Y^\pi) & = & \frac12(B(Y) + J B\circ j(Y^\pi)) \\
& = &  - \frac12  w^*\lambda \otimes \left((\CL_{R_\lambda}J)J Y^\pi\right)
+ \frac12 w^*\lambda \circ j \otimes J \left((\CL_{R_\lambda}J)J Y^\pi\right)\\
& =&  - \frac12  w^*\lambda \otimes \left((\CL_{R_\lambda}J)J Y^\pi\right)
+ \frac12 w^*\lambda \circ j \otimes \left((\CL_{R_\lambda}J Y^\pi \right).
\eeastar

By evaluating the zero-order part of \eqref{eq:DUpsilon1} against $\del_\tau$ respectively, we get
$$
B^{(0,1)}(\del_\tau)(Y^\pi) = - \frac12 \lambda \left(\frac{\del w}{\del \tau}\right) (\CL_{R_\lambda}J)J Y^\pi
+ \frac12 \lambda \left(\frac{\del w}{\del t}\right) (\CL_{R_\lambda}J) Y^\pi,
$$
and
$$
T_{dw}^{\pi, (0,1)}(\del_\tau)(Y^\pi) = \frac12 T^\pi\left(Y^\pi,\frac{\del w}{\del \tau}\right)
+ \frac12 J T^\pi\left(Y^\pi,\frac{\del w}{\del t}\right).
$$
We also write
$$
(\Upsilon_2(\del_\tau))^\C(\alpha): = \delbar \alpha -  \left|\frac{\del w}{\del \tau}\right|^2\,d\overline z,
\, z = \tau + \sqrt{-1} t.
$$

Then either by directly linearizing the above \eqref{eq:main-eq-isothermal}
or evaluating \eqref{eq:DUpsilon1}, \eqref{eq:DUpsilon2} we obtain the following
explicit form of the linearized operator.

\begin{lem} \label{lem:explicit-linearization}
Let $Y = \delta w$ and $\beta = \delta \alpha$ be the first variations of $w$ and $\alpha$.
Decompose $Y$ into $Y = Y^\pi + \lambda(Y) R_\lambda$ and
$\beta_Y = \lambda(JY) + \sqrt{-1} \lambda(Y)$.
Then in the strip-like coordinates on $[0,\infty) \times [0,1]$, we have
\bea
D\Upsilon_1^1(\del_\tau)(Y^\pi) & = & \frac12(\nabla_\tau Y^\pi + J \nabla_t Y^\pi)\nonumber\\
&{}&  - \frac12 \lambda \left(\frac{\del w}{\del \tau}\right) (\CL_{R_\lambda}J)J Y^\pi
+ \frac12 \lambda \left(\frac{\del w}{\del t}\right) (\CL_{R_\lambda}J) Y^\pi \nonumber\\
& {} &  + \frac12 T^\pi\left(Y^\pi,\frac{\del w}{\del \tau}\right)
+ \frac12 J T^\pi\left(Y^\pi,\frac{\del w}{\del t}\right) \label{eq:11-in-coordinates}\\
D(\Upsilon_2(\del_\tau))^\C(\beta) &= &  \delbar \beta_Y
- 2 \langle Y^\pi, \frac{\del w}{\del \tau} \rangle \, d\overline z.
\label{eq:22-in-coordinates}
\eea
\end{lem}
%
%\begin{rem}
%\emph{Here enters the requirement $d(w^*\lambda \circ j) = 0$ in the construction of
%approximate solutions:}  The above coordinate expression  of
%\eqref{eq:contacton-Legendrian-bdy-intro}, split into two equations \eqref{eq:main-eq-isothermal}
%and \eqref{eq:equation-for-alpha}, manifests that,  without the requirement
%$d(w^*\lambda \circ j) = 0$, construction of a right inverse of $D\Upsilon(w)$
%\emph{with the requirement given in Proposition \ref{prop:approQestimate} would be much
%harder}, if possible at all. This is because finding a right inverse involving the task of
%solving the linearized equation $D\Upsilon(w) = (\eta,f)$ for given
%$\eta \in \Omega^{(0,1)}_{k-1,p}$ and $f \in \Omega^0_{k-2,p}$,
%so that the (approximate) solution should satisfy the asymptotic exponential estimates.
%\end{rem}

\section{Gluing construction of a right inverse {$Q(\uapp)$}}
\label{sec:right-inverse}

In this section, we examine another step that involves new ingredients, construction of
an approximate inverse of the linearized operator $D\Upsilon(\uapp)$. This is the key step
that is distinctly different from that of other literature concerning the gluing problem
of pseudoholomorphic curves.

We first recall the expression of the linearization map
$$
D\Upsilon(w): \Omega^0_{k,p}(w^*TM,(\del w)^*{\vec R}; \vec \gamma) \to
\Omega^{(0,1)}_{k-1,p}(w^*\xi) \oplus \Omega^2_{k-2,p}(\dot \Sigma)
$$
given by the matrix
\be\label{eq:matrix}
\left(\begin{matrix}\delbar^{\nabla^\pi} + T_{dw}^{\pi,(0,1)} + B^{(0,1)}
 & \frac{1}{2} \lambda(\cdot) (\CL_{R_\lambda}J)J \del^\pi w \\
d\left((\cdot) \rfloor d\lambda) \circ j\right) & -\Delta(\lambda(\cdot)) \,dA
\end{matrix}
\right).
\ee
(See \cite[Section 11]{oh:contacton}. Also see \cite[Theorem 10.1]{oh-savelyev}
in the context of $\frak{lcs}$ instantons.)

Since the off-diagonal blocks are compact operators relative to the diagonal
 (Proposition \ref{prop:closed-fredholm} (1)), it is enough to first construct
a right inverse to the diagonal operator
\be\label{eq:diagonal-operator}
\left(\delbar^{\nabla^\pi} + T_{dw}^{\pi,(0,1)} + B^{(0,1)}\right) \oplus (-\Delta(\lambda(\cdot)) \,dA)
\ee
for $w = \PreG(u_-,u_+,R)$.
We note that $-\Delta(\lambda(\cdot)) \,dA = * \Delta(\lambda(\cdot))$.

Therefore we will construct a right inverse of the diagonal operator \eqref{eq:diagonal-operator}
which can be constructed by gluing  the right inverses $Q^\pi(u_\pm)$ of
$D\Upsilon_1^1(u_\pm)$  and $Q^\perp(u_\pm)$ whose details now follow.

\subsection{Construction of right inverses {$Q(u_\pm)$}}

As mentioned before, we will construct the approximate right inverse
$Q_{\text{\rm app}}(\uapp)$ in the diagonal form
$$
Q^\pi_{\text{\rm app}}(\uapp) \oplus Q_{\text{\rm app}} ^\perp(\uapp)
$$
for $\uapp = \PreG(u_-,u_+,R)$. We decompose any element
$$
Y \in \Omega^0_{k-2,p}(\uapp^*TM,(\del \uapp)^*{\vec R}; \vec \gamma)
$$
into
$$
Y = Y^\pi + \lambda(Y) R_\lambda.
$$
\begin{hypo}\label{hypo:surjectivity}
We first  take $J$ so that $D\Upsilon(u_\pm)$ are
surjective, as well as $D\Upsilon_1^1(u_\pm)$ and $D\Upsilon_2^2(u_\pm)$
respectively, so that all thereof have right inverses.
\end{hypo}

As the first step, we  take a right inverse of $D\Upsilon(u_\pm)$,
$$
Q(u_\pm): \Omega^{(0,1)}_{k-1,p}(u_\pm^*\xi) \oplus \Omega^2_{k-2,p}(\dot \Sigma)
\to \Omega_{k,p}(w^*\xi, (\del u_\pm)^*T\vec R) \oplus \Omega^2_{k,p}(\dot \Sigma, \del \dot \Sigma),
$$
respectively for each of $u_\pm$. It follow from the expression in Lemma \ref{lem:explicit-linearization} that
$$
D\Upsilon_1^1(u_\pm): \Omega^0_{k,p}(u_\pm^*\xi)
\to \Omega^{(0,1)}_{k-1,p}(u_\pm^*\xi)
$$
is a Fredholm operator. By the surjectivity hypothesis in Hypothesis \ref{hypo:surjectivity},
we can choose their right inverses $Q^\pi(u_\pm)$ so that
$$
D\Upsilon_1^1(u_\pm) Q^\pi(u_\pm)= id
$$
for $u_\pm$ respectively.

On the other hand we take the Green kernels $G_{u_\pm}$ for the right inverse of
$$
 \Delta= -* D\Upsilon_1^1(u_\pm)
$$
 and define
$$
Q^\perp(u_\pm)(\eta) = G_{u_\pm}(\lambda(\eta)) \cdot R_\lambda.
$$

Now we can construct a full right inverse $G(u_\pm)$ out of $G^\pi(u_\pm)$ and $G^\perp(u_\pm)$
as mentioned in Proposition \ref{prop:closed-fredholm} (1).
The right inverse is not unique and so we need to choose them smoothly
depending on
$
u_-\in \MM_0(\gamma_-,\gamma),\, u_+ \in \MM_0(\gamma',\gamma_+)
$
respectively so that
\be\label{eq:Q-norm}
\|Q(u_\pm))\| \leq C_\pm, \quad D\Upsilon(u_\pm)\circ Q(u_\pm) = \id
\ee
for some $C_\pm = C(\K_\pm)$ independent of $u_\pm \in \K_\pm$ respectively.

\subsection{Gluing of the right inverses $Q(u_\pm)$}
\label{subsec:gluingQ}

As before, we write
$$
\PreG(u_-,u_+,K) : =  \uapp
$$
and its covariant linearization by $ D\Upsilon(\uapp)$.
Similarly to the construction of approximate solution $\uapp$,
we explicitly construct an approximate right inverse of $D\Upsilon(w)$ following the standard
gluing procedure of an approximate right inverse.

The rest of this section will be occupied by the proof of the following proposition.
We would like to emphasize that this construction is one of the fundamental element
in our gluing theory of (bordered) contact instantons, which are markedly different from
the construction of similar right inverse in the gluing theory of pseudoholomorphic curves.

\begin{prop}\label{prop:approQestimate} Let $J$ be such that the
linearization operators $D\Upsilon(u_\pm)$ are surjective. Then there
exists $K_0 > 0$ such that for all $K \geq K_0$ and an operator
$Q_{\text{\rm app}}(\uapp)$ depending on
$$
(u_+,u_-,K) \in \K_+ \times \K_- \times [K_0,\infty)
$$
such that
\beastar
&{}& \|Q_{\text{\rm app}}(\uapp)\| \leq C,\\
&{}& \|D\Upsilon(u_{\uapp})\circ Q_{\text{\rm app}}(\uapp) -id\| \leq \frac{1}{2}.
\eeastar
In particular, $D\Upsilon (u_{\uapp})\circ Q_{\text{\rm app}}(\uapp)$ is
invertible uniformly over $\K_+ \times \K_- \times [K_0,\infty)$.
\end{prop}
We divide the proof into two parts, one is the definition of the approximate inverse
and the other is the proof of the required estimates.

\subsubsection{Definition of $Q_{\text{\rm app}}$}

Let $\eta \in L^p(\Lambda^{(0,1)}(\uapp^*TM))$. We define the families of the elements
in $\Omega^{(0,1)}_{1,p}(u_\pm^*TM)$ by
\beastar
\eta_{-,K}^\pi (\tau,t) & = &
\chi^+_{2K} \eta^\pi (\tau - 2K) \\
\eta_{+,K}^\pi(\tau,t) & = &
\chi^-_{2K} \eta^\pi(\tau + 2K)
\eeastar
for each $K \geq K_0$, respectively. We note
\beastar
\eta_{-,K}^\pi(\tau,t) & \in & \xi_{\uapp(\tau,t)} \subset T_{u_-(\tau-2K,t)}M,\\
 \eta_{+,K}^\pi(\tau,t) & \in & \xi_{\uapp(\tau,t)} \subset T_{u_+(\tau+2K,t)}M
\eeastar
by construction and so $\eta_{\pm,K}^\pi$ really define $L^{1,p}$-sections of
$\Lambda^{(0,1)}(u_\pm^* \xi)$ respectively.

By construction, we have
\bea\label{eq:suppetapmR}
\supp \eta_{-,K}^\pi & \subset & (-\infty, 2K] \times [0,1]\nonumber\\
\supp \eta_{+,K}^\pi & \subset & [-2K, \infty) \times [0,1].
\eea
Recall from the definition of $\uapp$ that
$$
\uapp ([-2K,2K] \times [0,1]) \subset U(\gamma_\pm).
$$
We denote by
$$
\Pi_{p}^q: T_pM \to T_qM
$$
the parallel transport along the short geodesic from $p$ to $q$
whenever $d(p,q) \leq \iota(g)$=injectivity radius of $g$. Obviously we
have $(\Pi_{p}^q)^{-1} = \Pi_q^p$ and by the property of triad connection,
$\Pi_{p}^q$ preserves the decomposition
$TM = \xi \oplus \R\langle R_\lambda \rangle$.

Denote
$$
(D\Upsilon(u_{\uapp}))^\pi: = \pi D\Upsilon(u_{\uapp})|_{W^{2,p}((\del \Sigma, \del \dot \Sigma),(\uapp^*TM,(\del\uapp)^*T\vec R))}
$$
where we use the decomposition
$$
TM = \xi \oplus \R\langle R_\lambda \rangle.
$$
We write
\be\label{eq:xi-decayR}
Y_{\pm,R} ^\pi =  Q_\pm^\pi (u_\pm)(\eta_{\pm,R}) \in W^{1,p}(u_\pm^*\xi).
\ee
We define the operator
$$
Q_{\text{\rm app}}^\pi (\uapp): L^p(u_{\text{\rm app}}^*\xi) \to W^{1,p}(u_{\text{\rm app}}^*\xi)
$$
by setting its value
$$
Y_{app}^\pi:= Q_{\text{\rm app}}^\pi(u_\pm,K)(\eta)
$$
to be
\be\label{eq:Qapp}
Y_{\text{\rm app}}^\pi(\tau,t)
= \begin{cases}Y_{-,K}^\pi (\tau+2K,t)  \hskip 2in \mbox{for } \, \tau \leq -K + 1\\
Y_{+,K}^\pi (\tau-2K,t) \hskip 2in \mbox{for } \, \tau \geq K - 1\\
\Pi_{\gamma'}^{\uapp(\tau,t)}\left(\chi_{2K}^-(\tau)
\Pi^{\gamma'}_{u_-(\tau,t)}(Y_{-,K}^\pi(\tau,t))
+ \chi_{2K}^+(\tau)\Pi^{\gamma'}_{u_+(\tau,t)}(Y_{+,K}^\pi (\tau,t))\right)\\
\hskip2.9in \mbox{for } \, \tau \in [-K +1, K -1].
\end{cases}
\ee

Next we determine the Reeb component as follows:
Let $G$  be the Green kernel of the operator of the $\Delta$, i.e., the value $G(f)$
is a solution satisfying
$$
\Delta(G(f)) = f
$$
for each given function $f \in L^p(\dot \Sigma, \del \dot \Sigma)$, whose
existence is a standard fact.
Then We define the operator $Q^\perp(\uapp)$
$$
Q^\perp(\uapp)(\eta): = G(\lambda(\eta)) \cdot R_\lambda
$$
In other words the function $g: = Q^\perp(\uapp)(\eta)$ solves the equation
\be\label{eq:*Deltag}
\begin{cases}
-\Delta g = \lambda(\eta) \quad  & \text{\rm on } \, \dot \Sigma \\
g = 0  \quad & \text{\rm on }\, \del \Sigma.
\end{cases}
\ee

\begin{defn}[The operator $Q_{\text{\rm app}}$]\label{defn:Q-app}
Let $Y_{\text{\rm app}}^\pi$ and $g$ be given as above. Then we define the operator
$Q_{\text{\rm app}}$ by setting its value to be
\be\label{eq:definition-Q}
Q_{\text{\rm app}}(\eta,g)  := Y_{\text{\rm app}}^\pi + g R_\lambda(\uapp).
\ee
\end{defn}

\subsubsection{Estimates for $Q_{\text{\rm app}}$}

Now we denote
$$
Y_{\text{\rm app}}:= Q_{\text{\rm app}}(\eta,g).
$$
To prove Proposition \ref{prop:approQestimate}, we need to establish
\be\label{eq:Qpi-norm}
\|Y_{\text{\rm app}}\|_{W^{2,p}} \leq C(\|\eta\|_{W^{1,p}} + \|g\|_{L^p})
\ee
and
\be\label{eq:DUpsilonQ-norm}
\|D\Upsilon_1(\uapp)(Y_{\text{\rm app}}^\pi) - \eta^\pi \| \leq \frac12
(\|\eta\|_{W^{1,p}} + \|g\|_{L^p}).
\ee
For the simplicity of notation, we will drop the subindex $\text{\rm app}$ and just write
$$
Y = Y_{\text{\rm app}}
$$
from now on in the rest of this section.

We recall from \eqref{eq:11-in-coordinates}, \eqref{eq:22-in-coordinates} that
the linearization operator   $D\Upsilon= D\Upsilon(\uapp)$ is decomposed into
\bea
D\Upsilon_1^1(\del_\tau)(Y^\pi) & = & \delbar^{\nabla^\pi}_\tau Y^\pi + B^{(0,1)}_\tau Y^\pi + T^{\pi,(0,1)}_{dw(\del_\tau)} Y^\pi
\label{eq:11-in-coordinates}\\
D\Upsilon_2^2(\del_\tau)(\beta) &= & \delbar \beta.
\label{eq:22-in-coordinates}
\eea

We start with the proof of \eqref{eq:DUpsilonQ-norm}.  Recalling \eqref{eq:Qapp}, we have
$$
Y^\pi(\tau,t)
= \Pi_{\gamma'(t)}^{\uapp(\tau,t)}\left(\chi_{2K}^-
\Pi^{\gamma'(t)}_{u_-(\tau,t)}(Y_{-,2K}^\pi(\tau,t))
+ \chi_{2K}^+ (\tau)\Pi^{\gamma'(t)}_{u_+(\tau,t)}(Y_{+,2K}^\pi (\tau,t))\right)
$$
for  $-2K \leq \tau \leq 2K$.  We observe that $Y_{\text{\rm app}}^\pi(\tau,t)
 \equiv 0$ on $[-2K +1, 2K-1]$.

Therefore we do estimates on the two regions $\pm [2K-1, \infty)$ separately.
Since the two regions are disjoint and the two cases are
similar, we will focus our calculations only on the region $[2K -1, \infty)$
and briefly comment the other at the end.

\medskip

{\bf On $[2K-1, \infty) \times [0,1]$:}
In this region, we have
$$
Y^\pi(\tau,t)
= \Pi_{\gamma'}^{\uapp(\tau,t)}\left(\chi_{2K}^+ (\tau)\Pi^{\gamma'}_{u_+(\tau,t)}
(Y_{+,2K}^\pi (\tau,t))\right).
$$
We express  the first order differential operator $D\Upsilon_1^1$ by
$$
D\Upsilon_1^1(\del_\tau) = \delbar^{\nabla^\pi}_{J_w} + \Gamma_w
$$
where we write $J_w : = J (w)$ and
$$
\delbar^{\nabla^\pi}_{J_w} Y^\pi = \frac12(\nabla^\pi_\tau + J_w \nabla^\pi_t)
$$
for the pull-back connection of $\nabla$ under the map $w$ which we
still denote by the same symbol $\nabla$, and $\Gamma_w$ the
zero-order operator
$$
\Gamma_w: = B^{(0,1)}_\tau + T^{\pi, (0,1)}_{dw(\del_\tau)}
$$
with
$$
B^{(0,1)}_\tau: = B^{(0,1)}(\del_\tau) .
$$
We also consider the translation
$I_K: \R \times [0,1] \to \R \times [0,1]$ given by
\be\label{eq:IK}
I_K(\tau,t): = (\tau - K,t).
\ee
Then we  denote by $w_K$ the translated map
$$
w_K = w\circ I_K.
$$

With these preparation, on $[0,\infty) \times [0,1]$, we compute
\beastar
D\Upsilon_1^1(\del_\tau) (Y^\pi_{\text{\rm app}})
& = &  (\delbar^{\nabla^\pi}_\tau+ \Gamma_w) Y^\pi_{\text{\rm app}} \\
& = &  (\delbar^{\nabla^\pi} _\tau+ \Gamma_w) \left(\chi_{2K}^+ (\tau)
\left(\Pi_{\gamma'}^{\uapp(\tau,t)}\left(\Pi^{\gamma'}_{u^+_{2K}}
 (Y_{+,2K}^\pi (\tau,t))\right)\right)\right) \\
 & = & \underbrace{ \chi_{2K}^+ (\tau) (\delbar^{\nabla^\pi}_
 \tau + \Gamma_{u^+_{2K}})
\left(\Pi_{\gamma'}^{u^+_{2K}}\left(\Pi^{\gamma'}_{u^+_{2K}}
 (Y_{+,2K}^\pi )\right)\right)}_{(I)}\\
 & {}& \quad   \underbrace{(\chi_{2K}^+)' (\tau)
\left(\Pi_{\gamma'}^{u^+_{2K}}\left(\Pi^{\gamma'}_{u^+{2K}}
 (Y_{+,2K}^\pi )\right)\right)}_{(II)}.
 \eeastar
Here for the last equality, we use the identity
$$
w(\tau,t) = \uapp(\tau,t) = u^+_{2K}(\tau,t)
$$
on $[2K, \infty) \times [0,1]$ and $(\delbar^{\nabla^\pi}_\tau \chi_{2K})(\tau,t) = \chi_{2K}'(\tau)$.
Therefore we have derived
\beastar
&{}& D\Upsilon_1^1(\del_\tau) (Y^\pi_{\text{\rm app}}) - \chi^+_{2K} \eta\\
& = & \underbrace{ \chi_{2K}^+ (\tau) \left((\delbar^{\nabla^\pi}_
 \tau + \Gamma_{u^+_{2K}})
\left(\Pi_{\gamma'}^{u^+_{2K}}\left(\Pi^{\gamma'}_{u^+_{2K}}
 (Y_{+,2K}^\pi )\right)\right) - \eta\right)}_{(I)}\\
 & {}& \quad   \underbrace{(\chi_{2K}^+)' (\tau)
\left(\Pi_{\gamma'}^{u^+_{2K}}\left(\Pi^{\gamma'}_{u^+_{2K}}
 (Y_{+,2K}^\pi )\right)\right)}_{(II)}.
\eeastar

For the term (I), we compute the commutator
$$
\left[\delbar^{\nabla^\pi}_\tau + \Gamma_{u^+_{2K}} ,
\Pi_{\gamma'}^{u^+_{2K}}\Pi^{\gamma'}_{u^+_{2K}}\right]
= 0
$$
since the triad connection $\nabla$ preserves $J$. Therefore we derive
\beastar
(I) & = & \chi_{2K}^+ (\tau)\left( (\delbar^{\nabla^\pi}_
 \tau + \Gamma_{u^+_{2K}})
\left(\Pi_{\gamma'}^{u^+_{2K}}\left(\Pi^{\gamma'}_{u^+_{2K}}
 (Y_{+,2K}^\pi) \right)\right) -\eta\right) \\
 & = & \chi_{2K}^+ (\tau)
\left(\Pi_{\gamma'}^{u^+_{2K}}\left(\Pi^{\gamma'}_{u^+_{2K}}(\delbar^{\nabla^\pi}_
 \tau + \Gamma_{u^+_{2K}}) (Y_{+,2K}^\pi)  - \eta\right)\right) \\
 &{}& + \chi_{2K}^+ (\tau)
\left(\Pi_{\gamma'}^{u^+_{2K}}\Pi^{\gamma'}_{u^+_{2K}}(\eta)  - \eta)\right)
 \eeastar
 On the other hand, we have
 $$
\| (\delbar^{\nabla^\pi}_
 \tau + \Gamma_{u^+_{2K}}) (Y_{+,2K}^\pi ) -\eta \|_{W^{2,p}} \leq C (\|\eta\|_{W^{1,p}} + \|g\|_{L^p})
 $$
 since $Y_{+,2K}^\pi  = (Q(u^+)(\eta,g))^\pi$ where $Q(u^+)$ is a right inverse of $D\Upsilon(u^+)$.
 Furthermore, it is easy to see
$$
\left|\chi_{2K}^+ (\tau)\left(
\Pi_{\gamma'}^{u^+_{2K}}\Pi^{\gamma'}_{u^+_{2K}}\eta\right)(\tau,t)  - \eta(\tau,t))\right| \leq C e^{-\delta K}|\eta(\tau,t)|
$$
because of the inequality
$$
\|\Pi_{\gamma'}^{u^+_{2K}}\Pi^{\gamma'}_{u^+_{2K}} -id\|_{C^0} \leq C e^{-\delta K}
$$
which in turn follows from $d_{C^0}(u^+_{2K}(\tau,\cdot) - \gamma') \leq C e^{-\delta K}$.

This proves
\be\label{eq:estimate-(I)}
\|(I)\|_{W^{2,p}} \leq  C (\|\eta\|_{W^{1,p}} + \|g\|_{L^p}).
\ee

Next we estimate $\|(II)\|_{W^{2,p}}$. We differentiate (II) twice and get
\beastar
\nabla^\pi (II) & = & (\chi_{2K}^+)'' (\tau)
\left(\Pi_{\gamma'}^{u^+_{2K}}\left(\Pi^{\gamma'}_{u^+_{2K}}
 (Y_{+,2K}^\pi )\right)\right)\, d\tau\\
& {}& \quad  + (\chi_{2K})'(\tau)
 \left(\Pi_{\gamma'}^{u^+_{2K}}\left(\Pi^{\gamma'}_{u^+_{2K}}
\nabla (Y_{+,2K}^\pi )\right)\right),
\eeastar
and
\beastar
\nabla^\pi \nabla^\pi (II) &= &(\chi_{2K}^+)''' (\tau)
\left(\Pi_{\gamma'}^{u^+_{2K}}\left(\Pi^{\gamma'}_{u^+_{2K}}
 (Y_{+,2K}^\pi )\right)\right)\, d\tau \otimes d\tau \\
& {}& \quad + (\chi_{2K})'(\tau)
 \left(\Pi_{\gamma'}^{u^+_{2K}}\left(\Pi^{\gamma'}_{u^+_{2K}}
\nabla^\pi\nabla^\pi(Y_{+,2K}^\pi )\right)\right)\\
& {}& \quad + 2(\chi_{2K}^+)''(\tau) \left(\Pi_{\gamma'}^{u^+_{2K}}\left(\Pi^{\gamma'}_{u^+_{2K}}
\nabla^\pi(Y_{+,2K}^\pi )\right)\right) \otimes d\tau.
\eeastar
Recalling that both $(\chi_{2K}^+)'''$ and $(\chi_{2K}^+)''$ are supported on
$[2K-1,2K]$ and
$$
Y_{+,2K}^\pi(\tau,t) = Y^\pi(\tau - 2K,t),
$$
we derive
\beastar
\|\nabla^\pi \nabla^\pi (II)\|_{L^p}^p &\leq &
\int_{2K-1}\chi^+_{2K} \int_0^1 \left( |Y_{+}^\pi(\tau - 2K,t)|^p\right. \\
&{}& \quad + \left. |\nabla^\pi Y_{+}^\pi(\tau - 2K,t)|^p
+ |(\nabla^\pi)^2Y_{+}^\pi(\tau - 2K,t)|^p\right) \\
& \leq & C^p \|Y^{+,\pi}\|_{W^{2,p}}^p.
\eeastar
This proves
\bea\label{eq:estimate-nabla(I)}
\|\nabla^\pi (II)\|_{L^p} \leq C \|Y^{+,\pi}\|_{W^{1,p}} \leq C \|Q(u^+)\| (\|\eta\|_{W^{1,p}} + \|g\|_{L^p} )\nonumber\\
\|\nabla^\pi \nabla^\pi (II)\|_{L^p} \leq C \|Y^{+,\pi}\|_{W^{2,p}} \leq C \|Q(u^+)\| (\|\eta\|_{W^{1,p}} + \|g\|_{L^p} ).
\eea
By adding up \eqref{eq:estimate-(I)} and \eqref{eq:estimate-nabla(I)}, we obtain
$$
\|D\Upsilon_1^1(\del_\tau) (Y^\pi_{\text{\rm app}}) - (\eta, g)\|_{[0,\infty) \times [0,1]}
 \leq C(\|\eta\|_{W^{1,p}} + \|g\|_{L^p}).
$$

\medskip

{\bf On the region $(-\infty,0] \times [0,1]$:}
By aplying the same calculation performed on $[0, \infty) \times [0,1]$ above
to the region $(-\infty,0] \times [0,1]$ and combining both, we obtain
$$
\|D\Upsilon_1^1(\del_\tau) (Y^\pi_{\text{\rm app}}) - (\eta, g)\|
 \leq C(\|\eta\|_{W^{1,p}} + \|g\|_{L^p}).
$$

Proving boundedness $\|Q_{\text{\rm app}}(\uapp)\| \leq C$
 is similar and easier and so omitted.
This finishes the proof of Proposition \ref{prop:approQestimate}.
\qed

Once Proposition \ref{prop:approQestimate} in our disposal,  we can construct a genuine right inverse by
$$
Q(\uapp) = Q_{\text{\rm app}}(\uapp)\circ (d\Upsilon \circ Q_{\text{\rm app}}(\uapp))^{-1}
$$
that satisfies \eqref{eq:Q-norm} and varies smoothly on the gluing parameters
$(u_+,u_-,K) \in \K+ \times \K_- \times [K_0,\infty)$.

Finally we apply the scheme of reducing the problem of finding a genuine solution
near to the approximate solution $\uapp$ to a fixed point problem which completes
the gluing problem of contact instantons such that the assignment
$$
(u_+,u_-,K) \mapsto \uapp: = u_+\widehat{\#}_K u_-
$$
defines a smooth embedding.

\part{Construction of contact instanton Legendrian Floer homology}
\label{part:homology}

In this part, we apply our gluing result combined with the compactification
result from \cite{oh:entanglement1}, and  give the construction of Legendrian
contact instanton  homology, denoted by $HI^*(\psi(R),R)$, of the pair $(\psi(R),R)$ for
any contactomorphism $\psi = \psi_H^1$ satisfying
$$
\|H\| < T_\lambda(M,R)
$$
that is required in Theorem \ref{thm:definition-intro}, leaning the full study in the DGA level
to \cite{oh:entanglement2} in general.  We will do this by first identifying
the chain group relevant to the problem of translated points of Sandon \cite{sandon-translated},
and then constructing the boundary map, the chain map and the chain homotopy map
on the group in the current context of Legendrian contact instanton  Floer homology.
The present part presumes the readers'
knowledge of the contents of \cite{oh:entanglement1} in addition to the standard Floer theory in
symplectic geometry given in e.g., \cite{floer-intersections}, \cite{hofer-salamon}, \cite[Section 3]{fooo:ham-Kuranishi},
especially the general construction of the boundary map, the chain map and the chain homotopy maps
in the symplectic Floer theory. This general scheme has been made into the abstract formalism of
\emph{linear K-systems} in \cite[Section 16]{fooo:book-kuranishi}, which can be also applied to the current
Legendrian contact instanton homology. Since we do not study orientations of the moduli spaces,
\emph{we will restrict ourselves to homology with $\Z_2$-coefficients without mentioning it in the rest of
the part leaving the study of orientation elsewhere.}

The following is a standard definition.

\begin{defn}\label{defn:spectrum} Let $\lambda$ be a contact form of contact manifold $(M,\xi)$ and $R \subset M$ a
connected Legendrian submanifold.
Denote by $\frak{Reeb}(M,\lambda)$ (resp. $\frak{Reeb}(M,R;\lambda)$) the set of closed Reeb orbits
(resp. the set of self Reeb chords of $R$).
\begin{enumerate}
\item
We define $\operatorname{Spec}(M,\lambda)$ to be the set
$$
\operatorname{Spec}(M,\lambda) = \left\{\int_\gamma \lambda \mid \lambda \in \frak Reeb(M,\lambda)\right\}
$$
and call the \emph{action spectrum} of $(M,\lambda)$.
\item We define the \emph{period gap} to be the constant given by
$$
T(M,\lambda): = \inf\left\{\int_\gamma \lambda \mid \lambda \in \frak{Reeb}(M,\lambda)\right\} > 0.
$$
\end{enumerate}
We define $\operatorname{Spec}(M,R;\lambda)$ and the associated $T(M,\lambda;R)$ similarly using the set
$\frak{Reeb}(M,R;\lambda)$ of Reeb chords of $R$.
\end{defn}
We set $T(M,\lambda) = \infty$ (resp. $T(M,\lambda;R) = \infty$) if there is no closed Reeb orbit (resp. no $(R_0,R_1)$-Reeb chord).
Then we define
\be\label{eq:TMR}
T_\lambda(M;R): = \min\{T(M,\lambda), T(M,\lambda;R)\} > 0
\ee
and call it the \emph{(chord) period gap} of $R$ in $M$.

\section{Definition of chain groups}
\label{sec:chain-groups}

Let $(M,\xi)$ be a contact manifold and $(R_0,R_1)$ be a pair of Legendrian submanifolds.

Following \cite{oh:entanglement1}, we consider contact triads $(M,\lambda,J)$ and  the associated boundary value problem
for $(\gamma, T)$ with $\gamma:[0,1] \to M$ and $T \in \R$.

\begin{defn}\label{defn:generators} We denote by
$$
\frak{X}(R_0,R_1)
$$
the set of nonnegative iso-speed Reeb chords, i.e., the set of pairs $(\gamma, T)$
satisfying
\be\label{eq:chord-equation}
\begin{cases}
\dot \gamma(t) = T R_\lambda(\gamma(t)),\\
\gamma(0) \in R_0, \quad \gamma(1) \in R_1.
\end{cases}
\ee
\end{defn}

It is useful for us to recall the following notion of
the \emph{Reeb trace} introduced in \cite{oh:entanglement1}.

\begin{defn}[Reeb trace] Let $R$ be a Legendrian submanifold of $(M,\lambda)$.
The \emph{Reeb trace} denoted by $Z_R$ is the union
$$
Z_R: = \bigcup_{t \in \R} \phi_{R_\lambda}^t(R).
$$
\end{defn}

We now recall the following nondegeneracy definition from \cite[Definition 10.1]{oh:entanglement1}.

\begin{defn}\label{defn:nondegeneracy-chords} We say a Reeb chord $(\gamma, T)$
of $(R_0,R_1)$ is nondegenerate if
the linearization map $\Psi_\gamma = d\phi^T(p): \xi_p \to \xi_p$ satisfies
$$
\Psi_\gamma(T_{\gamma(0)} R_0) \pitchfork T_{\gamma(1)} R_1  \quad \text{\rm in }  \,  \xi_{\gamma(1)}
$$
or equivalently
$$
\Psi_\gamma(T_{\gamma(0)} R_0) \pitchfork T_{\gamma(1)} Z_{R_1} \quad \text{\rm in} \, T_{\gamma(1)}M.
$$
\end{defn}

We denote by
$$
{\mathcal Leg}(M,\xi)
$$
the set of Legendrian submanifold and by ${\mathcal Leg}(M,\xi;R)$ its connected component
containing $R \in {\mathcal Leg}(M,\xi)$, i.e, the set of Legendrian submanifolds Legendrian isotopic to
$R$. We denote by
$$
\CP({\mathcal Leg}(M,\xi))
$$
the monoid of Legendrian isotopies $[0,1] \to {\mathcal Leg}(M,\xi)$. We have
natural evaluation maps
$$
\ev_0, \, \ev_1:\CP({\mathcal Leg}(M,\xi)) \to {\mathcal Leg}(M,\xi)
$$
and denote by
$$
\CP({\mathcal Leg}(M,\xi), R) = \ev_0^{-1}(R) \subset \CP({\mathcal Leg}(M,\xi))
$$
and
$$
\CP({\mathcal Leg}(M,\xi), (R_0,R_1)) = (\ev_0\times \ev_1)^{-1}(R_0,R_1) \subset \CP({\mathcal Leg}(M,\xi)).
$$

The following relative version of Theorem \ref{thm:ABW}  is proved in \cite[Appendix B]{oh:contacton-transversality}.

\begin{thm}[Theorem B.3 \cite{oh:contacton-transversality}]\label{thm:Reeb-chords-lambda-text}
Let $(M,\xi)$ be a contact manifold. Let  $(R_0,R_1)$ be a pair of Legendrian submanifolds
allowing the case $R_0 = R_1$.  There
exists a residual subset $\operatorname{Cont}^{\text{\rm reg}}_1(M,\xi) \subset \CC(M,\xi)$
such that for any $\lambda \in \operatorname{Cont}^{\text{\rm reg}}_1(M,\xi)$ all
Reeb chords from $R_0$ to $R_1$ are nondegenerate for $T > 0$, and
Bott-Morse nondegenerate when $T = 0$.
\end{thm}

Denote by $\LL(M;R_0,R_1)$ the space of paths
$$
\gamma: ([0,1], \{0,1\}) \to (M;R_0,R_1).
$$
We consider the assignment
\be\label{eq:Phi}
\Phi: (T,\gamma,\lambda) \mapsto \dot \gamma - T \,R_\lambda(\gamma)
\ee
as a section of the Banach vector bundle over
$$
(0,\infty) \times \CL^{1,2}(M;R_0,R_1) \times \Cont(M,\xi)
$$
where $\CL^{1,2}(M;R_0,R_1)$
is the $W^{1,2}$-completion of $\CL(M;R_0,R_1)$. We have
$$
\dot \gamma - T\, R_\lambda(\gamma)
\in \Gamma(\gamma^*TM; T_{\gamma(0)}R_0, T_{\gamma(1)}R_1).
$$
We  define the vector bundle
$$
\CL^2(R_0,R_1) \to (0,\infty) \times \CL^{1,2}(M;R_0,R_1) \times \Cont(M,\xi)
$$
whose fiber at $(T,\gamma,\lambda)$ is $L^2(\gamma^*TM)$. We denote by
$\pi_i$, $i=1,\, 2, \, 3$ the corresponding projections as before.

We denote $\frak{Reeb}(M,\lambda;R_0,R_1) = \Phi_\lambda^{-1}(0)$,
where
$$
\Phi_\lambda: = \Phi|_{ (0,\infty) \times \CL^{1,2}(M;R_0,R_1) \times \{\lambda\}}.
$$
Then we have
$$
{\frak Reeb}(\lambda;R_0,R_1) =  \Phi_\lambda^{-1}(0) = \frak{Reeb}(M,\xi) \cap \pi_3^{-1}(\lambda).
$$
The proof of Theorem \ref{thm:Reeb-chords-lambda-text}
relies on the following proposition which is also
 proved in \cite[Appendix B]{oh:contacton-transversality}.

\begin{prop} [Proposition B.4 \cite{oh:contacton-transversality}]
\label{prop:nondeneracy-chords-text}
 Suppose $R_0 \cap R_1 = \emptyset$.
A Reeb chord $(\gamma, T)$ of $(R_0,R_1)$ is nondegenerate if and only if
the linearization
$$
d_{(\gamma, T)}\Phi: \R \times W^{1,2}(\gamma^*TM;T_{\gamma(0)}R_0,T_{\gamma(1)}R_1)
\to L^2(\gamma^*TM)
$$
is surjective.
\end{prop}

With these preparations, we are now ready to give the definition of chain group
$$
CI_\lambda^*(R_0, R_1)
$$
by considering two cases, $\psi(R_0) \cap R_1 = \emptyset$ and $\psi(R_0) = R_0$.
We recall that for a generic choice of contactomorphism $\psi$, we have
$$
\psi(R_0) \cap R_0 = \emptyset
$$
by the dimensional reason.
(Here $CI_\lambda$ stands for \emph{contact instanton for $\lambda$} as well as the letter $C$ also
stands for \emph{complex} at the same time.)

For the case with $\psi(R_0) \cap R_0 = \emptyset$, the definition is simple. In this case,
we have a canonical one-to-one correspondence
$$
\frak{X}(\psi(R), R) \Longleftrightarrow \frak{Reeb}(\psi(R),R).
$$
Under the nondegeneracy assumption in the sense of Definition \ref{defn:nondegeneracy-chords}
we consider the $\Z_2$-vector space
$$
CI_\lambda^*(\psi(R), R) : = \Z_2\langle\frak{X}(\psi(R), R) \rangle.
$$
Here $CI_\lambda$ stands for \emph{contact instanton for $\lambda$} as well as the letter $C$ also
stands for \emph{complex} at the same time.
Note that the group is naturally \emph{positively} filtered by the action of Reeb orbits
$$
\CA(\gamma) = \int \gamma^*\lambda > 0
$$
for a generator $\gamma \in \frak{Reeb}(\psi(R,R)$. When $R$ is compact and
$M$ is tame, we have the period gap $T(M,\lambda;\psi(R),R) > 0$.

\subsection{The Morse-Bott nondegenerate case $R_0 = R_1 = R$}

Now we focus on the case with $\psi(R_0) = R_0 = R$ e.g., $\psi =id$.
This is the Morse-Bott nondegenerate case.

We first consider the Moore path space as the set
$$
\Theta(R;M): = \left\{ (\gamma,T) \, \Big\vert\, \gamma: [0,1] \to M,\,  \gamma(0), \, \gamma(1) \in R,
\quad \int \gamma^*\lambda = T \right \}.
$$
Recall in our definition that the set
$\frak{X}(R,R): = \frak{X}(M,R;\lambda)$ consists of iso-speed Reeb chords, i.e., the
set of pairs  $(\gamma, T)$  satisfying
$$
\dot \gamma = T R(\lambda)
$$
with $T \in \R$ and $\gamma$. Then we have the decomposition
$$
\frak{X}(R,R) = \frak{X}_{\{T=0\}}(R,R)  \sqcup \frak{X}_{\{T > 0\}}(R,R) \subset \Theta(R;M)
 $$
and $\frak{X}_{\{T=0\}}(R,R)$ is in one-to-one correspondence by the evaluation map
$$
(\gamma,0) \mapsto \gamma(0).
$$
The following definition is the open string analogue of \cite[Definition 1.7]{bourgeois} and \cite[Definition 1.2]{oh-wang:CR-map2}.

\begin{defn}\label{defn:morse-bott}
Let $\lambda$ be a contact form and consider $\frak{X}(R,R)$. We say a
connected component of $\frak{X}(R,R)$ is a Morse-Bott component
in $\Theta(R;M)$ if the following holds:
\begin{enumerate}
\item The tangent space at
every pair $(T,z) \in \frak{X}(R,R)$ therein coincides with $\ker d_{(\gamma,T)}\Phi_\lambda$.
\item The locus $Q \subset M$ of Reeb chords is embedded.
\item The 2-form $d\lambda|_Q$ associated to the locus $Q$ of Reeb chords has constant rank.
\end{enumerate}
\end{defn}

For the simplicity of notation, we denote
$$
\widetilde R: = \frak{X}_{T=0}(R,R)
\cong \{ \gamma:[0,1] \to M \mid \gamma(t) \equiv q \in R\}.
$$
The following general proposition seems to be of interest on its own.

\begin{prop}\label{prop:Bott-clean} Let $(M,\xi)$ be a contact manifold, and
let $R$ be a compact Legendrian submanifold.
Consider Moore chords $(\gamma, T)$ of $\widetilde R$ i.e., pairs
$\gamma:[0,1] \to M$ with
\be\label{eq:Moore-path-action}
\gamma(0),\, \gamma (1) \in R, \quad \int \gamma^*\lambda = T \geq 0.
\ee
Then for any choice of contact form $\lambda$ of $(M,\xi)$,
the set $\widetilde R$ is a Morse-Bott component of $\frak{X}(R,R)$ in $\Theta(R;M)$.
\end{prop}
\begin{proof} Clearly the locus of $\widetilde R$ in $M$ is $R$ which is embedded and
$d\lambda$ has constant rank 0 which verifies Property (2) and (3) in Definition \ref{defn:morse-bott}.
It remains to show Property (1).

Let $\gamma_q: [0,T] \to M$ with $q \in R$ be a constant Reeb chord
valued at $q \in R \subset M$.

We consider the assignment
\be\label{eq:Phi-R}
\Phi_\lambda: (\gamma,T) \mapsto \dot \gamma - T R_\lambda (\gamma)
\ee
as a section of the Banach vector bundle over
$$
[0,\infty) \times \CL^{1,2}(R;M)
$$
where $\CL^{1,2}(R;M)$ is the $W^{1,2}$-completion of $\Theta(R;M)$.
Clearly $\widetilde R \cong R$ is a smooth manifold, and its tangent space $(\gamma_q,0)$
is isomorphic
$$
T_q R \oplus \{0\}
$$
whose dimension is $\dim R$.
Now we compute $\ker d_{(\gamma_q,0)}\Phi_\lambda$. (For a future purpose, we also compute
$\coker d_{(\gamma_q,0)}\Phi_\lambda$.)

\begin{lem}
The linearized operator
$$
d_{(\gamma_q,0)}\Phi_\lambda: T_{(\gamma_q,0)}\Theta(R;M) \to \CL^{1,2}(R;M)
$$
has
\be\label{eq:ker-coker}
\ker d_{(\gamma_q,0)}\Phi_\lambda \cong T_q R, \quad \coker d_{(\gamma_q,0)}\Phi_\lambda \cong
(T_q R \oplus \R\langle R_\lambda(q)\rangle)^\perp.
\ee
In particular $\text{\rm Index } d_{(\gamma_q,0)} \Phi_\lambda = 0$.
\end{lem}
\begin{proof}
An element of $\ker d_{(\gamma_q,0)}\Phi_\lambda$ is characterized by a pair $(\xi,b)$
with $b \in \R$ and $\xi \in L^{1,2}(\gamma_q^*TM)$ satisfying the boundary value problem
\be\label{eq:ker-(xi,b)}
\begin{cases}
\frac{d \xi}{dt} - b R_\lambda(q) = 0 \\
\xi(0), \, \xi(1) \in T_q R.
\end{cases}
\ee
From this it is straightforward to check $b = 0$ and
$$
\ker d_{(\gamma_q,0)}\Phi_\lambda = \{(\xi, 0) \mid \xi(t) \equiv v \in T_qR\} \cong T_q R.
$$
This precisely the same as $T_{(\gamma_q,0)}\widetilde R$ which satisfies the property (1) of
Definition \ref{defn:morse-bott}.

Now we examine the cokernel of $d_{(\gamma_q,0)}\Phi_\lambda$. By considering the $L^2$-adjoint operator
$$
(d_{(\gamma_q,0)}\Phi_\lambda)^\dagger,
$$
and by integration by parts of the equation for $\eta$
\be\label{eq:L2-adjoint}
\int_0^1 \left\langle \frac{d\xi}{dt} (t)- b R_\lambda(q), \eta(t) \right \rangle \, dt = 0
\ee
for all $(\xi,b)$  satisfying  $\xi(0), \, \xi(1) \in T_q R$,
we derive that  an $L^2$-cokernel element $\eta \in L^2(\gamma_q^*TM)$ is
characterized by the equation
\be\label{eq:coker-eta}
\begin{cases}
\frac{d \eta}{dt} (t) = 0, \quad \left(\int_0^1 \eta(t)\, dt\right) \perp R_\lambda(q), \\
\eta(0), \, \eta(1) \perp T_q R.
\end{cases}
\ee
It is immediate to check that the only solution $\eta$ satisfying \eqref{eq:coker-eta}
is $\eta(t) \equiv v \in (T_q R \oplus \R\langle R_\lambda(q)\rangle)^\perp$.
We note that
$$
\dim (T_q R \oplus \R\langle R_\lambda(q)\rangle)^\perp = \dim  T_q R = \dim R.
$$
Therefore $\text{\rm Index}  d_{(\gamma_q,0)}\Phi_\lambda = 0$.
\end{proof}

This finishes the proof of the proposition.
\end{proof}

We apply this Morse-Bott property and define the Morse-Bott chain group to be
\be\label{eq:chaingroup-0}
CI_\lambda^*(R;M): = C^*(R) \oplus Z_2 \langle\frak{X}_{T > 0} (R;M)\rangle
\ee
where we take a suitable cochain complex $C^*(R)$ of $R$, e.g., the de Rham complex of
$R$ or the Morse complex of a Morse function $f$ on $R$. For the simplicity, we take
$$
C^*(R) = CM^*(f;R)
$$
here for a $C^2$-small Morse function $f$, in particular with
$$
\|f\|_{C^0} < T_\lambda(M;R).
$$
Here the group $CM^k(M;R)$ is the free abelian group generated by
the set $\text{Crit}(f)$ of critical points of $f$ of Morse index $k$.

\subsection{Description of short Reeb chords}

In this subsection, we describe the structure of group $CM^*(f;R)$ explicitly in terms of
a Darboux-Weinstein chart of the Legendrian submanifold $R \subset M$.
We adapt the exposition given in \cite[Section 2]{BKO:formality}
in the setting of Morse-Bott Lagrangian Floer theory similar to the current situation.

By taking a (Legendrian) Darboux-Weinstein chart on a neighborhood $U \supset R$,
we may assume the pair $(R, U)$  is (strictly) contactomorphic $(V, 0_{J^1R})$ for a
neighborhood $V \supset 0_{J^1 R}$ of the one-jet bundle
$$
(J^1R, \lambda_0), \quad \lambda_0 = dz - \pi_{J^1,T^*}^*\theta
$$
where $\pi_{J^1,T^*}: J^1R \to T^*R$ is the projection under the identification
$$
J^1R = T^*R \times \R
$$
and $\theta = pdq$ is the Liouville one-form on $T^*R$.

Let $f$ be a Morse function on
$R$ such that $0_{T^*R} \pitchfork \text{\rm Image } df$.
We regard $R_f: = \text{\rm Image } df$ as a subset of $J^1R$ or of $M$
depending on the circumstances through the identification via the Darboux-Weinstein chart.

The following lemma is an immediate
consequence of the implicit function theorem.

\begin{lem}\label{lem:perturbed-cord}
There exists some $0 < \epsilon_0 < \frac{T_\lambda(M,R)}2$ such that
$$
\frak X_{\leq \epsilon_0}(R_f \cap R) \cong \text{\rm Image } df \cap 0_{J^1 R}
$$
where $R_f= \text{\rm Image} df$, and so we have the decomposition
$$
CI^*(R_f,R;M) = CI^*_{< \epsilon_0}(R_f,R;M) \oplus CI^*_{> T_\lambda(M;R) - \epsilon_0}(R_f,R;M)
$$
provided $\|f\|_{C^2} \leq  \epsilon_0$.
\end{lem}
\begin{proof}  This is an immediate consequence of  $C^2$-smallness of $df$
the transversality requirement $0_{T^*R} \pitchfork \text{\rm Image } df$.
\end{proof}

By the generic transversality theorem, Theorem \ref{thm:Reeb-chords}
under the perturbation of Legendrian boundary from \cite[Appendix B]{oh:contacton-transversality},
we can choose $f$ so that
$\frak{X}(R_f,R)$ is nondegenerate and the statement of the above lemma holds.
We have one-to-one correspondence
$$
\frak{X}_{\leq \epsilon_0}(R_f,R) \cong \text{\rm Crit } f
$$
and hence
$$
\Z \langle \frak{X}_{\leq \epsilon_0}(R_f,R) \rangle \cong
\Z \langle \text{\rm Crit } f \rangle.
$$
Moreover we now show that $\Z\langle\frak{X}_{\leq \epsilon_0}(R_f,R)\rangle$ is a subcomplex
of $CI^*(R_f, R;M)$ which is isomorphic to
the Morse complex $(C^*(f),d)$ of the Morse function $f : R \to \R$.
We write by $\gamma_p$ the short Reeb chord associated to the point $p \in \Crit f$
and by $d$ the differential of the Morse complex $CM^*(f)$.

On the other hand we write $\frak m^1$ for the full complex
$CI_\lambda^*(R_f,R;M)$ with the general construction of
Fukaya-type categorical setting in our mind.  More specifically, we have the formula
\be\label{eq:m1}
\frak m^1(\langle \gamma_+ \rangle) = \sum_{\gamma_-} \#(\CM(\gamma_-,\gamma_+) \langle \gamma_- \rangle.
\ee
We will prove the general generic transversality result for the choice of $J$ such that the
relevant moduli spaces $\CM(\gamma_-,\gamma_+)$ are transversal in the next subsection.
For such a $J$ applied to the pair $(R_f,R)$ in $M$, we have the following structure theorem of $\frak m^1$.

\begin{prop}\label{prop:short-complex}
The operator $\mathfrak m^1$ has  the matrix form
$$
\mathfrak m^1 = \left(\begin{matrix} \pm d & 0 \\
* & \mathfrak m^1_{> 0}
\end{matrix}\right)
$$
with respect to the above decomposition.
\end{prop}
\begin{proof}
Let $\gamma_+ \in \frak{X}_{\leq -\epsilon_0}(R_f,R;M)$
and $\gamma_- \in \frak{X}_{> \epsilon_0}(R_f,R;M)$. Then we have
$$
\CA(\gamma_+) \leq  \epsilon_0, \quad \CA(\gamma_-) \geq  \frac{T_\lambda(M;R)}2
$$
and hence
\be\label{eq:<0}
\CA(\gamma_+) - \CA(\gamma_- )\leq  \epsilon_0 - \frac{T_\lambda(M;R)}2 < 0
\ee
On the other hand, \emph{if there exists a solution $u$ satisfying
\be\label{eq:contacton-RfR}
\begin{cases}
\delbar^\pi u = 0, \quad d(u^*\lambda \circ j) = 0 \\
u(\tau,0) \in R_f, \, u(\tau,1) \in R\\
u(\pm \infty) =\gamma_\pm
\end{cases}
\ee
} then
\beastar
0 & \leq & \int\int \left|\left(\frac{\del u}{\del \tau}\right)^\pi\right|^2_J\, dt\, d\tau = \int u^*d\lambda\\
& = &  \CA(u(+\infty)) - \CA(u(- \infty))
= \CA(\gamma_+) - \CA(\gamma_-)
\eeastar
which contradicts to \eqref{eq:<0}. This finishes the proof.
\end{proof}

\begin{cor}[Sandon, \cite{sandon-translated}]\label{cor:sandon}
Denote by $CI^*_{\text{\rm loc}}(R_f,R;M)$ the complex generated by the
short chords. Then we have
$$
HI^*_{\text{\rm loc}}(R_f,R;M) : = HI^*_{\leq \varepsilon_0}(R_f,R;M) \cong H^*(R;\Z_2).
$$
In particular, we have
$$
\#(\frak{Reeb}(M;\psi(R), R)) \geq \rank H^*(R;\Z_2)
$$
for any contactomorphism $\psi$ sufficiently $C^1$-close to the identity.
\end{cor}

\section{Cut-off Hamiltonian-perturbed contact instanton and energy estimates}
\label{sec:cut-off}

In this section, we recall the parameterized moduli space entering in
the proof of Sandon-Shelukhin's conjecture in \cite{oh:entanglement1} and its relevant
energy estimates. There are needed for the construction of chain-homotopy maps in
the proof of contact isotopy invariance of the contact instanton Floer cohomology we construct
as in the general Floer-type homology theory.

We will consider the two-parameter family of CR-almost complex structures and Hamiltonian functions:
$$
J = \{J_{(s,t)}\}, \, H= \{H_t^s\} \; \textrm{for} \;\; (s,t) \in
[0,1]^2.
$$
We write
$$
H_t^s(x): = H(s,t,x)
$$
in general for a given two-parameter family of functions $H = H(s,t,x)$.
Note that $[0,1]^2$ is a compact set and so $J,\, H$ are compact
families. We always assume the $(s,t)$-family $J$ or $H$ are
constant near $s = 0, \, 1$.

We will also consider the family $J' = J'_{(s,t)}$ defined by
\be\label{eq:Jst}
J'_{(s,t)} = (\psi_{H^s}^t(\psi_{H^s}^1)^{-1})^*J_{(s,t)} =  (\psi_{H^s}^1(\psi_{H^s}^t)^{-1})_*J_{(s,t)}
\ee
for a given $J = \{J_{(s,t)}\}$. We assume $J'$ satisfies
\be\label{eq:bdy-flat}
J'_{(s,t)} \equiv J_0 \in \CJ(\xi)
\ee
near $s = 0, \, 1$.

For each $K \in \R_+ =
[0,\infty)$, we define a family of cut-off functions $\rho_K:\R \to
[0,1]$ so that for $K \geq 1$, they satisfy
\be
\rho_K = \begin{cases} 0 & \quad \mbox{for } |\tau| \geq K+1 \\
1 & \quad \mbox{for }|\tau| \leq K.
\end{cases}
\ee
We also require
\bea
\rho_K' & \geq & 0 \quad \mbox{on }\, [-K-1,-K] \nonumber\\
\rho_K' & \leq & 0 \quad \mbox{on }\,  [K,K+1].
\eea
For $0 \leq  K \leq  1$, define $\rho_K = K \cdot \rho_1$. Note
that $\rho_0 \equiv 0$.
\par

Consider the following capped semi-infinite cylinders
\beastar
\Theta_- & = &\{ z\in \C  \mid \vert z\vert \le 1 \} \cup
\{ z\in \C \mid \text{Re} z \ge 0,\vert
\text{Im} z\vert \le 1\}
\\
\Theta_+ & = & \{ z\in \C \mid \text{Re} z \le 0,\vert
\text{Im} z\vert \le 1\}
\cup  \{ z\in \C  \mid \vert z\vert \le 1 \}.
\eeastar

We will fix the $K_0$ once and for all and consider $K$ with $0 \leq K \leq K_0$.
(See \cite[Proposition 8.9]{oh:entanglement1} for its required condition to satisfy.)
For such given $K_0$, we define the spaces
\beastar
\Theta_{-,K_0+1} & := & \{ z \in \Theta_- \mid \operatorname{Re} z \le K_0+1\}, \\
\Theta_{+,K_0+1} & := & \{ z \in \Theta_- \mid \operatorname{Re} z \ge -K_0-1\}.
\eeastar
We glue the three spaces
$$
\Theta_{-,K_0+1}, \quad [-K_0+1,K_0+1] \times [0,1], \quad \Theta_{+,K_0+1}
$$
subdomains of $\R \times [0,1]$, by making the identification
\beastar
(0,t) \in \Theta_{-,K_0+1} & \longleftrightarrow &  (-K_0-1,t) \in [-K_0-1,K_0+1] \times [0,1],\\
 (0,t) \in \Theta_{+,K_0+1} & \longleftrightarrow & (K_0+1,t) \in [-K_0-1,K_0+1] \times [0,1],
\eeastar
respectively. We denote the resulting domain as
\be\label{eq:Z-K}
\Theta_{K_0+1}: = \Theta_- \#_{K_0+1} (\R \times [0,1]) \#_{K_0+1} \Theta_+ \subset \C
\ee
and equip it with the natural complex structure induced from $\C$.
(See \cite[Figure 8.1.2]{fooo:book1} for the visualization of this domain.)
We can also decompose $\Theta_{K_0+1}$ into the union
\be\label{eq:Theta-K0+1}
\Theta_{K_0+1} := D^- \cup  [-2K_0 -1, 2K_0+1] \cup  D^+
\ee
where we denote
\be\label{eq:D+-}
D^\pm  = D^\pm_{K_0}:= \{ z\in \C  \mid \vert z\vert \le 1, \, \pm \text{\rm Im}(z) \leq 0 \} \pm (2K_0+1)
\ee
respectively.

We make the following specific choice of two-parameter Hamiltonians associated to each time-dependent
Hamiltonian $H = H(t,x)$  with slight abuse of notations.

\begin{choice} Take the family $H = H(s,t,x)$ given by
\be\label{eq:sH}
H^s(t,x) = s H(t,x)
\ee
\end{choice}

Then we consider the equation given by
\be\label{eq:K}
\begin{cases}
(du - X_{H_K}(u))^{\pi(0,1)} = 0, \quad d\left(e^{g_K(u)}(u^*\lambda + u^*H_K dt) \circ j\right) = 0,\\
u(\tau,0) \in R,\, u(\tau,1) \in R
\end{cases}
\ee
where we write
\be\label{eq:HK}
H_K(\tau,t,x) := H^{\rho_K(\tau}(t,x) = \rho_K(\tau) H(t,x)
\ee
and $g_K(u)$ is the function on $\Theta_{K_0+1}$ defined by
\be\label{eq:gKu}
g_K(u)(\tau,t): =  g_{\psi_{H_K}^1(\psi_{H_K}^t)^{-1}}(u(\tau,t))
\ee
for $0 \leq K \leq K_0$. We note that
if $|\tau| \geq K +1$, the equation becomes
\be\label{eq:contacton}
\delbar^\pi u = 0, \, \quad d(u^*\lambda \circ j) = 0.
\ee
\begin{defn}\label{defn:CM-K} Let $K_0 \in \R \cup \{\infty\}$ be given.
For $0 \leq K \leq K_0$, we define
\be\label{eq:MM-K}
\CM_K(M,R;J,H)  = \{ u:\R \times [0,1] \to M \, |\,  u
\; \mbox{satisfies \eqref{eq:K} and} \; E_{J_K,H}(u) < \infty\}
\ee
and
\be\label{eq:MM-para}
\CM_{[0,K_0]}^{\text{\rm para}}(M,R;J,H) = \bigcup_{K \in [0,K_0]} \{K\} \times \CM_K(M,R;J,H).
\ee
\end{defn}

With this arrangement of the moduli space  our proof of an extension of
Corollary \ref{cor:sandon} to the one stated in
Theorem \ref{thm:definition-intro} crucially relies on the following two energy inequalities
proven in \cite{oh:entanglement1}:
We recall the definition of oscillation function $\osc(H):[0,1] \to \R$ defined by
$$
\osc(H)(t) = \osc(H_t) = \max H_t - \min H_t.
$$
\begin{prop}[Proposition 8.6 \cite{oh:entanglement1}] \label{prop:pienergy-bound}
Let $u$ be any finite energy solution of \eqref{eq:K}. Then we have
\be\label{eq:pienergy-bound}
E_{(J_K,H)}^\pi(u) \leq \int_0^1 \osc(H_t) \, dt = :\|H\|
\ee
\end{prop}

\begin{prop}[Proposition 8.7 \cite{oh:entanglement1}]\label{prop:lambdaenergy-bound}
Let $u$ be any finite energy solution of \eqref{eq:K}. Then we have
\be\label{eq:lambdaenergy-bound}
E^\perp_{(J_K,H)}(u) \leq \|H\|.
\ee
\end{prop}

The proof also was
modulo the details of the full construction of (cylindrical) Legendrian contact
instanton homology the construction of which we provides
in the remaining part of the present part.

\section{Definition of boundary map}

Let $\psi \neq id$ be a contactomorphism with $\psi(R) \pitchfork R$. Then by the dimensional
reason we have
$$
\psi(R) \cap R = \emptyset.
$$
Under this condition of $\psi$, we consider the $\Z_2$-vector space
$$
CI_\lambda^*(\psi(R), R) : = \Z_2\langle\frak{X} (\psi(R),R)\rangle \cong \Z_2\langle \psi(R) \cap Z_R \rangle.
$$

\begin{prop}\label{prop:generic-transversality-CM}
There exists a residual subset of $J_0 \in \CJ(\lambda)$ such that the
moduli spaces $\CM(M,\lambda;R;\gamma_-,\gamma_+)$ are transversal for any
pair $\gamma_\pm \in \mathfrak{R}eeb(\psi(R),R)$.
\end{prop}

Then under the energy inequalities given in Proposition \ref{prop:pienergy-bound} and
Proposition \ref{prop:lambdaenergy-bound},
the bubbling analysis developed in \cite[Part 3]{oh:entanglement1} proves that
the moduli space $\CM(M,\lambda;R;\gamma_-,\gamma_+)$ is compact when it has dimension 0
and hence we can define a $\Z_2$-linear map
$$
\delta_{(\psi(R),R;H)} : CI_\lambda^*(\psi(R),R) \to  CI_\lambda^*(\psi(R),R)
$$
by its matrix element
$$
\langle \delta_{(\psi(R),R;H)}(\gamma^-), \gamma_+ \rangle : = \#_{\Z_2}(\CM(\gamma_-,\gamma_+)).
$$
Here $\CM(\gamma_-,\gamma_+)$ is the moduli space
$$
\CM(\gamma_-,\gamma_+) = \widetilde{\CM}(\psi(R),R;\gamma_-,\gamma_+)/\R
$$
of contact instanton Floer trajectories $u$ satisfying
$u(\pm\infty) = \gamma_\pm$. More precisely, as pointed out in \cite[Remark 10.7]{oh:entanglement1},
we note that the virtual dimension of the moduli space of holomorphic
half plane
 $$
 \CM((\H,\del \H), (M,R); \gamma)
 $$
for a given asymptotic Reeb chord of $R$ will have at least 2 thanks to the action of
automorphism of $(D^2 \setminus \{1\}, \del D^2 \setminus \{1\}) \cong (\H,\del \H)$ preserving the infinity.
Therefore if the starting pair of Reeb chords $\gamma_\pm \in\frak{X} (\psi(R),R)$
is the one with the virtual dimension of the associated moduli
space $\widetilde \CM(\psi(R),R;\gamma_+,\gamma_-)$ is one, the moduli space cannot
produce such a bubble.
This implies  compactness of the quotient
$$
\widetilde \CM(\psi(R),R;\gamma_+,\gamma_-)/\R
$$
when $\text{\rm vir.dim.} \widetilde \CM(\psi(R),R;\gamma_+,\gamma_-) = 1$.
Furthermore by the Floer-type compatible gluing theorem for the moduli spaces
$\CM(\gamma_-,\gamma_+)$ developed in the previous part in particular gives rise to the identity
$$
\delta_{(\psi(R),R;H)} \circ \delta_{(\psi(R),R;H)} = 0
$$
under the following two circumstances:
\begin{enumerate}
\item when $\psi = \psi_H^1$ for the contact Hamiltonian satisfying $\|H\| < T_\lambda(M,R)$,
\item after deforming the boundary map $\delta$ to $\delta^{\frak b, \psi_*\frak b}$
via a bounding cochain $\frak b$ killing the obstruction, the $\frak m^0$.
\end{enumerate}

Combining the fact that $\psi$ is contact isotopic to the identity,
we can also prove that this counting of moduli space of dimension one gives rise
to the boundary map $\delta$ satisfying $\delta^2 = 0$ by the arguments similar to the one
used in \cite{oh:cpam-obstruction} in the Lagrangian Floer theory. Therefore
the cohomology group $HI_\lambda^*(\psi(R),R)$ is well-defined.

\begin{rem}\label{rem:invariance} It is worthwhile to mention that our invariance proof
to be given in the following subsection is  straightforward which is similar
to that of Floer cohomology of \emph{compact} Lagrangian submanifolds given in \cite{oh:cpam1}
which uses the \emph{moving boundary condition}. Such a proof is rather subtle to use in the context of
\emph{Legendrian contact homology constructed via the symplectization} in the literature of
contact topology because of the issue of `moving the infinity of the cylindrical Lagrangian'.
The usual invariance proof of contact homology in the literature uses the argument using the exact symplectic cobordism
and is in the spirit rather different from our invariance proof in the next subsection.
(See \cite{EGH} for a sketch of such a proof. This was carried out in \cite[Appendix B]{ekholm-rational}
under some technical assumptions on the contact manifold.)
There is also another method called the bifurcation method which has been used in the literature
similar to Floer's original invariance proof given in \cite{floer-intersections}. A rigorous
analytic proof of this kind requires delicate gluing analysis
of the type established by Yi-Jen Lee in \cite{yijen-lee} even in the Lagrangian Floer theory
on closed symplectic manifolds.
Because of this, such a proof is established only in some special cases of the type
$(P \times \R, dz - \theta)$ for an exact symplectic manifold $(P,d\theta)$.
When the bifurcation method works, it would induce a \emph{stable tame isomorphism},
introduced by Chekanov \cite{chekanov:dga}, which is conjecturally stronger than a DGA
quasi-isomorphism in general. (See \cite[Appendix]{ekholm-rational} and
\cite[Proposition 1.6]{rizell-sullivan} for an explanation on the current status of this bifurcation method proof.)
\end{rem}

\section{Definitions of the chain map and the chain homotopy map}
\label{sec:cobordism}

In this section, we first define the chain map and the chain homotopy map restricting ourselves
to the case of this study that enters in the proof of Theorem \ref{thm:definition-intro}, and
prove the basic algebraic relation of the chain map to be a quasi-isomorphism by
applying the gluing theory of contact instantons developed in Part \ref{part:gluing-contacton}.
We postpone a full study of chain maps and their homotopies in the general context
till \cite{oh:entanglement2}.

In this section, we assume that $(\lambda, (\psi(R),R))$ is nondegenerate
in the sense of Definition \ref{defn:nondegeneracy-chords}.

For given Hamiltonian $H = H(t,x)$,
we consider the inverse Hamiltonian of $H$ given by
\be\label{eq:inverse-Hamiltonian}
\overline H(t,x) : = e^{-g_{\psi_H^t} \circ \psi_H^t} H(t,\psi_H^t(x))
\ee
(see \cite{mueller-spaeth-I}),  and consider the maps
$$
\Psi_H = (\Phi_H)_*:  CI_\lambda^*(R,R) \to  CI_\lambda^*(\psi(R),R)
$$
and
$$
\Psi_{\overline H} = (\Phi_H^{-1})_*: CI_\lambda^*(\psi(R),R) \to CI_\lambda^*(R,R)
$$
where $\Phi_H$ is the gauge transformation defined by
\be\label{eq:u-to-w}
\Phi_H(\ell)(t): = \psi_H^t (\psi_H^1)^{-1}(\ell(t))
\ee
defined on the path space $\CP(R_0,R_1)$. (See \cite[Section 6]{oh:entanglement1}.)

Recall $\psi_{\overline H}^1 = \psi^{-1}$.
Then we consider the composition
$$
\Psi_{\overline H}\circ \Psi_H:  CI_\lambda^*(R,R) \to  CI_\lambda^*(R,R).
$$
Following \cite[Section 8]{oh:entanglement1},
we consider the two-parameter family of CR-almost complex structures and Hamiltonian functions:
$$
J = \{J_{(s,t)}\}, \, H= \{H_t^s\} \; \textrm{for} \;\; (s,t) \in
[0,1]^2.
$$
We write
$$
H_t^s(x): = H(s,t,x)
$$
in general for a given two-parameter family of functions $H = H(s,t,x)$.
Note that $[0,1]^2$ is a compact set and so $J,\, H$ are compact
families. We always assume the $(s,t)$-family $J$ or $H$ are
constant near $s = 0, \, 1$.

We also consider the family $J' = J'_{(s,t)}$ defined by
\be\label{eq:Jst}
J'_{(s,t)} = (\psi_{H^s}^t(\psi_{H^s}^1)^{-1})^*J_{(s,t)} =  (\psi_{H^s}^1(\psi_{H^s}^t)^{-1})_*J_{(s,t)}
\ee
for a given $J = \{J_{(s,t)}\}$. We assume $J'$ satisfies
\be\label{eq:bdy-flat}
J'_{(s,t)} \equiv J_0 \in \CJ(\xi)
\ee
near $s = 0, \, 1$.

We consider a  one-parameter family of Hamiltonians $H_K$ given in \eqref{eq:HK}.
Using this family, we define a family of homomorphisms
$$
\Psi_{H_K}: CI_\lambda^*(R,R) \to  CI_\lambda^*(R,R)
$$
with $0 \leq K \leq K_0$
which defines a chain homotopy map
$$
\mathfrak H: CI_\lambda^*(R,R) \to  CI_\lambda^{*-1}(R,R)
$$
between $\Psi_{\overline H}\circ \Psi_H$ and
$id$ on $ CI_\lambda^*(R,R)$, i.e., it satisfies
$$
\Psi_{\overline H}\circ \Psi_H -id = \delta \mathfrak H + \mathfrak H \delta
$$
on $CI_\lambda^*(R,R)$
\emph{provided the relevant parameterized moduli space
$\CM^{\text{\rm para}}_{[0,K_0]}(M,R;J,H)$ that we considered in the previous section does not
bubble-off}. A standard algebraic argument then shows that the latter homotopy identity follows
as long as no bubbling occurs which is ensured by the inequality
$\|H\| < T_{\lambda}(M,R)$.
Therefore we have shown
$$
HI^*_\lambda(\psi(R),R) \cong HI^*_\lambda(R,R).
$$
Now  to complete the proof of Sandon-Shelukhin's conjecture,
it remains to show that $HI^*_\lambda(R,R) \cong H^*(R)$ whose proof
is explained in \cite[Section 10]{oh:entanglement1} and omitted.

\part{Gluing theory of (relative) pseudoholomorphic curves in SFT}
\label{part:gluing-SFT}

In this part, we first specialize the gluing construction carried out in
Part \ref{part:gluing-contacton} to the exact case, i.e.,
when the closed one-form $w^*\lambda \circ j$ is exact. In this case, we can
naturally lift the contact instanton $w$ to a map a $\widetilde J$-holomorphic map
$$
u = (w,f): \dot \Sigma \to V \times \R
$$
in the symplectization
$$
(W,\omega) = (V \times \R, d(e^s \lambda)), \quad s \in \R
$$
by considering a primitive $f$ satisfying $w^*\lambda \circ j = df$ and setting
$ s\circ u= : f$.
Such $f$ is unique modulo by addition of constant $f \mapsto f +c$.

We then explain how the gluing theory developed in
Part \ref{part:gluing-contacton} canonically produces the relevant gluing theory for the
pseudoholomorphic curves in SFT, especially on the symplectization, by an easy soft massaging of the construction.

The natural setting of gluing theory in SFT \cite{EGH} is to glue two
$J_\pm$-holomorphic curves $u_\pm$ in a two-story building. More precisely,
we consider a completed symplectic cobordism
\be\label{eq:cobordism}
W_i= \stackrel{\longrightarrow}{V_i^-,V_i^+}, \quad i = 1, 2
\ee
and consider a 2-story (bordered) stable curve
$(u_1, u_2)$ where $u_i = \{h_{i,j}\}$, $i = 1, 2$ is a collection of holomorphic curves
$$
h_{i,j}:S_j \to W_i, \quad j = 1, \cdots, k_i, \quad \text{\rm for } \, i = 1, \, 2.
$$
We require the matching condition
\be\label{eq:Gamma12-matching}
\Gamma_1^+ = \Gamma_2^-
\ee
of the asymptotic limits of $u_1$ and $u_2$
on the asymptotic boundary $V_1^+ = V_2^-$ of $W_i$'s. Details follow.
Because this is not the main part of the current paper,
our exposition of the present part will be concise  by focusing on the issue of how we
can lift the gluing problem of contact instantons to that of the standard gluing problem
of Gromov-Witten-Floer theory applied in the following steps:
\begin{enumerate}
\item First on the neck region of the buildings: this step involves only a soft
procedure of determining the $\R$-component $f$ of the map $u = (w,f)$ when $w$ is provided.
\item Next on the region away from the neck region: by now the first step provides
a middle piece of the pre-gluing map together with the pieces away from the
neck region.
\item Glue the aforementioned pieces and define a pregluing map $\text{\rm PreG}(u_1,u_2,K)$.
\item Establish the exponentially small error estimate for $\text{\rm PreG}(u_1,u_2,K)$.
\item The linearized version of the above steps then give rise to a good approximate
right inverse of the linearized operator.
\item Finally apply the ordinary gluing theory in the Floer theory on the tame symplectic manifold,
without imvolving any scale-dependent gluing beyond the ordinary gluing process of
Gromov-Witten-Floer theory.
\end{enumerate}

We start with the case of trivial cobordism.

\section{Cylindrical curves in trivial cobordism: symplectization}
appl
Let $(V,\xi)$ be a compact contact manifold (more generally a tame contact manifold
in the sense of \cite{oh:entanglement1})
equipped with a contact form $\lambda$,
and consider the symplectization
$$
(W,\omega) = (V \times \R, d(e^s \lambda)), \quad s \in \R.
$$
We consider two copies thereof and denote them by $W_1$ and $W_2$.
We also decompose their boundaries
$$
\del W_i = V_i^- \sqcup V_i^+, \quad i = 1, \, 2.
$$
We mention that $V_i^\pm = V$ for all $i$ and $\pm$ in the present case.

Let $(u_1,u_2)$ be a 2-story stable map such that $u_i = \{h_{i,j}\}$,  $i = 1, 2$ are collections of holomorphic curves
$$
h_{i,j}:S_j \to W_i, \quad j = 1, \cdots, k_i, \quad \text{\rm for } \, i = 1, \, 2.
$$
Each $h_{i,j}$ is decompose into
$$
h_{i,j} = (w_{i,j}, f_{i,j})
$$
such that it satisfies
\be\label{eq:equation-in-symplectization}
\delbar^\pi w_{i,j} = 0, \quad w_{i,j}^*\lambda \circ j = df_{i,j}.
\ee
In particular $w_{i,j}$ satisfies
\be\label{eq:wij-contacton}
\delbar^\pi w_{i,j} = 0,  \quad d(w_{i,j}^*\lambda \circ j) = 0.
\ee
Therefore all the a priori estimates established for contact instantons apply to these $w_{i,j}$'s too.
In addition for the current exact case, the equation $w_{i,j}^*\lambda \circ j = df_{i,j}$
uniquely determines $f_{i,j}$ modulo a uniform addition by constant $c_{i,j}$. By a suitable normalization,
we can fix the constant $c_{i,j}$ to be zero.

Since general case is not very different, we focus on the closed string case, i.e., assume that
the surface $\dot \Sigma$ has no boundary. As before we equip a punctured neighborhood of each
puncture with a cylindrical coordinates $(\tau,t) \in \pm [0,\infty) \times S^1$ with the sign
depending on the sign of the puncture. We further focus on the punctures
of $\dot \Sigma_1$ and $\dot \Sigma_2$ such that the relevant asymptotic limits of $w_{1,j}$ and $w_{2,j}$
coincide with the collection of Reeb orbits
$$
\Gamma' = \{\gamma_1', \cdots \gamma_k'\}
$$
Again without loss of generalities, we assume $k = 1$ and denote by $\gamma'$ the associated
asymptotic orbit, i.e.,
\be\label{eq:asymptotic-limit}
\lim_{\tau \to \infty}w_1(\tau) = \gamma' = \lim_{\tau \to -\infty} w_2(\tau).
\ee

Now for a sufficiently large $R$, we glue two $V \times \R$
by taking the union
$$
W_1 \#_R W_2: = V \times (-\infty, R]  \sqcup V \times [-R,\infty])/\sim
$$
with the identification of $(x,R)$ in the first copy with $(x,-R)$ in the second copy.
We denote by
$$
\tau_R: W_1 \#_R W_2 \to V \times \R
$$
the map defined by
$$
\tau_R([x,s]) = \begin{cases} (x,s -R) \quad s \in (-\infty, R] \\
(x, s+R) \quad s \in [-R,\infty).
\end{cases}
$$
\begin{rem} This map is the map realizing the so called \emph{neck-stretching operation}.
\end{rem}

Recall that $w_1, \, w_2$ are maps into the same contact manifold
$(V,\lambda)$ since $W_1 = V \times \R = W_2$, which satisfies the contact
instanton equation \eqref{eq:wij-contacton}. Then let
$$
w_K = w_1 \widehat \#_K w_2 = \text{\rm Glue}(w_1, w_2, K)
$$
be the gluing solution on $V$ constructed in Part \ref{part:gluing-contacton}
for
$$
(w_1,w_2, K) \in \K_1 \times \K_2 \times \pm [K_0,\infty)
$$
for a sufficiently large $K_0 > 0$ as before.

Now we consider the $\R$-component $f$ and  solve the equation
$$
df = w_K^*\lambda \circ j.
$$
\begin{cond}[Normalization]\label{cond:normalization}
We assume that each puncture is equipped with cylindrical coordinates $(\tau,t)$ on $\pm [0, \infty)$.
For each given sufficiently large $K > 0$, we normalize the choice of $f=f_K$ above by requiring it to satisfy
\be\label{eq:normalization}
f_K(0,0)= \frac{f_1(0,K) + f_2(0,-K)}2.
\ee
\end{cond}
This uniquely determines the function $f_K: V \times \R \to \R$ and $u_K: = (w_K,f_K)$ solves the equation
\be\label{eq:fR}
\begin{cases}
 df_K =w_K^*\lambda \circ j, \\
f_K(0,0)= \frac{f_1(0,K) + f_2(0,-K)}2.
\end{cases}
\ee
Then the following exponential estimates
immediately follows from the equation $w_i^*\lambda \circ j = df_i$ and the normalization condition \eqref{eq:normalization}
and the exponential decay estimates of
$u_i(w_i,f_i)$ $i = 1,2$ as $|\tau| \to \infty$  given in Proposition \ref{prop:dw-expdecay}.

More precisely, we have the following

\begin{prop} Let $K \geq K_0$ and denote by $f_{i,2K}$ the functions defined by
$$
f_{1,2K}(\tau,t) = f_1(\tau +2K,t), \quad f_{2,2K}(\tau) = f_2(\tau - 2K,t).
$$
Then if $K_0 > 0$ is sufficiently large depending only on $\K_1, \, \K_2$ such that for all $\ell \geq 2$, $$
|\nabla^\ell(f_K - f_{1, 2K} )(\tau,t)| \leq C e^{-\delta K}
$$
for all $\tau \geq K$ and
$$
|\nabla^\ell(f_K - f_{2,2K}) (\tau , t)| \leq C e^{-\delta K}
$$
for all $\tau \leq -K$
\end{prop}
\begin{proof} We look at the equation
$$
df_K = w_K^*\lambda \circ j
$$
for $f_K$ for the given $w_K = \Glue(w_1,w_2,K)$ with $u_i = (w_i,f_i)$, $i= 1, \, 2$.

The equation
$$
w_K^*\lambda \circ j = df_K
$$
implies
$$
\frac{\del f_K}{\del \tau} = \lambda\left(\frac{\del w_K}{\del t}\right), \quad
\frac{\del f_K}{\del t} = - \lambda\left(\frac{\del w_K}{\del \tau}\right).
$$
In particular we obtain
$$
\frac{\del f_K}{\del \tau} - T = \lambda\left(\frac{\del w_K}{\del t}\right) - T
$$
where $T$ is the action of the asymptotic limit $\gamma'$ of $w_K$ at the relevant puncture, i.e.,
$$
T = \int (\gamma')^*\lambda.
$$
Then the required exponential estimates
immediately follow by integrating over from $\tau$ to $\tau + K$ for any $\tau \geq K_0$(resp. from $\tau - K$ to $\tau$
for any $\tau \leq -K_0$) from the normalization condition \eqref{eq:normalization}
and the exponential decay estimates of $w_i$ $i = 1,2$ as $|\tau| \to \infty$
given in Proposition \ref{prop:dw-expdecay}. This finishes the proof.
\end{proof}

The above construction is canonical and hence produces a smooth family-gluing map
$$
\Glue: \K_1 \times\K_2 \times [K_0,\infty) \to  \CM(W_1 \#_K W_2, J_K;\Gamma_1^-,\Gamma_2^+)
$$
for any given compact subsets
$$
\K_i \subset \CM(W_i,\Gamma_i^-,\Gamma_i^+), \quad i = 1, \, 2.
$$
This finishes the gluing construction in the context of symplectization, i.e., in the
context of trivial cobordism.

\section{General 2-story curves in general cobordisms}
\label{sec:general-cobordism}

We now consider a general 2-story building $(W_1,W_2)$ and a 2-story (bordered) stable curve
$(u_1, u_2)$ where $u_i = \{h_{i,j}\}$, $i = 1, 2$ is a collection of holomorphic curves
$$
h_{i,j}:S_j \to W_i, \quad j = 1, \cdots, k_i.
$$
We require the matching condition
\be\label{eq:Gamma12-matching}
\Gamma_1^+ = \Gamma_2^-
\ee
on the asymptotic boundary $V_1^+ = V_2^-$ as before where $V_1^+$ is the floor  is the
floor of $W_1$ and the ceiling of $W_2$.

We fix a decomposition
\be\label{eq:decompose-Wi}
W_i = V_i^- \times (-\infty, 0]  \cup W_i^{\text{\rm mid}} \cup V_i^+ \times [0,\infty)
\ee
for each $i = 1, \, 2$. Then we have
$$
\del W_i^{\text{\rm mid}} = V_i^{\text{\rm mid},-} \cup V_i^{\text{\rm mid},+}
$$
where $ V_i^{\text{\rm mid},\pm}$ is a copy of $V_i^\pm$ given by
$$
V_i^{\text{\rm mid},\pm} = V_i^\pm \times \{0\}.
$$
The rest of the gluing process has two steps.

\subsubsection{Gluing the target}

We start with gluing the targets $W_1$ and $W_2$ along $V_1^+ = V_2^-$.
For each $R > 0$ sufficiently large, we decompose the image $C_i:= \Image u_i$ into the union
$$
C_i = C_i^{-R} \bigcup C_i^{\text{\rm mid},R} \bigcup C_i^{+R}
$$
for each $i=1,\, 2$ by taking the intersections
$$
C_1^{\pm R} := C_1 \cap (V_i \times (-\infty, -R]), \quad
C_2^{\pm R} := C_2 \cap (V_i \times [R,\infty))
 $$
and setting
$$
C_i^{\text{\rm mid, R}} := C_i \cap W_i^{\text{\rm mid}}
$$
\begin{lem} There exists a sufficiently large $R_0 > 0$ such that the preimages
$$
u_i^{-1}(C_i^{\pm R})
$$
are either empty or a finite collection of maps
$$
h_{i,j}^\pm: S_{i,j}^\pm \to \del^\pm_\infty W_i \times (\pm [R_i,\infty))
$$
each connected component of which has domain
$$
S_{i,j}^\pm \subset \R \times S^1
$$
has one cylindrical end and its boundary $\del S_{i,j}$ becoming a simple closed curve in $\R \times S^1$.
\end{lem}
\begin{proof} First recall from definition that we have the asymptotic boundary of
$\del_\infty W_i = \del_\infty^- W_i \sqcup \del_\infty^+ W_i$ is given by
$$
\del^\pm_\infty W_i = V_i^\pm
$$
for $i = 1, \, 2$ respectively. Then the lemma is an immediate consequence of the exponential convergence property
$u_i$ and the maximum principle.
\end{proof}

Therefore we can take a sufficiently large $K_1 > 0$ so that
$$
S_{i,j} \supset \pm[K_1, \infty).
$$
By restricting the relevant curve to $\pm[K_1, \infty)$, we assume that the domain of $h_{i,j}$ is
the semi-cylinder
$$
\pm[K_1, \infty) \times S^1.
$$
Then we regard $w_i = w_{i,j}$ as maps to the same contact manifold
$V := V_1^+ = V_2^-$ and solve the
gluing problem thereof for each $K \geq K_1$.

\subsubsection{Gluing the maps}

We denote by
the resulting glued solution by
$$
w_K = \Glue(w_1,w_2,K).
$$
For each $K_2 \geq K_1$, we can uniquely solve the equation
\be\label{eq:wi-dfi}
\begin{cases}
df_K = w_K^*\lambda \circ j, \\
f_K(0,0) =  \frac{f_1(0,+K_1) + f_2 (0,-K_1)}2
\end{cases}
\ee
on the neck-region $[-K_2, K_2] \times S^1$ of the domain $S_{1}^+ \cup_{K_1} S_{2} =: S_K$ which is
the domain glued along the cyindrical region around the relevant puncture.

This provides a pseudoholomorphic curve $u_{K_2}^{\text{\rm mid}}$ defined on
$$
\pm[-K_2, K_2] \times S^1.
$$
By considering sufficiently large $K_1$ and interpolating this curve with the curve away from the
neck region,
$$
u_K|_{\R \times S^1 \setminus [-K_1,K_1] \times S^1}.
$$
We denote by $u_{i;K_1}$ with $i = 1, \, 2$ the part of whose image is
contained in $W_i$ respectively. Then we define a pre-gluing
approximate solution
$$
\PreG(u_1,u_2,K): = u_{1,K}^- \# u_{K}^{\text{\rm mid}} \# u_{2,K}^+=: \uapp
$$
which defines a map
$$
\uapp: \R \times [0,1] \to W_1 \#_{R_0} W_2 =: W_{R_0}.
$$
We fix a sufficiently large $R_0> 0$ depending only on $\K_1, \, \K_2$.
Then the assignment
$$
(u_1,u_2,K) \mapsto \PreG(u_1,u_2,K)
$$
defines a smooth map
$$
\K_1 \times \K_2 \times [K_0, \infty) \to \CW^{k,p}(\dot \Sigma_K, W_{R_0}).
$$
We note that the target manifold $W_{R_0}$ equipped with the induced symplectic form and the induced
almost complex structure is a \emph{fixed} almost K\"ahler manifold with cylindrical ends.

At this stage, the following is again an immediate consequence of exponential convergence of $u_i$
in the strip-like coordinates near each puncture of $\dot \Sigma$.

\begin{prop} The above constructed $\PreG(u_1,u_2,K_0)$ is a
good approximate solution for the pseudoholomorphic curve equation.
\end{prop}

Once we have this construction of good approximate solution carried out, then the standard gluing
operation in the literature (e.g., such as the one given in \cite{ruan-tian}, \cite{mcduff-salamon-symplectic}, \cite{fooo:book2})
finishes the proof of gluing theorem.

\appendix

\section{Review of the contact triad connection}
\label{sec:connection}

Assume $(M, \lambda, J)$ is a contact triad for the contact manifold $(M, \xi)$, and equip with it the contact triad metric
$g=g_\xi+\lambda\otimes\lambda$.
In \cite{oh-wang:connection}, the authors introduced the \emph{contact triad connection} associated to
every contact triad $(M, \lambda, J)$ with the contact triad metric and proved its existence and uniqueness.

\begin{thm}[Contact Triad Connection \cite{oh-wang:connection}]\label{thm:connection}
For every contact triad $(M,\lambda,J)$, there exists a unique affine connection $\nabla$, called the contact triad connection,
 satisfying the following properties:
\begin{enumerate}
\item The connection $\nabla$ is  metric with respect to the contact triad metric, i.e., $\nabla g=0$;
\item The torsion tensor $T$ of $\nabla$ satisfies $T(R_\lambda, \cdot)=0$;
\item The covariant derivatives satisfy $\nabla_{R_\lambda} R_\lambda = 0$, and $\nabla_Y R_\lambda\in \xi$ for any $Y\in \xi$;
\item The projection $\nabla^\pi := \pi \nabla|_\xi$ defines a Hermitian connection of the vector bundle
$\xi \to M$ with Hermitian structure $(d\lambda|_\xi, J)$;
\item The $\xi$-projection of the torsion $T$, denoted by $T^\pi: = \pi T$ satisfies the following property:
\be\label{eq:TJYYxi}
T^\pi(JY,Y) = 0
\ee
for all $Y$ tangent to $\xi$;
\item For $Y\in \xi$, we have the following
$$
\del^\nabla_Y R_\lambda:= \frac12(\nabla_Y R_\lambda- J\nabla_{JY} R_\lambda)=0.
$$
\end{enumerate}
\end{thm}
From this theorem, we see that the contact triad connection $\nabla$ canonically induces
a Hermitian connection $\nabla^\pi$ for the Hermitian vector bundle $(\xi, J, g_\xi)$,
and we call it the \emph{contact Hermitian connection}.

Moreover, the following fundamental properties of the contact triad connection was
proved in \cite{oh-wang:connection}, which will be useful to perform tensorial calculations later.

\begin{cor}\label{cor:connection}
Let $\nabla$ be the contact triad connection. Then
\begin{enumerate}
\item For any vector field $Y$ on $M$,
\be\label{eq:nablaYX}
\nabla_Y R_\lambda = \frac{1}{2}(\CL_{R_\lambda}J)JY;
\ee
\item $\lambda(T|_\xi)=d\lambda$.
\end{enumerate}
\end{cor}

We refer readers to \cite{oh-wang:connection} for more discussion on the contact triad connection and its relation with other related canonical type connections.

\section{Newton's iteration scheme and outline of Taubes' gluing}
\label{sec:gluing-outline}

In this appendix, we duplicate the outline of Taubes' gluing scheme
explained in \cite[Subsection15.5.1]{oh:book2} and provide an outline of Taubes' gluing scheme.

We first motivate the gluing construction
by comparing it with the well-known Newton's iteration scheme of solving the equation
$f(x) = 0$ for a real-valued function, starting from an approximate solution $x_0$ that is sufficiently close to a genuine solution nearby. Along the way we visualize how
the error estimate and how the approximate (right) inverse enters in the iteration scheme.

Newton's iteration scheme is the inductive scheme achieved by the
recurrence relation
\be\label{eq:Newton's}
x_{n+1} = x_n - \frac{f(x_n)}{f'(x_n)}
\ee
starting from the initial trial solution $x_0$. Then we have
\be\label{eq:difference}
|x_{n+1} - x_n| = \left|\frac{f(x_n)}{f'(x_n)}\right|
\ee
and the expected solution is given by the infinite sum
$$
x_\infty = x_0 + \sum_{n=1}^\infty (x_{n+1} - x_n)
$$
\emph{provided that the series converges}. A simple way of ensuring the required convergence is to
establish existence of constant $C > 0, \, 0 \leq \mu < 1$ independent of $n$ such that
\be\label{eq:geometric}
\left|\frac{f(x_n)}{f'(x_n)}\right| \leq C \mu^n
\ee
for all $n$. This last inequality requires that \emph{$f(x_n)$ ,especially the initial error
$|f(x_0)|$, should be small relative to the inverse of the derivative $|f'(x_0)|$.}

With this iteration scheme in our mind, we recall the abstract formulation
of Taubes' gluing scheme explained in \cite[Subsection 15.5.1]{oh:book2}.
Let $F: B_1 \to B_2$ be a smooth map between
Banach spaces $B_1, \, B_2$. We would like to solve the equation
\be\label{eq:F(x)=0}
F(x) = 0.
\ee
The starting point is the presence of an approximate solution $x_0$, i.e.,
$F(x_0)$ is small in $B_2$, say
\be\label{eq:Ferror}
|F(x_0)| \leq \e
\ee
In applications, an approximate solution is manifest in the given
geometric context, and very often comes from some \emph{degenerate}
configurations. Then we would like to perturb $x_0$ to $x_0 + h$
to a solution of the form $x = x_0 +h$. So we Taylor-expand
$$
F(x_0 + h) = F(x_0) + dF(x_0) h + N(x_0,h)
$$
where $N$ satisfies the following two inequalities:
\begin{enumerate}
\item[(1)]
\beastar
|N(x_0,h)| &\leq& o(|h|)|h|\\
|N(x_0,h_1) - N(x_0,h_2)| &\leq& o(|h_1| + |h_2|)|h_1-h_2|.
\eeastar
\item[(2)] Assume that
there exists a function $\varepsilon_1 = \varepsilon_1(\delta)$ such that
$\varepsilon_1(\delta) \to 0$ as $\delta \to 0$  and
\bea\label{eq:N}
\frac{|N(x_0,h)|}{|h|} & \leq & \varepsilon_1(|h|) \\
\frac{|N(x_0,h_1) - N(x_0,h_2)|}{|h_1-h_2|} & \leq & \varepsilon_1(|h_1| + |h_2|).
\eea
\item[(3)]
We also assume that $\varepsilon_1(\delta) \to 0$ uniformly over $x_0$ when we consider
a family of approximate solutions $x_0$.
\end{enumerate}
Now, setting $F(x_0 + h) = 0$, we obtain
$F(x_0) + dF(x_0) h + N(x_0,h) = 0$ or
\be\label{eq:dF}
dF(x_0)h = -F(x_0) + N(x_0,h).
\ee
\begin{enumerate}
\item[(4)]
Now suppose $dF(x_0)$ is surjective linear map and let $Q(x_0)$ be
its right inverse so that $dF(x_0)\circ Q(x_0) = \id$.
\item[(5)]
When we consider a family of approximate solutions $x_0$, then we assume
that there exists a uniform constant $C > 0$ with
\be\label{eq:uniform|Q|}
\|Q(x_0)\| \leq C
\ee
for $x_0$.
\end{enumerate}

Then we put an Ansatz $h = Q(x_0)k$ for $k \in B_2$. Then \eqref{eq:dF}
is reduced to
\be\label{eq:fixedpteq}
k=-F(x_0) + N(x_0, Q(x_0)k).
\ee
Now we regard the right hand side as a map from $B_2$ to $B_2$
and denote this map by
\be\label{eq:Gk}
G(k) = -F(x_0) + N(x_0, Q(x_0)k).
\ee
At this point, we note that $G(0) = -F(x_0)$ is assumed to be sufficiently small.
To solve the fixed point problem $k = G(k)$ of the map $G$, we will apply Picard fixed point theorem. To apply
the fixed point theorem, we need to prove existence of $\delta > 0$ such that
\begin{enumerate}
\item $G$ maps a closed ball $B(\delta) \subset B_1$ to $B(\delta)$ itself.
\item $G$ is a contraction map, i.e., it satisfies
$$
|G(x_1) - G(x_2)| \leq \lambda|x_1 -x_2|
$$
for some $0 < \lambda < 1$.
\end{enumerate}

Then the above discussion is summarized into the following

\begin{prop}[Proposition15.5.1 \cite{oh:book2}]\label{prop:model} Let $B_1, \, B_2$ be Banach spaces and
let $F:B_1 \to B_2$ be a smooth map. Let $x_0 \in U$ satisfy
$$
|F(x_0)| \leq \e
$$
where $U \subset B_1$ be an open subset.
Suppose that $dF(x_0)$ is surjective and $Q(x_0)$ is its right
inverse and assume $|Q(x_0)| \leq C$ uniformly over $x_0 \in U$.
Choose any $\delta_2 = 2\e$ as above.
Then $F(x)$ has a unique solution of the form $x = x_0 + Q(x_0)k$ with
$$
|k| \leq \delta_2.
$$
In particular, there exists some $\e_0 > 0$ such that whenever $ 0 < \e \leq \e_0$
the perturbation error $|Q(x_0)k|$ can be made in the same
order of $\e$ so that
$$
|Q(x_0)k| \leq C \e
$$
i.e., so that the genuine solution $x$ lies in the ball $B_{x_0}(C\e) \subset U$
uniformly over $x_0$.
\end{prop}

Then the final gluing result we would like to achieve is the following type

\begin{thm}[Compare with Theorem 15.5.2 \cite{oh:book2}]\label{thm:glueu1u2} Let $\K_- \subset  \MM(\gamma_-,\gamma')$,
$\K_+ \subset \MM(\gamma',\gamma_+)$ be given compact subsets.
\begin{enumerate}
\item  There exists some $K_0 > 0$ and a diffeomorphism
$$
Glue: \K_-\times \K_+ \times (R_0,\infty) \to
\MM(\gamma_-,\gamma_+)
$$
onto its image.
\item Furthermore there exists a homeomorphism
$$
\varphi: [R_0,\infty) \cup \{\infty\} \to (-\e_0,0]
$$
such that the reparameterized diffeomorphism
$$
Glue_\varphi: \K_- \times \K_+ \times (-\e_0,0)
\to \MM(\gamma_-,\gamma')
$$
extends to the map
$$
\overline{Glue}_\varphi: \K_- \times \K_+ \times (-\e_0,0]
\to \overline \MM(\gamma_-,\gamma_+)
$$
so that the following diagram commutes:
$$
\xymatrix{\K_- \times \K_+ \times \{0\} \ar[r]^(0.6){\cong}
 & \K_- \times \K_+ \\
\K_- \times \K_+ \times (-\e_0,0] \ar@{<-_{)}}[u] \ar[dr]
\ar[r]^(0.6){\overline{Glue}_\varphi}\ar[dr]
& \overline \MM(\gamma_-,\gamma_+) \ar[d] \ar@{<-^{)}}[r] & \MM(\gamma_-,\gamma') \times \MM(\gamma',\gamma_+)
\ar@{<-_{)}}[ul]\ar[d]
\\
{} &(-\e_0,0]  \ar@{<-^{)}}[r] &\{0\}}.
$$
\end{enumerate}
\end{thm}

\bibliographystyle{amsalpha}

\bibliography{biblio}

\end{document}